\documentclass[amssymb,amsfonts,refcheck,12pt,verbatim,righttag]{amsart}

\usepackage{amssymb}

\usepackage{graphicx,color,tikz,caption,subcaption}

\numberwithin{equation}{section} 
\theoremstyle{plain}

\newtheorem{thm}{Theorem}[section]
\newtheorem{lem}[thm]{Lemma}
\newtheorem{pro}[thm]{Proposition}
\newtheorem{cor}[thm]{Corollary}

\newtheorem{de}[thm]{Definition}
\newtheorem{rem}[thm]{Remark}

\def\P{{\mathcal P}}
\def\R {{\Bbb R}}
\def\N {{\Bbb N}}
\def\Z {{\Bbb Z}}

\def\M {{\mathcal M}}
\def\P {{\mathcal P}}

\def\I {{\mathcal I}}
\def\A {{\mathcal A}}
\def\B {{\mathcal B}}

\def\E{{\bf E}}
\def\ba{{\bf a}}

\begin{document}
\baselineskip 13.7pt
\title[Dimension of invariant measures for affine IFS]{Dimension of invariant measures for affine iterated function systems}

\author{De-Jun Feng}
\address{
Department of Mathematics\\ The Chinese University of Hong Kong\\ Shatin,  Hong Kong\\ } \email{djfeng@math.cuhk.edu.hk}

\keywords{Iterated function systems, self-affine sets and measures, invariant measures, exact dimensionality, Hausdorff dimension, packing dimension}
\thanks {
2000 {\it Mathematics Subject Classification}: 28A80, 37C45}

\date{}

\begin{abstract}
  Let $\{S_i\}_{i\in \Lambda}$ be a finite contracting affine iterated function system (IFS) on $\R^d$. Let $(\Sigma,\sigma)$ denote the two-sided full shift over the alphabet $\Lambda$, and $\pi:\Sigma\to \R^d$ be the  coding map associated with the IFS.  We prove that the projection of an ergodic $\sigma$-invariant measure on $\Sigma$ under $\pi$ is always exact dimensional, and its Hausdorff dimension satisfies a Ledrappier-Young type formula. Furthermore, the result extends  to average contracting affine IFSs. This completes several previous results and answers a folklore open question in the community of fractals.  Some applications are given to the dimension of self-affine sets and measures.
\end{abstract}

\maketitle
\section{Introduction}\label{S-1}
\subsection{Motivation and the main result}

Let ${\rm Mat}_d(\R)$ denote the set of real $d\times d$ matrices.  By an {\it affine iterated function system} (affine IFS) on $\R^d$  we mean a finite  family ${\mathcal S}=\{S_j\}_{j\in \Lambda}$ of affine mappings from $\R^d$ to $\R^d$,
taking  the form
\begin{equation}
\label{e-form}
S_j(x)=M_jx+a_j,\qquad j\in \Lambda,
\end{equation}
where $M_j\in {\rm Mat}_d(\R)$ and $a_j\in \R^d$. Here, in contrast to the usual definition of affine IFS,   we  do not assume that $M_j$ are invertible or contracting (in the sense that $\|M_j\|<1$ where $\|\cdot\|$ is the matrix operator norm).   We say that   ${\mathcal S}$ is {\it   contracting}  if all $M_j$ are contracting. It is  well-known that if $\mathcal S$ is contracting, there  exists a unique non-empty compact set $K\subset \R^d$ such that
$$
K=\bigcup_{j\in \Lambda} S_j(K).
$$
We call $K$ the {\em self-affine set} generated by ${\mathcal S}$. In particular, if all the maps in ${\mathcal S}$ are contracting similitudes, we call $K$ a {\em self-similar set}.  As usual, a contracting ${\mathcal S}$ is said to satisfy the {\em open set condition} if there exists a non-empty open set $U\subset \R^d$ such that $S_j(U)$, $j\in \Lambda$, are disjoint subsets of $U$; moreover,  ${\mathcal S}$ is said to satisfy the {\em strong separation condition} if $S_j(K)$, $j\in \Lambda$, are disjoint.

Let $(\Sigma,\sigma)$ be the two-sided full shift over the alphabet $\Lambda$, i.e.~$\Sigma=\Lambda^\Z$ and $\sigma:\Sigma\to \Sigma$ is  the left shift map.  Endow $\Sigma$ with the  product topology and let $\M_\sigma(\Sigma)$ denote the space of $\sigma$-invariant Borel probability measures on $\Sigma$.

\begin{de}
\label{de-1.1}
Let $m\in \M_\sigma(\Sigma)$. An affine IFS ${\mathcal S}=\{M_jx+a_j\}_{j\in \Lambda}$  is said to be average contracting with respect to $m$ if,
for $m$-a.e.~$x=(x_n)_{n=-\infty}^\infty\in\Sigma$, the top Lyapunov exponent $\lambda(x)$ defined by
\begin{equation*}
\label{e-lyapunov}
\lambda(x)=\lim_{n\to \infty} \frac{1}{n}\log \|M_{x_{0}}\cdots M_{x_{n-1}}\|
\end{equation*}
is strictly negative.
\end{de}

We remark that  the above limit in defining $\lambda(x)$  exists and takes values in $[-\infty, \infty)$ for $m$-a.e.~$x$. This follows from the Furstenberg-Kesten theorem \cite{FurstenbergKesten1960} or  Kingman's sub-additive ergodic theorem \cite{Kingman1968}.

Now let $m\in \M_\sigma(\Sigma)$ and suppose that ${\mathcal S}$ is average contracting with respect to $m$. The canonical coding map $\pi: \Sigma\to \R^d$, given by
\begin{equation}
\label{e-pi1.4}
\begin{split}
\pi(x)&=\lim_{n\to \infty} S_{x_0}\circ S_{x_1}\circ \cdots\circ S_{x_{n}}(0)\\
&=\lim_{n\to \infty} \left(a_{x_0}+M_{x_0}a_{x_1}+\cdots +M_{x_{0}}\cdots M_{x_{n-1}}a_{x_n}\right),
\end{split}
\end{equation}
is well-defined on $\Sigma$ up to a set of zero $m$-measure (\cite{Brandt1986, BougerolPicard1992}); see Section \ref{S-3} for a self-contained proof.  The push-forward  $\pi_*m$ of $m$ by $\pi$, given by
$$
(\pi_*m)(F)=m(\pi^{-1}(F)) \quad \mbox{ for any Borel set }F\subset\R^d,
$$
is called an {\it invariant measure} or {\it stationary measure} for  ${\mathcal S}$.  When $m$ is ergodic,
$\pi_*m$ is called an {\it ergodic invariant measure} for $\mathcal S$.  Moreover if $m$ is a Bernoulli product measure, $\pi_*m$ is called a {\it self-affine measure} generated by ${\mathcal S}$; if in addition, ${\mathcal S}$ consists of similarities,  then  $\pi_*m$ is called a {\em self-similar measure}.

The main purpose of this paper is to study the dimension of invariant measures for affine IFSs.  Recall that for a probability measure $\eta$ on a metric space $X$,  the {\it local upper and lower dimensions} of $\eta$ at $x\in X$ are defined respectively by
$$\overline{\dim}_{\rm loc}(\eta, x)=\limsup_{r\to 0}\frac{\log \eta (B(x,r))}{\log r},\quad \underline{\dim}_{\rm loc}(\eta, x)=\liminf_{r\to 0}\frac{\log \eta (B(x,r))}{\log r},$$
 where $B(x,r)$ stands for  the closed ball of radius $r$ centered at $x$. If $$\overline{\dim}_{\rm loc}(\eta, x)=\underline{\dim}_{\rm loc} (\eta, x),$$ the common value is denoted as $\dim_{\rm loc}(\eta,x)$ and is called the {\it local dimension} of $\eta$ at $x$. We say that  $\eta$  is  {\it exact dimensional}
if there exists a constant $C$ such that the  local dimension
$\dim_{\rm loc}(\eta, x)$
exists and equals $C$ for $\eta$-a.e.~$x\in X$. It is well-known that if $\eta$ is an exact dimensional measure in ${\Bbb R}^d$, the Hausdorff and  packing
dimensions of $\eta$  coincide and are equal to the involved constant $C$, and so are some other notions of dimension (e.g.~ entropy dimension); see \cite{Young1982, Falconer1997}. Recall that the Hausdorff and packing dimensions of $\eta$ are defined by
\begin{align*}
{\dim}_{\rm H}\eta &=\inf \{\dim_{\rm H}F:\; \eta(F)>0\mbox{ and $F$ is a Borel set}\},\\
{\dim}_{\rm P}\eta &=\inf \{\dim_{\rm P}F:\; \eta(\R^d\setminus F)=0\mbox{ and $F$ is a Borel set}\},
\end{align*}
 where  $\dim_{\rm H}F, \dim_{\rm P}F$ stand for the Hausdorff and packing dimensions of $F$, respectively (cf.~\cite{Falconer2003}).

A folklore open problem in fractal geometry asks whether every ergodic invariant measure for an affine IFS is  exact dimensional. As the main result of this paper, we give the following affirmative answer.

\begin{thm}
\label{thm-1.0}
Let ${\mathcal S}=\{M_jx+a_j\}_{j\in \Lambda}$  be an affine IFS on $\R^d$ and   $m\in \M_\sigma(\Sigma)$. Suppose that  ${\mathcal S}$ is average contracting with respect to $m$. Let $\mu=\pi_*m$. Then \begin{itemize}
\item[(i)] $\dim_{\rm loc}(\mu, x)$ exists for $\mu$-a.e.~$x\in \R^d$.
\item[(ii)] Assume furthermore that $m$ is ergodic. Then $\mu$ is exact dimensional and $\dim_{\rm H}\mu$ satisfies a Ledrappier-Young type dimension formula.
\end{itemize}
\end{thm}

The precise dimension formula of $\mu$ and some of its applications will be given in Sections~\ref{S-DF}-\ref{S-SC}.

 We remark that Theorem \ref{thm-1.0} also holds in its one-sided version. To be more precise, let $(\Sigma^+,\sigma)$ denote the one-sided full shift over the alphabet $\Lambda$, i.e.,
$$
\Sigma^+=\{(x_n)_{n=0}^\infty:\; x_n\in \Lambda \mbox{ for all }n\geq 0\},$$
and $\sigma$ is the left shift map.  Let $\tau:\; \Sigma\to \Sigma^+$ be the natural projection defined by
$$
x=(x_n)_{n=-\infty}^\infty\mapsto x^+=(x_n)_{n=0}^\infty.
$$
It is well known that the push-forward map $\tau_*:\; {\mathcal M}_\sigma(\Sigma)\to {\mathcal M}_\sigma(\Sigma^+)$, $m\mapsto\tau_*m$, is bijective, and moreover, $\tau_*m$ is ergodic if and only if $m$ is ergodic  (see e.g. \cite[pp.~21-22]{Bowen1975}).  Let  $m^+\in {\mathcal M}_\sigma(\Sigma^+)$ and  assume that ${\mathcal S}$ is average contracting with respect to  $m=(\tau_*)^{-1}(m^+)$. Define $\pi^+:\;\Sigma^+\to \R^d$ by $(x_n)_{n=0}^\infty\mapsto \lim_{n\to \infty}
S_{x_0}\circ\cdots \circ S_{x_n}(0)$.   Then $\pi(x)=\pi^+(x^+)$ and so $\pi^+$ is well defined $m^+$-a.e. Moreover,  $(\pi^+)_*(m^+)=\pi_*m$. Hence the conclusions of Theorem \ref{thm-1.0} hold for $(\pi^+)_*(m^+)$.

Below we first give some background information about the above study.

The problem of the existence of local dimensions has a long history in smooth dynamical systems, as well as in the study of IFSs. It is of great importance  in dimension theory of dynamical systems and fractal geometry.
In \cite{Young1982}, Young proved
that an ergodic hyperbolic measure invariant under a $C^{2}$
surface diffeomorphism is always  exact dimensional. (Here by hyperbolic one means that the measure has no zero Lyapunov
exponent.) For a hyperbolic  measure
$\mu$ in higher-dimensional $C^{2}$ systems,
Ledrappier and Young \cite{LedrappierYoung1985} proved the existence of $\delta^u$ and
$\delta^s$, the local dimensions along stable and unstable local
manifolds, respectively, and the upper local dimension of $\mu$ is
bounded by the sum of $\delta^u$ and  $\delta^s$; moreover they obtained a formula for $\delta^u$ and  $\delta^s$ in terms of conditional entropies and Lyapunov exponents, which nowadays is called ``Ledrappier-Young  formula".     Eckmann and Ruelle \cite{EckmannRuelle1985} indicated that
it is unknown whether the local dimension of $\mu$ is equal to the sum of
$\delta^u$ and $\delta^s$ if $\mu$ is a hyperbolic measure. Then the
problem was referred as Eckmann-Ruelle conjecture, and was finally
answered affirmatively by Barreira, Pesin and Schmeling  in 1999 for $C^{1+\alpha}$ diffeomorphisms \cite{BarreiraPesinSchmeling1999}. Later, the result of exact dimensionality was further extended by Qian and Xie \cite{QianXie2008} and Shu \cite{Shu2010}  to $C^2$ expanding endomorphisms  and $C^2$ non-degenerate endomorphisms, respectively.

For the study of IFSs, it is well-known that if ${\mathcal S}$ is a contractive IFS consisting of similarity maps, or more generally, a contracting $C^1$ conformal IFS, then under  an additional assumption of the open set condition, the push-forward measure of any ergodic invariant measure  by the coding map is exact dimensional with  dimension given by the classical entropy divided by the Lyapunov exponent (cf.~\cite{Bedford1991,Hutchinson1981,Patzschke1997}). The result essentially follows from  the  Shannon-McMillan-Breiman theorem in entropy theory. However, the problem becomes much more complicated without assuming the open set condition. In \cite{FengHu2009}, by introducing a notion of projection entropy and adopting some  ideas from \cite{LedrappierYoung1985}, Feng and Hu   proved that for any contracting $C^1$ conformal IFS, the push-forward measure of every ergodic invariant measure under the coding map is  exact dimensional, with  dimension given by the projection entropy divided by the Lyapunov exponent. Later this result was further extended to some random self-similar measures \cite{FalconerJin2014, SagliettiShmerkinSolomyak2018} and push-forward measures of ergodic invariant measures for some random conformal IFSs \cite{MihailescuUrbanski2016}. It is worth pointing out that the exact dimensionality of overlapping self-similar measures was first claimed by Ledrappier;  nevertheless no proof has been written out (cf.~\cite[p.~1619]{PeresSolomyak2000}). This property was also conjectured  later  by Fan, Lau and Rao in \cite{FanLauRao2002}.

The first result for affine IFSs is due to Bedford \cite{Bedford1984} and McMullen \cite{McMullen1984}, who independently calculated the Hausdorff and box-counting dimensions of a special class of planar self-affine sets (which are now called Bedford-McMullen carpets) and showed that they are usually different.  McMullen \cite{McMullen1984} also implicitly  proved the exact dimensionality of self-affine measures on the Bedford-McMullen carpets, and calculated the precise value of the dimension.   Later, Gatzouras and Lalley \cite{GatzourasLalley1992} and Bara\'{n}ski \cite{Baranski2007} obtained  similar results for a class  of  more general carpet-like self-affine sets in the plane. In \cite{KenyonPeres1996}, Kenyon and Peres extended Bedford and McMullen's result to higher dimensional self-affine carpets, and moreover,  they proved the exact dimensionality and gave a Ledrappier-Young type dimension formula for  arbitrary ergodic invariant measure on these carpets.  For more related results on carpet-like self-affine sets, see the survey paper \cite{Falconer2013}.

 In \cite{FengHu2009},  Feng and Hu proved that for  each contracting invertible affine IFS in $\R^d$,    Theorem \ref{thm-1.0} holds under an additional assumption that the linear parts of the IFS  commute (i.e. $M_iM_j=M_jM_i$). It remained open whether this additional assumption could be removed.
Very recently,   B\'{a}r\'{a}ny and K\"{a}enm\"{a}ki \cite{BaranyKaenmaki2015} made a substantial  progress.  They proved that for contracting invertible affine IFSs,  every planar self-affine measure (more generally, every self-affine measure in $\R^d$ having  $d$ distinct Lyapunov exponents) is exact dimensional, and moreover, under certain domination condition on the linear parts $\{M_j\}$,  the push-down of every quasi-Bernoulli measure on the self-affine set is exact dimensional, with dimension given by a Ledrappier-Young type formula.  Some other partial results were also obtained in \cite{Barany2015, Rapaport2015, FraserJordanJurga2017}. Along another direction, it is proved  that for a given ergodic $m\in {\mathcal M}_{\sigma}(\Sigma)$,  $\pi_*m$ is exact dimensional for ``almost all'' contracting invertible affine IFSs satisfying $\|M_j\|<1/2$ (\cite{JordanPollicottSimon2007, Jordan2011, Rossi2014});  however, the result does not apply to any concrete case.

Theorem \ref{thm-1.0} finally gives a full affirmative answer to   the problem of  the existence of local dimensions in the context of  affine IFSs. It completes  the aforementioned previous works on the problem.

Exact dimensionality and Ledrappier-Young type dimension formula play a significant role in many of the recent advances in dimension theory of self-affine sets and measures (see e.g.~\cite{Barany2015,  BaranyHochmanRapaport2017, BaranyKaenmaki2015, BKK2018, BaranyRams2018, BaranyRamsSimon2016, BRS2018, BarralFeng2013, DasSimmons2017, FalconerKempton2018, HochmanSolomyak2017, MorrisShmerkin2018, PrzytyckiUrbanski1989,Rapaport2015}).  In the remaining part of this section, we will present some applications of Theorem \ref{thm-1.0} along the lines of these developments.

The proof of  Theorem \ref{thm-1.0} is based on some ideas from the work of Ledrappier and Young \cite{LedrappierYoung1985}. It  also adopts and extends some ideas used in \cite{FengHu2009, BaranyKaenmaki2015, QianXie2008} for the construction of measurable partitions and the density estimates of associated conditional measures. Since our construction of measurable partitions  is much different from these works (see Remark \ref{rem-tt}), and the IFSs in consideration may be non-invertible and non-contractive,  many estimates of conditional measures need to be rebuilt or re-justified.  A key part of our arguments is on the  estimation of  the so-called ``transverse dimension'' of these conditional measures, where significant efforts are made to handle   the  situation when the linear parts of the IFS do not satisfy any domination condition (in such case the angles of Oseledets subspaces  may be arbitrarily close to zero). Our strategy is to build an induced dynamics  so that we are able to focus on the trajectories where  the angles of Oseledets subspaces are  larger than a positive constant.

\subsection{Dimension formulas}
\label{S-DF}
Throughout this subsection, under the assumptions of Theorem \ref{thm-1.0},  we further assume that $m$ is ergodic.  We are going to present certain  dimension formulas for  $\mu=\pi_*m$ and related conditional measures.

First notice that in this ergodic case, the condition (3) in Definition~\ref{de-1.1} is equivalent to
\begin{equation}
\label{e-1.3}
\lambda:=\lim_{n\to \infty}\frac{1}{n}\int \log \|M_{x_0}\cdots M_{x_{n-1}}\|\; d m(x)<0.
\end{equation}

By Oseledets' multiplicative ergodic theorem \cite{Oseledets1968}, there exist an integer $1\leq s\leq d$, numbers $\lambda=\lambda_1>\cdots>\lambda_s\geq -\infty$, positive integers $k_1, \ldots, k_s$ with $\sum_{i=1}^sk_i=d$, and   measurable linear subspaces
$$
\R^d=V_x^0\supsetneq V_x^1\supsetneq\cdots\supsetneq  V_x^s=\{0\}, \quad x\in \Sigma,
$$
such that for $m$-a.e.~$x=(x_n)_{n=-\infty}^\infty$,
\begin{itemize}
\item[(i)] $M_{x_{-1}}V_x^i\subset V^i_{\sigma^{-1}x}$;
\item[(ii)] $\dim V_x^i=\sum_{j=i+1}^s k_j$;
\item[(iii)] $\lim_{n\to \infty} \frac{1}{n}\log \|M_{x_{-n}}\cdots M_{x_{-1}}v\|=\lambda_{i+1}$ for $v\in V_{x}^i\backslash V_{x}^{i+1}$.
\end{itemize}

 When the matrices $M_j$ ($j\in \Lambda$) are assumed to be invertible, then (i) becomes an equality. It in general is a containment because $M_j$ may be singular. The numbers $\lambda_1,\ldots, \lambda_s$ are called the {\em Lyapunov exponents} of $(M_j)_{j\in \Lambda}$ with respect to $m$, and $k_i$ the {\em multiplicity} of $\lambda_i$, $i=1,\ldots,s$.  Recall that $\pi(x)$ is well-defined for $m$-a.e.~$x$. Hence there exists a Borel set $\Sigma'\subset\Sigma$ with  $\sigma(\Sigma')=\Sigma'$ and $m(\Sigma')=1$ such that $\pi$ is well-defined on $\Sigma'$ and the above properties (i)-(iii) hold for $x\in \Sigma'$.

We remark that these linear subspaces $V_x^i$ only depend on $i$ and $x^{-}:=(x_j)_{j=-\infty}^{-1}$ since by (i)-(iii), one has
 $$
 V_{x}^i=\left\{v\in \R^d: \; \lim_{n\to \infty} \frac{1}{n}\log \|M_{x_{-n}}\cdots M_{x_{-1}}v\|\leq \lambda_{i+1}\right\}.
 $$
 Using this property, we construct a family of measurable partitions $\xi_0,\ldots, \xi_s$ of $\Sigma'$ as follows:
$$
\xi_i(x):=\{y\in \Sigma':\; y^{-}=x^{-}, \; \pi y-\pi x\in V_x^i\},
$$
here $\xi_i(x)$ is the $\xi_i$-atom that contains $x$ (see Sections~\ref{Sect-con} and \ref{S-4} for the details).    Moreover, let
\begin{equation}
\label{e-pp}
\P=\{[j]\cap \Sigma':\; j\in \Lambda\}
\end{equation}
be the
canonical partition of $\Sigma'$, where $[j]:=\{x=(x_n)_{n=-\infty}^\infty\in \Sigma:\; x_0=j\}$.
  Define
\begin{equation}
\label{e-hi1}
h_i=H_m(\P|\widehat{\xi_i}),\quad i=0,\ldots, s,
\end{equation}
where  $H_m(\cdot|\cdot)$ stands for  the conditional entropy and $\widehat{\xi_i}$ is the $\sigma$-algebra generated by $\xi_i$ (see Sections~\ref{Sect-ent}-\ref{Sect-con} for the definitions).

 We remark that the spaces $V_x^i$ are strictly decreasing, $\R^d=V_x^0\supsetneq V_x^1\supsetneq\cdots\supsetneq  V_x^s=\{0\}$. Therefore, the partitions $\xi_i$ become finer as $i$ increases: the partition $\xi_0$ is the partition according to the ``past'' $x^{-}$, then $\xi_1$ is the partition according to the past joined with the partition according to translations of $V_x^i$, etc. Therefore $h_i$ decrease with $i$, since $h_{i+1}$ is conditioned on partition $\xi_{i+1}$ which is finer than the partition $\xi_i$ on which $h_i$ is conditioned.

Now we are ready to present the dimension formula for $\pi_*m$.

\begin{thm}
\label{thm-1.1} Let ${\mathcal S}=\{M_jx+a_j\}_{j\in \Lambda}$  be an affine IFS on $\R^d$ and   $m$ be an ergodic $\sigma$-invariant measure on $\Sigma$. Suppose that  ${\mathcal S}$ is average contracting with respect to $m$. Let $\mu=\pi_*m$.
Then
\begin{equation}
\label{ly-dim}
\dim_{\rm H}\mu=\sum_{i=0}^{s-1}\frac{h_{i+1}-h_i}{\lambda_{i+1}},
\end{equation}
where $h_i$ are defined as in \eqref{e-hi1}.
\end{thm}

 We remark that both the nominators and denominators in \eqref{ly-dim} are non-positive.  Next we give similar dimension formulas for certain conditional measures associated with $m$. For $i=0,\ldots, s$, let  $\{m_x^{\xi_i}\}$ be the system of conditional measures of $m$ associated with the partition $\xi_i$ (cf.~Section~\ref{Sect-con}). For a linear subspace $W$ of $\R^d$, let $W^\perp$ denote the orthogonal complement of $W$ in $\R^d$, and let $P_W:\; \R^d\to W$  denote the orthogonal projection from $\R^d$ to $W$.
\begin{thm}
\label{thm-1.2}
Under the assumptions of Theorem \ref{thm-1.1},
 for any $0\leq i<j\leq s$ and  $m$-a.e.~$x\in \Sigma'$, the push-forward measures $\pi_*(m_x^{\xi_i})$,
$\big(P_{{(V_x^j)}^\perp}\pi\big)_*(m_x^{\xi_i})$ of $m_x^{\xi_i}$
 are exact dimensional with
\begin{align}
&\dim_{\rm H}\left( \pi_*(m_x^{\xi_i})\right)=\sum_{\ell=i}^{s-1}\frac{h_{\ell+1}-h_\ell}{\lambda_{\ell+1}}, \label{e-con-dim}\\
&\dim_{\rm H}\left(\big(P_{{(V_x^j)}^\perp}\pi\big)_*(m_x^{\xi_i})\right)=\sum_{\ell=i}^{j-1}\frac{h_{\ell+1}-h_{\ell}}{\lambda_{\ell+1}}, \label{e-con-proj-dim}
\end{align}
 Moreover, for $m$-a.e.~$x\in \Sigma'$  and any $1\leq j\leq s$,
\begin{equation}\label{e-proj-dim}
\dim_{\rm loc}\left( \big(P_{{(V_x^j)}^\perp}\pi\big)_*m,\; P_{{(V_x^j)}^\perp}( \pi x)\right)=\sum_{\ell=0}^{j-1}\frac{h_{\ell+1}-h_{\ell}}{\lambda_{\ell+1}}.
\end{equation}
\end{thm}

From the above theorem, we  can  deduce certain dimension conservation property for the measures
 $\pi_*\big(m^{\xi_0}_x\big)$ and $\mu$. To state the result, let $G(d,k)$ denote the Grassmannian manifold of  $k$-dimensional linear subspaces of $\R^d$.  For a Borel probability measure $\eta$ on $\R^d$ and  $W\in G(d,k)$, let $\{\eta_{W,z}=\eta^{\zeta_W}_z\}_{z\in \R^d}$ denote the
the system of conditional measures of $\eta$ associated with the measurable partition $\zeta_W$ given by
$$
\zeta_W=\{W+a:\; a\in W^\perp\}.
$$
These conditional measures are also called the {\it slicing measures} of $\eta$ along the subspace $W$ (cf. \cite[\S10.1]{Mattila1995}).  Following Furstenberg \cite{Furstenberg2008}, we give the following.
\begin{de}
A measure $\eta$ is said to be dimension conserving  with respect to the projection $P_{W^\perp}$, if
$$
\dim_{\rm H}\eta =\dim_{\rm H} \eta_{W, z}+\dim_{\rm H}\left(  (P_{W^\perp})_*\eta\right)
$$
  for $\eta$-a.e.~$z\in \R^d$.
\end{de}

 For $i\in \{0, \ldots, s-1\}$, define $\Pi_i:\Sigma'\to G(d, \sum_{j=i+1}^sk_j)$ by \begin{equation}
\label{e-Pi}
\Pi_i(x)= V_x^i.
\end{equation}
The push-forward measures $(\Pi_i)_*m$, $i=1,\ldots, s-1$, are called the {\em Furstenberg measures} or {\em Furstenberg-Oseledets measures} associated with $(M_j)_{j\in \Lambda}$ and $m$.
  An ergodic measure $\nu\in \M_\sigma(\Sigma)$ is said to be {\it quasi-Bernoulli} if there exists a positive constant $C$ such that
$$C^{-1}\nu([I])\nu([J])\leq \nu([IJ])\leq C\nu([I]) \nu([J])$$
for any finite words $I$, $J$ over $\Lambda$, where
$$[I]:=\{x\in \Sigma:\; x_j=i_j \mbox{ for }0\leq  j\leq n-1\}$$ for $I=i_0\ldots i_{n-1}$.
 Similarly, we say that $\nu$ is {\it sub-multiplicative} if there exists a positive constant $C$ such that  $\nu([IJ])\leq C\nu([I]) \nu([J])$
for any finite words $I$, $J$ over $\Lambda$.
\begin{thm}
\label{cor-1.0}
Under the assumptions of Theorem \ref{thm-1.1},
we further assume that $s\geq 2$.
 Let $i\in \{1,\ldots, s-1\}$. Then  the following statements hold.
 \begin{itemize}
 \item[(i)] For $m$-a.e.~$x\in \Sigma'$, $\pi_*\big(m_x^{\xi_0}\big)$ is dimension conserving with respect to  $P_{{(V_x^i)}^\perp}$ and moreover, the projected measure
 $\big(P_{{(V_x^i)}^\perp}\pi\big)_*\big(m_x^{\xi_0}\big)$ is exact dimensional, and so are the slicing measures  $\big(\pi_*\big(m_x^{\xi_0}\big)\big)_{V_x^i, y}$ for $\pi_*\big(m_x^{\xi_0}\big)$-a.e.~$y$.
 \item[(ii)] Assume that $m$ is quasi-Bernoulli. Then for $(\Pi_i)_*m$-a.e.~$W$, $\mu$ is
            dimension conserving with respect to $P_{W^\perp}$, and moreover, the associated projected measure and almost all slicing measures are exact dimensional.
 \item[(iii)] Assume that $m$ is sub-multiplicative. Then for $(\Pi_i)_*m$-a.e.~$W$, there exists a subset  $A_W$ of~$\R^d$ with  $\mu(A_W)>0$ such that for every $z\in A_W$,
  \begin{equation*}
  \begin{split}
  \dim_{\rm loc} (\mu_{W,z}, z)&=\sum_{\ell=i}^{s-1}\frac{h_{\ell+1}-h_\ell}{\lambda_{\ell+1}},\\
  \dim_{\rm loc} \left(\big(P_{W^\perp}\big)_*\mu, P_{W^\perp}(z)\right)&=\sum_{\ell=0}^{i-1}\frac{h_{\ell+1}-h_\ell}{\lambda_{\ell+1}},
  \end{split}
  \end{equation*}
 and so,
 $\dim_{\rm H}\mu=\dim_{\rm loc} (\mu_{W,z}, z)+\dim_{\rm loc} \left( \big(P_{W^\perp}\big)_*\mu, P_{W^\perp}(z)\right).$  When $m$ is quasi-Bernoulli then one can take the set $A_W$ such that $\mu(A_W)=1$.
 \end{itemize}
 \end{thm}

We remark that part (ii) of Theorem \ref{cor-1.0} was previously proved in \cite{BaranyKaenmaki2015} under the  assumptions that ${\mathcal S}$ is contracting,\ invertible and its linear parts satisfy certain domination condition.
  According to part (iii) of the  theorem, when $m$ is sub-multiplicative,  $\mu$ partially satisfies dimension conservation.  It is unknown  whether part (ii) always holds when $m$ is only assumed to be ergodic. However, as is illustrated  in the following theorem,  this is true in the special case that  the linear parts of the IFS commute.
\begin{thm}
\label{thm-1.7'}
Let ${\mathcal S}=\{M_jx+a_j\}_{j\in \Lambda}$  be an affine IFS on $\R^d$,  average contracting with respect to an ergodic $m\in {\mathcal M}_\sigma(\Sigma)$.  Let $\mu=\pi_*m$. Assume $s\geq 2$ and  in addition that $M_jM_{j'}=M_{j'}M_j$ for $j,j'\in \Lambda$. Then for $i\in \{1,\ldots, s-1\}$, $V_x^i$ is constant $m$-a.e., denoted by $W_i$, moreover, $\mu$ is dimension conserving with respect to $P_{(W_i)^\perp}$.
\end{thm}

  It is worth pointing out that if $\mu$ is a  contracting self-similar measure in $\R^d$ with a finite rotation group, then for each proper subspace $W$ of $\R^d$, $\mu$ is dimension conserving with respect to $P_W$. The result  is due to Falconer and Jin \cite{FalconerJin2014}. Under an additional assumption of the strong separation condition, this result can  be alternatively derived from a general result of Furstenberg (cf.~\cite[Theorem 3.1]{Furstenberg2008}). We remark that this dimension conservation property also extends to ergodic invariant measures for rotation-free self-similar IFSs (see Remark \ref{rem-6.3}).  However, as was proved by Rapaport \cite{Rapaport2017}, this dimension conservation actually can fail for some self-similar measures with infinite rotation groups.  Finally we remark that Theorems \ref{cor-1.0}-\ref{thm-1.7'} can be applied to analyze  slices and projections of certain self-affine sets (see Remark \ref{rem-6.4}).

\subsection{Semi-continuity of dimension and applications}
\label{S-SC}
Here we present a semi-continuity result on the dimension of ergodic invariant measures for affine IFSs and give its application to the dimension of self-affine sets.
Again let $\mathcal S=\{M_ix+a_i\}_{i\in \Lambda}$ be an affine IFS on $\R^d$, average contracting with respect to an ergodic invariant measure $m$ on $\Sigma$.  Write
$\ba=(a_i)_{i\in \Lambda}$. To emphasize the dependence on $\ba$, let $\pi_\ba$ be the coding map associated to $\mathcal S$ and let $h_{i,\ba}$ ($i=1,\ldots, s$) be the conditional entropies of $m$ defined in \eqref{e-hi1}.   Then we have the following.

\begin{thm}
\label{thm-1.3}
\begin{itemize}
\item[(1)]
The mapping $\ba\mapsto h_{i,\ba}$ is upper semi-continuous for each $i\in \{1,\ldots, s\}$.
\item[(2)] Moreover,
the mapping $\ba\mapsto \dim_{\rm H} \left( (\pi_{\ba})_*m\right)$ is lower semi-continuous.
\end{itemize}
\end{thm}

 Part (1) of the above result was first proved by Rapaport \cite[Lemma 8]{Rapaport2015} in the  case when $m$ is a Bernoulli product measure and ${\mathcal S}$ is invertible and contracting.  Part (2) was shown by Hochman and Shmerkin \cite{HochmanShmerkin2012} for a special class of self-similar measures on $\R$. In Remark \ref{rem-8.1} we give a further extension of Theorem \ref{thm-1.3}.

Below we present an application of Theorem \ref{thm-1.3} to the dimension of self-affine sets and associated stationary measures.
For this purpose, in the remaining part of this subsection   we assume that $\|M_j\|<1$ for $j\in \Lambda$ and write ${\bf M}=(M_j)_{j\in \Lambda}$.
Let $K({\bf M},\ba)$ be the self-affine set generated by the IFS ${\mathcal S}=\{M_jx+a_j\}_{j\in \Lambda}$. In 1988,  Falconer \cite{Falconer1988} introduced a quantity associated to ${\bf M}$, nowadays usually called the {\em affinity dimension} $\dim_{\rm AFF}({\bf M})$, which is always an upper bound for the upper box-counting dimension of $K({\bf M},\ba)$, and such that when $\|M_j\|<1/2$ for all $j$, then for ${\mathcal L}^{d|\Lambda|}$-a.e.~$\ba$,   $\dim_{\rm H}K({\bf M},\ba)=\min (d, \dim_{\rm AFF}({\bf M}))$. In fact, Falconer proved this with $1/3$ as the upper bound on the norms; it was subsequently shown by Solomyak \cite{Solomyak1998} that $1/2$ suffices.

The analogue of affinity dimension for measures is the Lyapunov dimension, which we denote $\dim_{\rm LY}(m, {\bf M})$; see Section~\ref{S-7} for its definition.  In \cite{JordanPollicottSimon2007}, Jordan, Pollicott and Simon proved that the Lyapunov dimension  $\dim_{\rm LY}(m, {\bf M})$ is always an upper bound for the Hausdorff dimension of $ (\pi_{\ba})_*m$, and moreover when $\|M_j\|<1/2$ for all $j$,  then for ${\mathcal L}^{d|\Lambda|}$-a.e.~$\ba$,
$\dim_{\rm H} ((\pi_{\ba})_*m)=\min(d, \dim_{\rm LY}(m,{\bf M}))$.

Recall a set in a topological space is said to be of {\em first category} if  it can be written as the countable union of nowhere dense subsets.   As an application of Theorem \ref{thm-1.3}, we get the following result.
\begin{thm}
\label{cor-1.7}
Suppose that $\|M_j\|<1/2$ for $j\in \Lambda$. Then the following hold.
\begin{itemize}
\item[(i)] For every ergodic $\sigma$-invariant measure $m$ on $\Sigma$,  the exceptional set
$$
\left\{\ba\in \R^{d|\Lambda|}:\; \dim_{\rm H} \left((\pi_{\ba})_*m\right)\neq \min(d, \dim_{\rm LY}(m, {\bf M}))\right\}
$$
is of first category  in $\R^{d|\Lambda|}$.
\item[(ii)]
The exceptional set
$$
\left\{\ba\in \R^{d|\Lambda|}:\; \dim_{\rm H} K({\bf M}, \ba)\neq \min(d, \dim_{\rm AFF}({\bf M}))\right\}
$$
is of first category  in $\R^{d|\Lambda|}$.

\end{itemize}
\end{thm}

The above result says that these exceptional sets are also small in a topological sense. A fundamental and challenging question is to specify those translation vectors not lying in the exception sets. Significant progresses have been made recently in \cite{Barany2015,FalconerKempton2018,BaranyKaenmaki2015, Rapaport2015}, showing that under certain additional assumptions, the Hausdorff  and Lyapunov dimensions of a self-affine measure  (or more generally, the push-forward of a quasi-Bernoulli measure)    coincide if the involved Furstenberg measures have enough large dimension. In next theorem we will drop off some redundant assumptions used in these works and  further extend the result to the push-forward measures of ergodic sub-multiplicative measures.

Recall that for a Borel probability measure $\eta$ on a metric space, its {\it upper Hausdorff dimension}  $\dim_{\rm H}^*\eta$ is  the smallest Hausdorff dimension of a Borel set $F$  of $\eta$ measure $1$. Set $d_0=0$ and $d_\ell=k_1+\cdots+k_\ell$ for $1\leq \ell\leq s$.

\begin{thm}
\label{thm-1.4}
Let $\mathcal S=\{M_jx+a_j\}_{j\in \Lambda}$ be a contracting affine IFS on $\R^d$ satisfying the strong separation condition and $m\in \mathcal M_\sigma(\Sigma)$ be ergodic and sub-multiplicative.   Let $i$ be the unique element in $\{1,\ldots, s\}$ so that $d_{i-1}  \leq \dim_{\rm LY} (m, {\bf M})<d_i$. Then
\begin{equation} \label{e-tz}
\dim_{\rm H} (\pi_*m)=\dim_{\rm LY}(m, {\bf M})
\end{equation}
provided one of the following conditions holds:
\begin{itemize}
\item[(a)] $s=1$.
\item[(b)] $i=s>1$, $\lambda_s\neq -\infty$ and \begin{equation}
\label{e-condLY}
 \dim_{\rm H}^*\left(  (\Pi_{s-1})_*m\right) + \dim_{\rm LY} (m, {\bf M})\geq  d_{s-1} (d-d_{s-1}+1).
\end{equation}
\item[(c)] $1\leq i \leq s-1$, and
    \begin{align}
        &\dim_{\rm H}^*\left(  (\Pi_{i})_*m\right)-\dim_{\rm LY}(m, {\bf M})\geq d_{i}(d-d_i-1), \label{e-al2}\\
    &\dim_{\rm H}^*\left( (\Pi_{i-1})_*m \right) + \dim_{\rm H}  (\pi_*m)\geq d_{i-1}(d-d_{i-1}+1). \label{e-al1}
    \end{align}
   \end{itemize}
\end{thm}

 The conditions (b), (c) in the above theorem were introduced in  \cite{Rapaport2015} and   \cite{BaranyKaenmaki2015}, respectively, in slightly stronger forms.
 For a contracting invertible affine IFS,  Rapaport  \cite{Rapaport2015}  proved the implication (b)$\Rightarrow$ \eqref{e-tz} in the case when $m$ is Bernoulli and $(M_j)_{j\in \Lambda}$ satisfies an irreducibility assumption; whilst B\'{a}r\'{a}ny and K\"aenm\"aki \cite{BaranyKaenmaki2015} proved  \eqref{e-tz} under the assumptions that the conditions \eqref{e-al2}-\eqref{e-al1} hold for all $i\in \{1,\ldots, s-1\}$, $m$ is Bernoulli and $d=2$, or $m$ is quasi-Bernoulli and $\{M_j\}_{j\in \Lambda}$ satisfies a domination condition.

   We remark that \eqref{e-al1} always holds whenever  $i=1$,  since $d_0=0$. It is worth pointing out that for every affine IFS  $\mathcal S=\{M_jx+a_j\}_{j\in \Lambda}$ on $\R^d$, there exists at least one ergodic $m\in \mathcal M_\sigma(\Sigma)$, called {\em K\"aenm\"aki measure}, so that $\dim_{\rm LY}(m, {\bf M})=\dim_{\rm AFF}({\bf M})$. This was first proved by K\"aenm\"aki \cite{Kaenmaki2004} in the case when $\mathcal S$ is invertible, and
it extends to the general case by the sub-additive thermodynamic formalism \cite{CaoFengHuang2008}. Very recently, Bochi and Morris \cite{BochiMorris2018} showed that whenever $\mathcal S$ is invertible, each K\"aenm\"aki measure is sub-multiplicative. Hence for an invertible ${\mathcal S}$ satisfying the strong separation condition,  if one of the conditions (a)-(c) in Theorem \ref{thm-1.4} fulfills for some K\"aenm\"aki measure $m$, then $\dim_{\rm H} K({\bf M}, {\bf a})=\dim_{\rm AFF}({\bf M})=\dim_{\rm  H}  (\pi_*m)$.

To check the conditions (b)-(c) in Theorem \ref{thm-1.4}, one needs to estimate the (upper) Hausdorff dimension of Furstenberg measures  $(\Pi_{i})_*m$. So far there have been only a few dimensional results on these measures. In the case $d=2$, Hochman and Solomyak  \cite{HochmanSolomyak2017} calculated the Hausdorff dimension of Furstenberg measures for Bernoulli  $m$ under some mild assumptions. In \cite[Sect.~2.4]{BRS2018}, B\'ar\'any, Rams and Simon determined the Hausdorff dimension of Furstenberg measures for some special triangular affine IFSs in $\R^d$, in which $m$ could be any ergodic measure.

We remark that the conditions of Theorem \ref{thm-1.4} might not be sharp.  Very recently,  B\'ar\'any, Hochman and Rapaport \cite{BaranyHochmanRapaport2017} made a significant progress in dimension theory of affine IFSs,  showing that the Hausdorff and affinity dimensions of a planar self-affine set coincide under the strong separation  condition and certain irreducibility assumption; and similarly, the Hausdorff and Lyapunov dimensions of a planar self-affine measure coincide under the same assumptions.  In  \cite{HochmanRapaport2019}, Hochman and Rapaport further showed that the strong separation condition can be replaced by the exponential separation condition, which is substantially weaker.

\subsection{Organization of the paper} The paper is organized as follows. In Section~\ref{S-2}, we provide some density results about conditional measures, and present a version of Oseledets's multiplicative ergodic theorem due to Froyland et al.~\cite{FroylandLloydQuas2010}. In Section~\ref{S-3}, we give some auxiliary results on the coding maps for average contracting affine IFSs. In Section~\ref{S-4}, we construct a finite family of measurable partitions of $\Sigma$ for a given average contracting affine IFS and give some necessary properties.  In Section~\ref{S-5} we prove an inequality for the transverse dimensions of the conditional measures associated with these measurable partitions.  In Section~\ref{S-6}, we prove Theorems \ref{thm-1.0}-\ref{thm-1.2}, \ref{cor-1.0}-\ref{thm-1.7'}.  In Section~\ref{S-7}, we give some properties of Lyapunov dimension. In Section~\ref{S-8}, we prove Theorems \ref{thm-1.3}-\ref{thm-1.4}.
\section{Preliminaries}
\label{S-2}

\subsection{Conditional information and entropy}
\label{Sect-ent}
Let $(X,\B,m)$ be a probability space.
For a sub-$\sigma$-algebra $\A$ of $\B$ and $f\in L^1(X,\B,m)$,
we denote by ${\bf E}_m(f|\A)$ the the {\it conditional expectation of
$f$ given $\A$}.
For a countable $\B$-measurable partition $\xi$ of $X$,
we denote by ${\bf  I}_m(\xi|\A)$ the {\it conditional
information of $\xi$ given $\A$}, which is given by the formula
\begin{equation}
\label{e-1.2}
 {\bf I}_m(\xi|\A)=-\sum_{A\in \xi}\chi_A\log \E_m(\chi_A |\A),
\end{equation}
where $\chi_A$ is the characteristic function on $A$.
The {\it conditional entropy of $\xi$ given $\A$}, written as
$H_m(\xi|\A)$, is defined by the formula
\begin{equation*}
 H_m(\xi|\A)=\int {\bf I}_m(\xi|\A)\; dm.
\end{equation*}
(See e.g. \cite{Parry1981, Walters1982} for more details.)
The above information and entropy are unconditional when
$\A={{\mathcal N}}$, the trivial $\sigma$-algebra consisting of sets
of measure zero and one, and in this case we write
\begin{equation*}
 {\bf I}_m(\xi|{{\mathcal
N}})=:{\bf I}_\nu(\xi)\quad\mbox{and}\quad H_m(\xi|{{\mathcal
N}})=:H_m(\xi).
\end{equation*}

For a countable partition $\xi$, we use $\widehat{\xi}$ to denote the $\sigma$-algebra generated by $\xi$.  If $\xi$ is  an uncountable measurable partition of $X$ (which will be defined in Section~\ref{Sect-con}), $\widehat{\xi}$ is defined as the sub-$\sigma$-algebra of ${\mathcal B}$ whose sets are $\xi$-saturated (i.e. unions of elements in $\xi$).    If $\xi_1$,\ldots, $\xi_n$ are countable partitions, then $\xi_1\vee \cdots \vee \xi_n=\bigvee_{i=1}^n \xi_i$ denotes the partition consists of the sets $A_1\cap\cdots \cap A_n$ with $A_i\in \xi_i$. Similarly for $\sigma$-algebras $\A_1$, $\A_2$,\dots,  $\A_1\vee \A_2\vee \cdots$ or $\bigvee_{i} \A_i$ denotes the $\sigma$-algebra generated by $\bigcup_i \A_i$.

 In the following lemma, we list some basic  properties of the (conditional)  expectation, information and entropy.   The reader is referred to \cite[pages 20-21 and 38]{Parry1981} for details.

\begin{lem}
\label{lem-par} Let $T$ be a measure-preserving transformation of a separable probability space  $(X,\B,m)$. Let $\xi,\eta$ be two countable Borel partitions of $X$ with $H_m(\xi)<\infty$, $H_m(\eta)<\infty$, and $\A$  a sub-$\sigma$-algebra of $\B$. Then we have
 \begin{itemize}
 \item[(i)] ${\bf E}_m(f|\A)\circ T={\bf E}_m(f\circ T|T^{-1}\A)$ for $f\in L^1(X, \B, m)$.
\item[(ii)] ${\bf I}_{m}(\xi|\A)\circ T={\bf I}_m(T^{-1}\xi|T^{-1}\A)$.
\item[(iii)] ${\bf I}_m(\xi\vee \eta|\A)={\bf I}_m(\xi|\A)+{\bf I}_m(\eta|\widehat{\xi}\vee \A)$.
\item[(iv)] $H_m(\xi\vee \eta|\A)=H_m(\xi|\A)+H_m(\eta|\widehat{\xi}\vee \A)$.
\item[(v)] If $\A_1\subset \A_2\subset\cdots $ is an increasing sequence of sub-$\sigma$-algebras with $\A_n\uparrow \A$, then $\sup_n{\bf I}_m(\xi|\A_n)\in L^1$, and ${\bf I}_m(\xi|\A_n)$ converges almost everywhere and in $L^1$ to ${\bf I}_m(\xi|\A)$. In particular, $\lim_{n\to \infty} H_m(\xi|\A_n)=H_m(\xi|\A)$.
\end{itemize}
\end{lem}

 For the convenience of the reader,  below we state an  almost trivial property of the conditional expectation.  For a proof, see e.g.  \cite[Lemma 3.10]{FengHu2009}.
\begin{lem}
\label{lem-2.10}
Let  $(X,\B,m)$ be a  probability space and  $\A$  a sub-$\sigma$-algebra of $\B$. Let $A\in \B$ with $m(A)>0$. Then
$$
\E_m(\chi_A|\A)(x)>0
$$
for $m$-a.e.~$x\in A$.
\end{lem}

The following lemma is a variant of Maker's ergodic theorem (\cite{Maker1940}).
\begin{lem}[\cite{Mane1987}, Corollary 1.6, p. 96]
\label{lem-3.15}
Let $T$ be a measure-preserving transformation of a probability space  $(X,\B,m)$.  Let $g_k\in L^1(X, \B, m)$ be a sequence that converges almost everywhere
and in $L^1$ to $g\in L^1(X, \B, m)$. Then
$$
\lim_{k\to +\infty}\frac{1}{k}\sum_{j=0}^{k-1}g_{k-j}(T^jx)=\E_m(g|\I)(x)
$$
almost everywhere and in $L^1$,  where $\I=\{B\in \B:\; T^{-1}(B)=B\}$.
\end{lem}

\subsection{Conditional measures}
\label{Sect-con}

Here we give a brief introduction to Rohlin's theory of Lebesgue
spaces, measurable partitions and conditional measures.
The reader is referred to \cite{Rohlin1949, Parry1969, EinsiedlerWard2011} for more details.

 A probability space $(X, \B, m)$ is called a {\it Lebesgue space} if it is isomorphic  \mbox{(mod 0)} to a probability space which is the union of  $[0,s]$ for  some $s\in [0, 1]$ with Lebesgue measure and a finite or countable number of atoms.
Now let $(X, \B,m)$ be a Lebesgue space. A {\it measurable partition} $\eta$ of $X$ is a partition of $X$ such that, up to a set of measure zero, the quotient
space $X/\eta$ is separated by a countable number of measurable sets $\{B_i\}$. The quotient space $X/\eta$ with its inherited probability space structure, written as $(X_\eta, \B_\eta, m_\eta)$,  is again a Lebesgue space.
Also, any measurable partition $\eta$ determines a sub-$\sigma$-algebra of $\B$,
 denoted by $\widehat{\eta}$, whose elements are unions of elements of $\eta$.
Conversely, any sub-$\sigma$-algebra  $\B'$ of $\B$ is also countably generated, say by $\{B_i'\}$, and therefore all the sets of the form $\cap A_i$,
where $A_i=B_i^\prime$ or its complement, form a measurable partition.
In particular, $\B$ itself is corresponding to a partition into single points.
An important property of Lebesgue spaces and measurable partitions is the following.

\begin{thm}[Rohlin \cite{Rohlin1949}]
\label{thm-2.1}
 Let $\eta$ be a measurable partition of a Lebesgue space $(X, \B, m)$. Then, for every $x$ in a set of full $m$-measure, there is a probability measure $m^\eta_x$ defined on $\eta(x)$, the element of $\eta$ containing $x$. These measures are uniquely characterized (up to sets of $m$-measure $0$) by the following properties: if $A\subset X$ is a measurable set, then
$x\mapsto m^\eta_x(A)$ is $\widehat{\eta}$-measurable and $m(A)=\int
m^\eta_x(A)d m(x)$. These properties imply that for any $f\in
L^1(X,\B, m)$, $m_x^\eta(f)=\E_m(f|\widehat{\eta})(x)$ for $m$-a.e.~$x$,
and $m(f)=\int \E_m(f|\widehat{\eta})dm$.
\end{thm}

The family of measures $\{m^\eta_x\}$ in the above theorem is called the {\it canonical system of  conditional measures associated with $\eta$}.

Throughout the remaining part of this subsection,  we  assume that $(X,\B,m)$ is
a Lebesgue space.   Suppose that $Y$ is a  complete separable metric space and  $\pi: X\to  Y$ is a
$\B$-measurable map.  Let $\B(Y)$ denote the
Borel-$\sigma$-algebra on $Y$.

According to Rohlin's theory (cf.~\cite[Section~2.5]{Rohlin1949}, \cite[Chapter IV]{Parry1969}), the mapping $\pi$ induces a measurable partition
\begin{equation}
\label{e-xi1}
\xi=\{\pi^{-1}(y):\; y\in Y\}
\end{equation}
of $X$ with  $\widehat{\xi}=\pi^{-1}\B(Y)\mbox{ (mod 0)}$,  and $(X_\xi, \B_\xi, m_\xi)$ is isomorphic $(\mbox{mod 0})$ to $(Y, \B(Y), \pi_*m)$.
The system of conditional measures $\{m^\xi_x\}$  is also called the {\it disintegration of $m$ with respect to $\pi$}.

For $y\in Y$, we use $B(y,r)$
to denote the closed ball in $Y$ of radius $r$ centered at $y$. Moreover  we write for $x\in X$,
\begin{equation}
\label{e-ball}
B^\pi(x,r)=\pi^{-1} B(\pi x,r).
\end{equation}

Furthermore, we say that $Y$ is a {\it Besicovitch space} if $Y$ is a  complete separable metric space and the Besicovich covering lemma (see e.g.~\cite{Mattila1995}) holds in  $Y$.  Besicovich spaces include, for instance,  Euclidean spaces, compact finite-dimensional Riemannian manifolds and  complete separable  ultrametric spaces.

\begin{lem}
\label{lem-2.4} Let $\pi: X\to Y$ be a measurable mapping from a Lebesgue space $(X, \B, m)$ to a Besicovitch space $Y$. Let $\eta$ be a measurable partition of $X$.  Then the following properties hold.
\begin{itemize}
\item[(1)] Let $A\in \B$. Then for $m$-a.e.~$x\in X$,
\begin{equation*}
 \lim_{r\to 0}\frac{ m_x^\eta(B^\pi(x,r)\cap
A)}{ m_x^\eta(B^\pi(x,r))}=\E_m(\chi_A|\hat{\eta}\vee
\pi^{-1}\B(Y))(x).
\end{equation*}
\item[(2)] Let $\alpha$ be a finite or countable measurable partition of $X$.
Then for $m$-a.e.~$x\in X$,
\begin{equation*}
\lim_{r\to 0}\log \frac
{m^\eta_x\left(B^\pi(x,r)\cap \alpha(x)\right)}
{m^\eta_x\left(B^\pi(x,r)\right)}
=-{\bf I}_m
\left(\alpha|\hat{\eta}\vee\pi^{-1}\B(Y)\right)(x).
\end{equation*}
Furthermore, set
\begin{equation*}
g(x)=-\inf_{r>0}\log\frac
{m^\eta_x\left(B^\pi(x,r)\cap \alpha(x)\right)}
{m^\eta_x\left(B^\pi(x,r)\right)}
\end{equation*}
and assume $H_m(\alpha)<\infty$. Then $g\geq 0$ and $g\in L^1(X,{\B},m)$.
\end{itemize}
\end{lem}
\begin{proof}
These properties have been  proved in \cite[Lemma 3.3, Proposition 3.5]{FengHu2009} in the case when $Y=\R^d$.  The proofs there remain valid for the general case when $Y$ is a Besicovitch space.
\end{proof}

\begin{rem}In the above lemma, we have  $\E_m(\chi_A|\hat{\eta}\vee
\pi^{-1}\B(Y))=\E_m(\chi_A|\hat{\eta}\vee
\hat{\xi})$ and ${\bf I}_m
\left(\alpha|\hat{\eta}\vee\pi^{-1}\B(Y)\right)={\bf I}_m
\left(\alpha|\hat{\eta}\vee\hat{\xi}\right)$ $m$-a.e., where $\xi$ is given by \eqref{e-xi1}. This is because $\widehat{\xi}=\pi^{-1}\B(Y)\mbox{ (mod 0)}$.
\end{rem}

\begin{de}
Two probability measures $m_1$ and $m_2$ on a measurable space $(X, \B)$ are said to be strongly equivalent if there exists a positive constant $C$ such that
$C^{-1}m_1(A)\leq  m_2(A)\leq C m_1(A)$ for all $A\in \B$.
\end{de}

\begin{lem}
\label{lem-2.8}
Let $\pi: X\to Y$ be a measurable mapping from a Lebesgue space $(X, \B, m_1)$ to a Besicovitch space $Y$. Let $\xi$ be the measurable partition of $X$ given in \eqref{e-xi1}.
Suppose $m_2$ is another probability measure on $(X,\B)$ strongly equivalent to $m_1$. Then for $m_1$-a.e.~$x$,
$(m_1)^\xi_x$ and $(m_2)^\xi_x$ are strongly equivalent.
\end{lem}
\begin{proof}
 It is a direct consequence of the following standard result (see, e.g. \cite[Proposition 6.1]{LiuQian1995}): Let $\alpha$ be a measurable partition of a Lebesgue space $(X, {\mathcal B}, m)$ and $\nu$  another probability measure on ${\mathcal B}$ which is absolutely continuous with respect to $m$. Then for $\nu$-a.e.$x$, the conditional measure $\nu^{\alpha}_x$ is absolutely continuous with respect to $m^\alpha_x$ on $\alpha(x)$ and
$$
\frac{d \nu^{\alpha}_x}{dm^{\alpha}_x}
=\frac{g|_{\alpha(x)}}{\int_{\alpha(x)} g \;dm^{\alpha}_x},
$$
where $g:=d\nu/dm$.
\end{proof}

\begin{lem}
\label{lem-2.9}
Let $\pi: X\to Y$ be a measurable mapping from a Lebesgue space $(X, \B, m)$ to a Besicovitch space $Y$. Let $\xi$ be the measurable partition of $X$ given in \eqref{e-xi1}.
Suppose $A\in \B$ with $m(A)>0$ and let $m_A$ be the probability measure given by $m_A(E)=m(A\cap E)/m(A)$ for $E\in \B$. Then for $m$-a.e.~$x\in A$,  $(m_A)^\xi_x=(m^\xi_x)_A$, that is,
$$
(m_A)^\xi_x(E)=\frac{m^\xi_x(A\cap E)}{m_x^\xi(A)}  \qquad \mbox{ for all } E\in \B.
$$
 \end{lem}
\begin{proof}
Again it is a direct consequence of \cite[Proposition 6.1]{LiuQian1995}.
\end{proof}
\subsection{Induced transformations}
\label{Sect-2.3}
Let $(X,\B, m, T)$ be an invertible measure-preserving system.  Fix $N\in \N$ and $F\in \B$  with $m(F)>0$.  By the Poincar\'{e} recurrence theorem, the {\it first return map to $F$ associated with $T^N$}, defined by
$$r_F(x)=\inf\{n\geq 1:\; T^{Nn}(x)\in F\},
$$ exists almost everywhere. The map $T_F:\; F\to F$ defined almost everywhere by
$$
T_F(x)=T^{Nr_F(x)}(x)
$$
is called the {\it transformation induced by $T^N$ on the set $F$}.

For $n\geq 1$, set $F_n=\{x\in F:\; r_F(x)=n\}$.  Write $$\B|_F:=\{B\cap F:\; B\in \B\}, \quad  m_F:=\frac{1}{m(F)}m|_F,$$
where $m|_F$ stands for the restriction of $m$ on $F$, that is,  $m|_F(B)=m(B\cap F)$ for $B\in \B$.
The following result is well-known (see e.g.  \cite[pp.~61-63]{EinsiedlerWard2011} and
\cite[pp.~257-258]{Petersen1983} for a proof).

\begin{lem}
\label{lem-induce}
\begin{itemize}
\item[(i)]
The induced transformation $T_F$ is a measure-preserving transformation on the space $(F, \B|_F, m_F)$.
\item[(ii)]
The family of sets $\{T^{Nj} F_n\}_{n\geq 1, \;0\leq j\leq n-1}$ are disjoint, and hence $$\sum_{n=1}^\infty n\;m(F_n)\leq 1.
\footnote{This is actually an equality in the ergodic case, but since we consider also non-ergodic measures, there is only an inequality.}
$$
\item[(iii)] $\displaystyle -\sum_{n=1}^\infty m(F_n)\log m(F_n)<\infty$.
\end{itemize}
\end{lem}

Set $\I=\{B\in \B:\; T^{-1}(B)=B\}$ and $\I_F:=\{B\in \B|_F:\; (T_F)^{-1}(B)=B\}$.   Recall that $N$ is a fixed positive integer and $r_F$ is the first return map to $F$ with respect to $T^N$.  The following  result will be needed in the proof of   Theorem \ref{thm-1.0}.

\begin{lem}
\label{lem-inderg}
Let $g\in L^1(X, \B,m)$. Set $G(x)=\sum_{j=0}^{Nr_F(x)-1} g(T^jx)$ for $x\in F$. Then $G\in L^1(F, \B|_F, m_F)$. Moreover,
\begin{equation}
\label{e-identity}
N \E_m(g|\I)(x)=\frac{\E_{m_F}(G|\I_F)(x)}{\E_{m_F}(r_F|\I_F)(x)}
\end{equation}
for $m$-a.e.~$x\in F$.
\end{lem}
\begin{proof}
First notice that
\begin{equation*}
\begin{split}
\int_F |G| \; dm_F & = \frac{1}{m(F)} \sum_{n=1}^\infty \int_{F_n} |G|\;dm\\
& \leq \frac{1}{m(F)}\sum_{n=1}^\infty\sum_{p=0}^{Nn-1} \int_{F_n} |g\circ T^p|\; dm\\
&= \frac{1}{m(F)} \sum_{n=1}^\infty\sum_{p=0}^{Nn-1} \int_{T^{p} F_n} |g|\; dm \quad  \mbox{(since $T$ is invertible and preserves $m$)} \\
&= \frac{1}{m(F)} \sum_{n=1}^\infty\sum_{k=0}^{N-1}\sum_{j=0}^{n-1} \int_{T^{Nj+k} F_n} |g|\; dm\\
&= \frac{1}{m(F)} \sum_{k=0}^{N-1}\sum_{n=1}^\infty\sum_{j=0}^{n-1} \int_{T^{Nj+k} F_n} |g|\; dm\\
&\leq  \frac{N}{m(F)}\int_{X} |g|\; dm,
\end{split}
\end{equation*}
where in the last inequality we have used the fact that for any $k$, the sets in the collection $\{T^{Nj+k} F_n:\; n\in \N, 0\leq j\leq n-1\}$ are disjoint  (see Lemma \ref{lem-induce}(ii)). Hence $G\in L^1(m_F)$.  Below we prove \eqref{e-identity}.

Consider  the sequence of integer-valued functions $(n_k(x))_{k=0}^\infty$, which are defined on $F$ almost everywhere by $n_0(x)=0$, and
$$
n_k(x)=\sum_{j=0}^{k-1}r_F(T_F^jx)\quad \mbox{ for } k\geq 1,
$$
where $T_F^j:=(T_F)^j$.
Clearly,  $n_k(x)\geq k$ and $T_F^k(x)=T^{N n_k(x)}(x)$. Hence,
\begin{equation*}
\begin{split}
\sum_{j=0}^{k-1}G(T_F^jx)&=\sum_{j=0}^{k-1}\sum_{p=0}^{Nr_F(T_F^jx)-1} g(T^p(T_F^jx))\\
&=\sum_{j=0}^{k-1}\sum_{p=0}^{Nr_F(T^j_Fx)-1} g(T^{Nn_j(x)+p}x)\\
&=\sum_{j=0}^{k-1}\sum_{\ell=Nn_j(x)}^{Nn_{j+1}(x)-1} g(T^\ell x)\\
&=\sum_{i=0}^{Nn_k(x)-1} g(T^ix).
\end{split}
\end{equation*}
By the Birkhoff ergodic theorem, we have
\begin{equation}
\label{e-Gg}
\begin{split}
\lim_{k\to +\infty}\frac{1}{n_k(x)}\sum_{j=0}^{k-1}G(T_F^jx) &=\lim_{k\to +\infty}\frac{1}{n_k(x)}\sum_{i=0}^{Nn_k(x)-1} g(T^ix)\\
&=N \E_m(g|\I)(x)
\end{split}
\end{equation}
for $m$-a.e.~$x\in F$. Applying the Birkhoff ergodic theorem again, we have
\begin{align*}
&\lim_{k\to +\infty}\frac{1}{k}\sum_{j=0}^{k-1}G(T_F^jx)=\E_{m_F}(G|\I_F)(x) \quad \mbox{ and } \\
&\lim_{k\to +\infty}\frac{n_k(x)}{k}=\lim_{k\to +\infty}\frac{1}{k}\sum_{j=0}^{k-1}r_F(T_F^jx)=\E_{m_F}(r_F|\I_F)(x)
\end{align*}
\mbox{ for $m$-a.e.~$x\in F$}. Here we have used the fact that $r_F\in L^1(F, \B|_F, m_F)$, which follows directly from Lemma \ref{lem-induce}(ii).  Taking quotient we get
$$
\lim_{k\to +\infty}\frac{1}{n_k(x)}\sum_{j=0}^{k-1}G(T_F^jx)\\
=\E_{m_F}(G|\I_F)(x)/\E_{m_F}(r_F|\I_F)(x)
$$
for $m$-a.e.~$x\in F$. Combining this with \eqref{e-Gg} yields \eqref{e-identity}.
\end{proof}

\subsection{Oseledets' multiplicative ergodic theorem}
\label{Sect-ose}
 For $x,y\in \R^d\backslash \{0\}$, let $\measuredangle (x,y)$ denote the angle between the lines $\ell_x$ and $\ell_y$, where $\ell_x$ stands for the line in $\R^d$ passing through the origin and $x$. In such definition,  we always have  $\measuredangle (x,y)\in [0,\pi/2]$ and
$$
\sin \measuredangle (x,y)=\frac{(\|x\|^2\|y\|^2-\langle x, y\rangle^2)^{1/2}}{\|x\|\|y\|},
$$
where $\langle\cdot,\cdot\rangle$ is the standard inner product in $\R^d$. Similarly the angle between linear subspaces $U,V$ of  $\R^d$  with $U\cap V=\{0\}$  is defined by
$$
\sin \measuredangle(U, V)=\inf_{x\in U\backslash\{0\},\; y\in V\backslash\{0\}}\sin \measuredangle (x,y).
$$
  We will require the following version of Oseledets' multiplicative ergodic theorem,
due to Froyland et al. \cite[Theorem 4.1]{FroylandLloydQuas2010}:

\begin{thm}
\label{thm-2}  Let $T$ be an invertible measure-preserving transformation of the Lebesgue space $(X,\B, m)$.  Let $M: X\to {\rm Mat}_d(\R)$ be a measurable function such that
 $$\int \log^+\|M(x)\|\; dm(x)<\infty.$$
Then there exists a measurable set $X'\subseteq X$ with $T(X')=X'$ and $m(X')=1$, such that   for each $x\in X'$, there are positive integers $s(x), k_1(x),\ldots,k_{s(x)}(x)$ with $k_1(x)+\cdots+ k_{s(x)}(x)=d$,  numbers $\lambda_1(x)>\cdots>\lambda_{s(x)}(x)\geq -\infty$ and a splitting  $\R^d=E_x^1\oplus \cdots \oplus E_x^{s(x)}$ so that the following hold.
\begin{itemize}
\item[(i)] $\dim E_x^i=k_i(x)$.
\item[(ii)]  $M(x) E_x^i\subseteq E^i_{T x}$ (with equality if $\lambda_i(x)>-\infty$).
\item[(iii)] For $1\leq i\leq s(x)$ and  $v\in E_x^i\backslash\{0\}$,
\begin{equation*}
\lim_{n\to \infty}  \frac{1}{n} \log \|M(T^{n-1}x)\cdots M(x)v\|=\lambda_i(x),
\end{equation*}
 with uniform convergence on any compact subset of $E_x^i\backslash\{0\}$.
\item[(iv)] For $1\leq i\leq s(x)$,
\begin{equation*}
\begin{split}  &\lim_{n\to \infty}  \frac{1}{n}  \max_{v\in E^i_{T^{-n}x},\; \|v||=1} \log \|M(T^{-1}x)\cdots M(T^{-n}x)v\|\\
&=\lim_{n\to \infty}  \frac{1}{n}  \min_{v\in E^i_{T^{-n}x},\; \|v||=1} \log \|M(T^{-1}x)\cdots M(T^{-n}x)v\|=\lambda_i(x).
\end{split}
\end{equation*}
\item[(v)] $\displaystyle\lim_{n\to \pm \infty} \frac{1}{n} \log \measuredangle ( \oplus_{i\in I}E^i_{T^nx}, \; \oplus_{j\in J} E^{j}_{T^n x})=0$ whenever $I\cap J=\emptyset$,
  \item[(vi)] The function $s:X'\to \N$ is measurable and $T$-invariant.
\item[(vii)] The mappings $x\mapsto \lambda_i(x), E^i_x, k_i(x)$  are  measurable on  $\{x: s(x)\geq i\}$, and $\lambda_i(Tx)=\lambda_i(x)$, $k_i(Tx)=k_i(x)$.
\end{itemize}
\end{thm}

 \begin{rem}
 \label{rem-2.10}
 {\rm
\begin{itemize}
\item[(1)] Theorem \ref{thm-2} is only stated in \cite{FroylandLloydQuas2010} for the case  when $m$ is ergodic. It extends directly to the general case by using ergodic decomposition. When $M(x)$ is invertible for all $x$ this is the classic Oseledets' multiplicative ergodic theorem, but we emphasize that the above is valid even in the non-invertible case (in which case the usual statements of Oseledets' theorem only provide a flag and not a splitting).
\item[(2)] The uniform convergence in part (iii) of Theorem \ref{thm-2} is not stated in \cite{FroylandLloydQuas2010}. However it is well-known when $A$ takes values in $GL(\R,d)$, and the argument works also in the general case of ${\rm Mat}_d(\R)$-valued cocycles.  See e.g. \cite[p.~1111]{FengShmerkin2014} for a sketched proof. Part (iv) of Theorem \ref{thm-2} was only implicitly included in the proof of \cite[Theorem 4.1]{FroylandLloydQuas2010}.
\item[(3)]
  The numbers $\lambda_1(x),\ldots, \lambda_{s(x)}(x)$ are called the {\em Lyapunov exponents}  of $M$ at $x$ with respect to $m$. The number $k_i(x)$ is called the {\em multiplicity} of $\lambda_i(x)$. Moreover, $\{(\lambda_i(x), k_i(x))\}_{1\leq i\leq s(x)}$ is called the {\em Lyapunov spectrum} of $(M, m)$ over $X'$.
\item[(4)] The decomposition  $\bigoplus_{i=1}^{s(x)}E_x^i$  is called the {\em Oseledets splitting} of $\R^d$, and $E_x^i$,  $1\leq i\leq s(x)$,  are called the {\em Oseledets subspaces}.
\end{itemize}
}
\end{rem}

\section{Canonical coding maps for average contracting affine IFSs}
\label{S-3}
\

 For $z\geq 0$, write $\log^+z=\max\{0, \log z\}$ and $\log^{-}z=\max\{0, -\log z\}$, with the convention $\log 0=-\infty$.
In this section, we prove the following proposition, which will be used in the proof of our main result.
\begin{pro}
\label{pro-3.1}
Let ${\mathcal S}=\{S_j(x)=M_jx+a_j\}_{j\in \Lambda}$ be an affine IFS on $\R^d$ and $m\in \M_\sigma(\Sigma)$. Suppose that ${\mathcal S}$ is average contracting with respect to $m$.  Let $\pi: \Sigma\to \R^d$ be given by \eqref{e-pi1.4}. Then there exists a  Borel set $E\subset \Sigma$ with $\sigma(E)=E$ and $m(E)=1$ such that for any $x=(x_n)_{n=-\infty}^\infty\in E$,
\begin{itemize}
\item[(i)]
$\pi(x)$ is well-defined, i.e.~the limit in defining $\pi(x)$ in \eqref{e-pi1.4} exists and is finite.
\item[(ii)]
$S_{x_0}(\pi\sigma x)=\pi(x)$.
\item[(iii)]
$
\lim_{n\to \infty}
\frac{1}{n} \log^+ \| \pi(\sigma^n x)\|=0
$.
\end{itemize}
\end{pro}

Part (i) of the above proposition was first proved by Brandt \cite{Brandt1986} in the special case when $m$ is a Bernoulli product measure, and it was then extended by Bougerol and Picard \cite{BougerolPicard1992} to the general case when $m$ is ergodic. For the convenience of the reader, we shall provide a self-contained proof of part (i).

Before proving Proposition \ref{pro-3.1}, we shall first prove the following auxiliary result, which is a variant of Proposition 2.1 in \cite{FengKaenmaki2011}.

\begin{pro}
\label{lem-2.1} Let\,   $T:\;X\to X$ be an ergodic measure-preserving  transformation on a probability space $(X,\B, m)$.
Let $\{f_n\}_{n=1}^\infty$ be a sequence of non-negative  measurable functions on $X$ such that $\log^+f_1\in L^1(m)$ and
\begin{equation}
\label{e-pro}
f_{n+k}(x)\leq f_n(x)f_k(T^n x)
\end{equation}
for all $n,k\in \N$ and $x\in X$. Set $\lambda=\lim_{n\to \infty}({1}/{n})\int \log f_n \; dm$. Then
for any $\epsilon>0$, the following properties hold:
\begin{itemize}
\item[(i)] If $\lambda \neq -\infty$, then for $m$-a.e.~ $x\in X$, there exists a positive integer $n_0(x)$ such that
\begin{equation}
\label{e-2.2} |\log f_n(T^k x)-n\lambda|\leq (n+k)\epsilon
\end{equation}
for all $n\geq n_0(x)$ and $k\geq 0$.
\item[(ii)]
If $\lambda=-\infty$, then for any $N>0$ and  $m$-a.e.~$x\in X$, there exists a positive integer $n_0(x)$ such that
\begin{equation}
\label{e-2.5}
\log f_n(T^k x)\leq -Nn+ (n+k)\epsilon
\end{equation}
for all $n\geq n_0(x)$ and $k\geq 0$.
\end{itemize}
\end{pro}
\begin{proof} Here we  modify the arguments  of   \cite[Proposition 2.1]{FengKaenmaki2011}.  By sub-additivity, $\lambda\leq \int \log f_1\; dm\leq \int \log^+f_1\; dm<\infty$. Below we first prove (i).

 Assume that $\lambda\neq -\infty$.  We first prove that $\log f_j\in L^1(m)$ for each $j\in \N$. To see this, observe that by \eqref{e-pro},  $$\log^+ f_j\leq \sum_{k=0}^{j-1} \log^+ (f_1\circ T^k)\in L^1(m).$$
 It remains to show that $\log^- f_j\in L^1(m)$. Suppose this is not true, then $$\int \log f_j\; dm= \int \log^+ f_j-\log^- f_j\; dm=-\infty,$$ so by the sub-additivity of $\{f_n\}$,   $$\lambda=\inf_k\frac{1}{k}\int \log f_k\; dm=-\infty,$$ leading to a contradiction. This proves that $\log f_j\in L^1(m)$ for each $j$.

  Next let $\epsilon>0$ and take $0<\delta<\epsilon/6$. By the Kingman's sub-additive ergodic theorem,
 for $m$-a.e.~$x\in X$ there exists $n_1(x)$ such that
$$
|\log f_n(x)-n\lambda |\leq n\delta \quad \mbox{ for all }n \geq n_1(x).
$$
Setting $n_2(x):=\max_{1\leq j\leq n_1(x)}|\log f_j(x)-j\lambda|/\delta$, we see that  $n_2(x)<\infty$ a.e. and
\begin{equation}
\label{e-rev1}
|\log f_k(x)-k\lambda|\leq (n_2(x)+k)\delta\quad \mbox{ for all }k \in \N.
\end{equation}
Hence by \eqref{e-pro} and \eqref{e-rev1},  for every $n\geq n_2(x)$ and $k\geq 0$ we have
\begin{equation}
\label{e-t}
\begin{split}
\log f_n(T^k x)&\geq \log f_{n+k}(x) - \log f_k(x)\\
&\geq \big((n+k)\lambda-(n_2(x)+n+k)\delta\big)-\big(k\lambda+(n_2(x)+k)\delta\big)\\
&= n\lambda-(2n_2(x)+n+2k)\delta\\
&\geq n\lambda-(n+k)\epsilon.
\end{split}
\end{equation}

To see the opposite inequality, take $\ell$ large enough such that $|\beta-\lambda|<\delta$, where
$$
\beta:=\frac{1}{\ell}\int \log f_\ell \; dm.
$$
Applying the Birkhoff ergodic theorem to the  integrable functions $\log f_j$ ($j=1,\ldots, 2\ell$), we obtain
\begin{equation}
\label{e-(13)}
\lim_{p\to \infty} \frac{1}{p}\log f_j(T^px)=0 \quad \mbox{ for $1\leq j\leq 2\ell$ and $m$-a.e.~$x$}.
 \end{equation}

Let $n\geq 2\ell$ and $x\in X$. Write $n=q\ell +s$ with $\ell\leq s\leq 2\ell-1$. By sub-multiplicativity, we have
$$
f_n(x)\leq f_j(x) \left(\prod_{p=0}^{q-1} f_\ell(T^{p\ell+j}x)\right) f_{s-j}(T^{q\ell+j}x), \quad  j=0, 1, \ldots, \ell-1,
$$
where we take the convention that $f_0\equiv 1$. Taking  product of these  inequalities yields
$$
(f_n(x))^\ell \leq  \left(\prod_{j=0}^{\ell-1} f_j(x)\right) \left(\prod_{p=0}^{q\ell -1} f_\ell(T^{p}x)\right) \left(\prod_{j=0}^{\ell-1}f_{s-j}(T^{q\ell +j}x)\right),
$$
so for $k\geq 0$,
$$
(f_n(T^kx))^\ell\leq  \left(\prod_{j=0}^{\ell-1} f_j(T^k x)\right) \left(\prod_{p=k}^{q\ell +k-1} f_\ell(T^{p}x)\right) \left(\prod_{j=0}^{\ell-1}f_{s-j}(T^{q\ell +k+j}x)\right).
$$
Taking logarithm and dividing both sides by $\ell$ we have
\begin{equation}
\label{e-(14)}
\log (f_n(T^kx))\leq \left(\sum_{i=0}^{n+k-s-1} \frac{1}{\ell} \log f_\ell(T^{i}x)\right) - \left(\sum_{i=0}^{k-1}  \frac{1}{\ell} \log f_\ell(T^{i}x)\right)+\Lambda_1+\Lambda_2,
\end{equation}
where $\Lambda_1:=\sum_{j=0}^{\ell-1} \frac{1}{\ell} \log f_j(T^k x)$,
$\Lambda_2:=\sum_{j=0}^{\ell-1} \frac{1}{\ell}
\log f_{s-j}(T^{q\ell +k+j}x)$.

 Similar to \eqref{e-rev1}, by using the Birkhoff ergodic theorem and \eqref{e-(13)}, we see that for $m$-a.e.~$x$ there exists $n_3(x)$ such that for every $n\geq n_3(x)$ and every $k\geq 0$ and $1\leq j\leq 2\ell$,
\begin{equation*}
\begin{split}
&\left|\sum_{i=0}^{n+k-s-1}\frac{1}{\ell} \log f_\ell(T^ix)-(n+k-s)\beta\right|  \leq (n+k-s)\delta,\\
&\left|\sum_{i=0}^{k-1}\frac{1}{\ell} \log f_\ell(T^ix)-k\beta\right|  \leq (n_3(x)+k)\delta,\\
&|\log f_j(T^kx)|\leq (n_3(x)+k)\delta,\\
\end{split}
\end{equation*}
where the third inequality implies that $\Lambda_1\leq  (n_3(x)+k)\delta$ and $$\Lambda_2\leq (n_3(x)+q\ell+k+\ell-1)\delta\leq (n_3(x)+n+k)\delta.$$

Applying the above inequalities to  \eqref{e-(14)}, we see that  for $m$-a.e.~$x\in X$, for every $n\geq n_3(x)$ and every $k\geq 0$,
\begin{align*}
\log f_n(T^k x) &\leq  (n+k-s)(\beta+\delta) - (k\beta-(n_3(x)+k)\delta)+(n_3(x)+k)\delta\\
&\qquad\qquad  +(n_3(x)+n+k)\delta\\
 &= (n-s)\beta+(2n+4k+3n_3(x)-s)\delta\\
 &\leq n \lambda +(n+k)\epsilon.
\end{align*}
From this and (\ref{e-t}) we see that (\ref{e-2.2}) holds for every $n\geq n_0(x):=\max\{n_2(x), n_3(x)\}$ and every $k\geq 0$. This prove (i).

To see (ii), suppose $\lambda=-\infty$. By the Kingman's sub-additive ergodic theorem, $\lim_{n\to \infty}(1/n)\log f_n(x)=-\infty$ for $m$-a.e.~$x$. Fix $N$ and define $$\tilde{f}_n(x)=\max\{f_n(x), e^{-nN}\} \quad  \mbox{ for }n\in \N, \; x\in X.$$
 Then
$\lim_{n\to \infty}(1/n)\log \tilde{f}_n(x)=-N$ for $m$-a.e.~$x$. Meanwhile it is direct to check that $\{\tilde{f}_n\}_{n=1}^\infty$ is sub-multiplicative (i.e., \eqref{e-pro} holds for $\{\tilde{f}_n\}$), so by the Kingman's sub-additive ergodic theorem, $$\lim_{n\to \infty}(1/n)\int \log \tilde{f}_n \; dm=-N.$$
 Applying (i) to  $\{\tilde{f}_n\}_{n=1}^\infty$ yields that for  $m$-a.e.~$x\in X$, there exists a positive integer $n_0(x)$ such that
$$
\log f_n(T^k x)\leq \log \tilde{f}_n(T^k x)\leq -Nn+ (n+k)\epsilon
$$
for all $n\geq n_0(x)$ and $k\geq 0$.  This completes the proof of the proposition.
\end{proof}

As a direct corollary of Proposition \ref{lem-2.1}, we have the following.

\begin{cor}
\label{cor-2.1} Under the assumptions of Proposition \ref{lem-2.1}, for any
$\epsilon, N>0$ and for $m$-a.e.~$x\in X$, there exists
$c(x)>0$ such that
$$|f_n(T^k x)|\leq c(x) \exp(n \max\{\lambda, -N\})\exp ((n+k)\epsilon)$$
for all  $n\geq 1$ and $k\geq 0$.
\end{cor}

\begin{proof}[Proof of Proposition \ref{pro-3.1}]
Without loss of generality we may assume that  $m$ is ergodic, since  the general case can be proved  by considering the ergodic decomposition of $m$.

Set $f_n(x)=\| M_{x_0} \cdots M_{x_{n-1}}\|$ for $x\in \Sigma$ and $n\geq 1$.  Let $f_0(x)\equiv 1$ for convention. Since ${\mathcal S}$ is average contracting with respect to $m$,
we have
$$
\lim_{n\to \infty} \frac{1}{n}\int \log f_n \;dm=:\lambda<0. $$

Let $0<\epsilon<-\lambda/3$. Applying Corollary \ref{cor-2.1} to $\{f_n\}$ and the shift map $\sigma:\; \Sigma\to \Sigma$ (in which we take $N=2\epsilon$), we see that for $m$-a.e.~$x$,  there exists $c(x)>0$ such that
$$
f_n(\sigma^k x)\leq c(x)  e^{-2n\epsilon} e^{(n+k)\epsilon}
$$
for any $n\geq 1$ and $k\geq 0$.  It follows that for $m$-a.e.~$x$,
\begin{align*}
\sum_{n=0}^\infty \|M_{x_{k}}\cdots M_{x_{k+n-1}} a_{x_{k+n}}\| &\leq
(\max_i\|a_i\|) \sum_{n=0}^\infty f_n(\sigma^k x)  \\
& \leq (\max_i\|a_i\|) c(x) \sum_{n=0}^\infty   e^{-2n\epsilon} e^{(n+k)\epsilon}\\
&=(\max_i\|a_i\|) c(x) (1-e^{-\epsilon})^{-1} e^{k\epsilon}
\end{align*}
for all $k\geq 0$.  It follows that for $m$-a.e.~$x$, $\pi(\sigma^kx)$ is well-defined and
$\|\pi(\sigma^kx)\|\leq (\max_i\|a_i\|) c(x) (1-e^{-\epsilon})^{-1} e^{k\epsilon}$
 for all $k\geq 0$.  That is enough to conclude the proposition.
\end{proof}

\section{Measurable partitions associated with affine IFSs}
\label{S-4}

Let ${\mathcal S}=\{M_j x+a_j\}_{j\in \Lambda}$ be an affine IFS on $\R^d$ and $m\in \M_\sigma(\Sigma)$. Suppose that ${\mathcal S}$ is average contracting with respect to $m$.
In this section, under an additional assumption formulated later in \eqref{e-assump},  we construct a finite family of measurable partitions of $\Sigma$ and give some properties of these partitions and the corresponding conditional measures of $m$.

Define $M:\;\Sigma\to {\rm Mat}_d(\R)$ by
$$
M(x)=M_{x_{-1}},\quad x=(x_n)_{n=-\infty}^\infty.
$$

  Applying Theorem \ref{thm-2} to the measure-preserving system $(\Sigma, \sigma^{-1}, m)$ and the matrix cocycle $M$, we get  a measurable $\Sigma'\subset \Sigma$ with $\sigma(\Sigma')=\Sigma'$ and  $m(\Sigma')=1$,  so that the Lyapunov spectrum
$$
\{(\lambda_i(x), k_i(x))\}_{1\leq i\leq s(x)}
$$
and the Oseledets splitting
$$
\R^d=E_x^1\oplus \cdots \oplus E_x^{s(x)}
$$
are well-defined for $x\in \Sigma'$ (cf. Remark~\ref{rem-2.10}). In this case, for any $x\in \Sigma'$ and $1\leq i\leq s(x)$,

\begin{equation}
\label{e-neg}
\lim_{n\to \infty}  \frac{1}{n} \log \|M_{x_{-n}}\cdots M_{x_{-1}}v\|=\lambda_i(x) \quad \mbox{ for $v\in E_x^i\backslash\{0\}$},
\end{equation}
 with uniform convergence on any compact subset of $E_x^i\backslash\{0\}$,
\begin{equation}
\label{e-e?}
\begin{split}  &\lim_{n\to \infty}  \frac{1}{n}  \max_{v\in E^i_{\sigma^nx},\; \|v||=1} \log \|M_{x_0}\cdots M_{x_{n-1}}v\|\\
 &\mbox{}\quad  =\lim_{n\to \infty}  \frac{1}{n}  \min_{v\in E^i_{\sigma^n x},\; \|v||=1} \log \|M_{x_0}\cdots M_{x_{n-1}}v\|=\lambda_i(x),
\end{split}
\end{equation}
and
\begin{equation}
\label{e-e4.3}
\limsup_{n\to \infty}  \frac{1}{n}  \max_{v\in \oplus_{j=i}^{s(x)}E^j_{\sigma^nx},\; \|v||=1} \log \|M_{x_0}\cdots M_{x_{n-1}}v\|\leq
\lambda_i(x).
\end{equation}

In addition, by Proposition \ref{pro-3.1} we may assume that the coding map $\pi$  is well-defined on $\Sigma'$ and that
\begin{equation}
\label{e-e4.1}
\lim_{n\to \infty}\frac{1}{n} \log^+ \|\pi(\sigma^n x)\|=0 \quad \mbox{ for } x\in \Sigma'.
\end{equation}

Define for $x\in \Sigma'$,
\begin{equation}
V_x^i:=\oplus_{j=i+1}^{s(x)} E_x^j \quad \mbox{ for }  i=0,\ldots,s(x)-1, \quad \mbox{ and } \quad V_x^{s(x)}:=\{0\}.
\end{equation}
By \eqref{e-neg}, we have
\begin{equation}
\label{e-v1}
V_x^i=\left\{v\in \R^d:\; \limsup_{n\to \infty}  \frac{1}{n} \log \|M_{x_{-n}}\cdots M_{x_{-1}}v\|\leq \lambda_{i+1}(x)\right\}
\end{equation}
for $x\in \Sigma'$, $i=0,\ldots, s(x)-1$.

For $x=(x_j)_{j=-\infty}^\infty\in \Sigma$, we write $x^{-}=(x_j)_{j=-\infty}^{-1}$. The  following simple fact  is our starting point in constructing measurable partitions of $\Sigma'$.

\begin{lem}
\label{lem-3.1}
Let $x,y\in \Sigma'$ with $x^{-}=y^{-}$. Then   $s(x)=s(y)$ and  $\lambda_i(x)=\lambda_i(y)$ for $1\leq i\leq s(x)$. Moreover,
$V_x^{i}=V_y^{i}$ for $0\leq i\leq s(x)$.
\end{lem}
 \begin{proof}
 For $x\in \Sigma'$ and $v\in \R^d\setminus \{0\}$, define
 $$
 \lambda(x,v):=\lim_{n\to \infty} \frac{1}{n}\log \|M_{x_{-n}}\cdots M_{x_{-1}}v\|.
 $$
 By \eqref{e-neg}, the above limit always exists and takes values in $\{\lambda_i(x):\; 1\leq i\leq s(x)\}$.   Clearly  $ \lambda(x,v)$ only depends on $v$ and $x^{-}$. Hence for $x, y\in \Sigma'$ with $x^-=y^-$, we have $s(x)=s(y)$ and $\lambda_i(x)=\lambda_i(y)$ for $1\leq i\leq s(x)$; by \eqref{e-v1} we also have $V_x^i=V_y^i$ for $1\leq i\leq s(x)$.
This completes the proof of the lemma. \end{proof}

In the remaining part  of this section, we  always make the following assumption:
\begin{equation}
\label{e-assump} \mbox{$s(x), k_1(x),\ldots, k_{s(x)}(x)$  are constant for $m$-a.e.~$x\in \Sigma'$}.
\end{equation}
Here we don't make  the stronger assumption that $m$ is ergodic. Let us write these constants as $s, k_1,\ldots, k_s$.

Below we construct  a finite family of measurable partitions $\xi_0,\ldots, \xi_s$ of $\Sigma'$.

Let $\xi_0$ be the partition of $\Sigma'$ so that the $\xi_0$-atom containing $x=(x_j)_{j=-\infty}^{+\infty}\in \Sigma'$ is given by
$$
\xi_0(x)=\{y=(y_j)_{j=-\infty}^\infty\in \Sigma':\; y_j=x_j \mbox{ for }j\leq -1\}.
$$
By Lemma  \ref{lem-3.1},  $V_y^i=V_x^i$ for any $y\in \xi_0(x)$ and $i\in \{0,1,\ldots, s\}$.

Similarly, for  $i\in \{1,\ldots, s\}$, we define the partition $\xi_i$ of $\Sigma'$ by
$$
\xi_i(x)=\{y=(y_j)_{j=-\infty}^\infty\in \xi_0(x):\; \pi y-\pi x  \in V_x^i\},\qquad x\in \Sigma'.
$$

\begin{lem}
$\xi_0,\ldots, \xi_s$ are measurable partitions of $(\Sigma',\B(\Sigma'), m)$.
 \end{lem}
 \begin{proof}
By Rohlin theory (cf. \cite[Section~2.5]{Rohlin1949}, \cite[Chapter IV]{Parry1969}), it is enough to show that for every $i\in \{0,1,\ldots, s\}$,  one can construct a measurable mapping $\pi_i$ from $\Sigma'$ to a complete separable metric space $Y_i$ such that $\xi_i$ is induced by $\pi_i$, in the sense that $\xi_i=\{\pi_i^{-1}(y):\; y\in Y_i\}$. Below we construct such mappings $\pi_i$.

Let $\Sigma^-:=\{(x_n)_{n=-\infty}^{-1}:\; x_n\in \Lambda\mbox{ for all }n\leq -1\}$ and endow it with a suitable metric compatible to the  product topology.   For $j\in \{0,\ldots, d\}$, the set of all $j$-dimensional affine subspaces in $\R^d$  forms a closed smooth manifold, which is called the $(d, j)$-{\it affine Grassmannian} and is denoted by ${\rm Graff}(d, j)$.

Set $Y_i=\Sigma^-\times {\rm Graff}(d, k_{i+1}+\cdots+k_s)$ for $i\in \{0,\ldots, s-1\}$ and $Y_s=\Sigma^-\times \R^d$. Define $\pi_i:\;\Sigma'\to Y_i$ ($i=0,1,\ldots, s$) by
$$
x\mapsto (x^-, V_x^i+\pi x).
$$
It is readily checked that for each $i$, $\pi_i$ is measurable and $\xi_i$ is induced by $\pi_i$.
\end{proof}

\begin{rem}
\label{rem-tt} The above construction of the measurable partitions $\xi_0,\ldots, \xi_s$ is different from that built in the previous work of \cite{FengHu2009, BaranyKaenmaki2015}. In \cite{FengHu2009}, the partitions were made on the one-sided shift space due to the simple structure of Oseledets splitting subspaces. In \cite{BaranyKaenmaki2015}, the partitions were made on the product space of the self-affine set and the flag manifolds.
\end{rem}

Let $\P$ be the canonical partition of $\Sigma'$ given  in \eqref{e-pp}.
For $n\in \N$, set $$\P_0^{n-1}=\bigvee_{j=0}^{n-1} \sigma^{-j}\P,$$ where $\vee$ stands for  the join of partitions (cf. \cite{Parry1981}).

For convenience, write
\begin{equation}
\label{e-qn}
Q_{n,\epsilon}:=\left\{x\in \Sigma':\;  \|\pi \sigma^j x\|\leq (1/2) e^{j\epsilon/2} \mbox{ for all }j\geq n\right\}
\end{equation}
for $n\in\N$ and $\epsilon>0$. Below we give several lemmas to further illustrate the properties of  $\xi_i$ and the associated conditional measures.

\begin{lem}
\label{lem-3.4}
\begin{itemize}
\item[(1)]
For  $x\in \Sigma'$, $i\in \{0,\ldots, s\}$ and $n\in \N$,
$$
\xi_i(x)\cap \P_{0}^{n-1} (x)= \sigma^{-n}(\xi_i(\sigma^n x)).
$$
As a consequence, $\xi_i\vee \P_{0}^{n-1}=(\sigma^{-n}\xi_i)\vee  \P_{0}^{n-1}=\sigma^{-n}\xi_i$.
\item[(2)]
Let  $x\in \Sigma'$ and $\epsilon>0$. Then  there exists $n_0(x)$ such that for $i\in \{0,\ldots, s-1\}$,
\begin{equation}
\label{e-l2}
Q_{n,\epsilon}\cap \xi_i(x)\cap \P_{0}^{n-1} (x)\subset \left \{\begin{array}{ll}
B^\pi(x, e^{n (\lambda_{i+1}(x)+2\epsilon)}) & \mbox{ if }\lambda_{i+1}(x)\neq -\infty\\
B^\pi(x, e^{-n/\epsilon}) & \mbox{ if }\lambda_{i+1}(x)=-\infty
\end{array}
\right.
\end{equation}
when  $n\geq n_0(x)$, here $B^\pi(x,r)$ is defined as in \eqref{e-ball}.   Moreover,
\begin{equation}
\label{e-l3}
Q_{n,\epsilon}\cap \P_{0}^{n-1} (x)\subset\left \{\begin{array}{ll}
B^\pi(x, e^{n (\lambda_{1}(x)+2\epsilon)}) & \mbox{ if }\lambda_{1}(x)\neq -\infty\\
B^\pi(x, e^{-n/\epsilon}) & \mbox{ if }\lambda_{1}(x)=-\infty
\end{array}
\right.
\end{equation}
when $n\geq n_0(x)$.
\end{itemize}
\end{lem}
\begin{proof}
We first prove (1). Let $x=(x_j)_{j=-\infty}^\infty\in \Sigma'$,  $i\in \{0,\ldots, s\}$ and  $n\in \N$.   We only prove that $\xi_i(x)\cap \P_{0}^{n-1} (x)\subset\sigma^{-n}(\xi_i(\sigma^n x))$. The proof of the other direction is similar.

Let $y=(y_j)_{j=-\infty}^\infty \in \xi_i(x)\cap \P_{0}^{n-1} (x)$. Then $\pi y-\pi x\in V_x^i$ and $y_j=x_j$ for $j\leq n-1$. By Proposition \ref{pro-3.1}(ii), \begin{equation}
\label{e-pi}
\begin{split}
 \pi y-\pi x & =S_{y_0\ldots y_{n-1}}(\pi \sigma^n y)-S_{x_0\ldots x_{n-1}}(\pi \sigma^n x)\\
 &=S_{x_0\ldots x_{n-1}}(\pi \sigma^n y)-S_{x_0\ldots x_{n-1}}(\pi \sigma^n x)\\
 &=M_{x_0\ldots x_{n-1}}(\pi \sigma^n y-\pi \sigma^n x),\\
 \end{split}
 \end{equation}
 here and afterwards we write $M_{i_1\ldots i_n}$ for $M_{i_1}\cdots M_{i_n}$.
Since $\pi y-\pi x\in V^i_x$, by \eqref{e-pi} and \eqref{e-v1} we have

 \begin{equation*}
\begin{split}
\limsup_{k\to \infty}&\frac{1}{n+k} \log \|M_{x_{-k}\ldots x_{-1}x_0\ldots x_{n-1}}(\pi \sigma^n y-\pi \sigma^n x)\|\\
=&\limsup_{k\to \infty}\frac{1}{n+k}\log  \|M_{x_{-k}\ldots x_{-1}}(\pi  y-\pi  x)\|
\leq  \lambda_{i+1}(x)=\lambda_{i+1}(\sigma^nx).
\end{split}
\end{equation*}
Applying  \eqref{e-v1} to $V^i_{\sigma^nx}$ gives $\pi \sigma^n y-\pi \sigma^n x\in V_{\sigma^nx}^i$. In the meantime, since $y_j=x_j$ for $j\leq n-1$,  we  have also  $\sigma^n y\in \xi_0(\sigma^n x)$.  Therefore $y\in  \sigma^{-n}(\xi_i(\sigma^n x))$. This proves  $\xi_i(x)\cap \P_{0}^{n-1} (x)\subset\sigma^{-n}(\xi_i(\sigma^n x))$.

Next we prove (2). Let $x\in \Sigma'$,   $i\in \{0,\ldots, s-1\}$ and $\epsilon>0$. By \eqref{e-e4.3} and \eqref{e-e4.1}, there exists  $n_0=n_0(x)$ such that  for any $n\geq n_0$,
\begin{equation}
\label{e-l1}
 \max_{v\in V_{\sigma^nx}^i,\; \|v\|=1} \|M_{x_0\ldots x_{n-1}}v\|\leq \left\{
 \begin{array}{ll}
  e^{n(\lambda_{i+1}(x)+{\epsilon})} & \mbox{ if } \lambda_{i+1}(x)\neq -\infty\\
  e^{-2n/\epsilon} & \mbox{ if } \lambda_{i+1}(x)= -\infty
  \end{array}
  \right.
\end{equation}
  and
\begin{equation}
\label{e-pisigma}
\| \pi \sigma^n x\|\leq \frac{1}{2}e^{n\epsilon/2}.
\end{equation}
Now let $n\geq n_0$ and $y\in Q_{n,\epsilon} \cap \xi_i(x)\cap \P_{0}^{n-1} (x)$. Then $\|\pi\sigma^n y\|\leq (1/2) e^{n\epsilon/2}$, $y^{-}=x^{-}$,  $\pi y-\pi x\in V_x^i$ and furthermore by (1), $\pi\sigma^ny-\pi\sigma^nx\in V_{\sigma^nx}^i$. By \eqref{e-pi}-\eqref{e-pisigma},
\begin{align*}
{\|\pi y-\pi x\|} &= \|M_{x_0\ldots x_{n-1}}(\pi\sigma^n y-\pi \sigma^n x)\|\\
&\leq \left( \max_{v\in V_{\sigma^nx}^i,\; \|v\|=1} \|M_{x_0\ldots x_{n-1}}v\| \right) \|\pi \sigma^n y-\pi \sigma^n x\|\\
& \leq \left\{
 \begin{array}{ll}
  e^{n(\lambda_{i+1}(x)+2{\epsilon})} & \mbox{ if } \lambda_{i+1}(x)\neq -\infty\\
  e^{-n/\epsilon} & \mbox{ if } \lambda_{i+1}(x)= -\infty
  \end{array}
  \right.
  .
\end{align*}
This proves  \eqref{e-l2}. Moreover, since $V_x^0=\R^d$,  the above argument for the case $i=0$ actually proves \eqref{e-l3}.
\end{proof}

Recall that for a measurable partition $\eta$ of $\Sigma'$,  $\{m_x^\eta\}$ stands for  the canonical system of conditional measures associated with $\eta$ (cf. Section~\ref{Sect-con}).

\begin{lem}
\label{lem-3.6}
Let $i\in \{0,1,\ldots, s\}$. Then for $m$-a.e.~$x\in \Sigma'$, the following hold.

\begin{itemize}
\item[(1)]
$
m_x^{\sigma^{-n}\xi_i}(A) =m_{\sigma^n x}^{\xi_i}(\sigma^n A)$  for any $n\in \N$ and  measurable $A\subset \Sigma'$.
\medskip
\item[(2)] $\displaystyle
m_x^{\sigma^{-n} \xi_i}(A)=
\frac {m_x^{\xi_i}(A\cap \P_0^{n-1}(x))} {m_x^{\xi_i}(\P_0^{n-1}(x))}
$ for any $n\in \N$ and measurable  $A\subset \Sigma'$.
\medskip
\item[(3)] $\displaystyle
\frac {m_x^{\xi_i} (\sigma^{-n} A\cap \P_0^{n-1}(x))} {m_{\sigma^n x}^{\xi_i}(A)}=m_x^{\xi_i}(\P_0^{n-1}(x))
$ for any $n\in \N$ and measurable  $A\subset \Sigma'$.

\end{itemize}
\end{lem}
\begin{proof}
All the  results follow from the $\sigma$-invariance of $m$ and the uniqueness of conditional measures.  For the reader's convenience, we include below the detailed arguments.

To see (1), fix $n\in \N$ and  define a family of probability measures $\{\mu_x\}_{x\in \Sigma'}$ such that $\mu_x$ is supported on $(\sigma^{-n}\xi_i)(x)=\sigma^{-n}(\xi_i(\sigma^n x))$ and satisfies
$$
\mu_x(A)=m_{\sigma^n x}^{\xi_i}(\sigma^n A) \quad \mbox{ for any measurable }  A\subset \Sigma'.
$$
Then by Theorem \ref{thm-2.1}, for every measurable $A\subset  \Sigma'$ and $m$-a.e.~$x$,
\begin{align*}
\mu_x(A)&=\E_m(\chi_{\sigma^n A}|\widehat{\xi_i})(\sigma^n x)  \\
&=\E_m(\chi_{\sigma^n A}\circ \sigma^n |\sigma^{-n} \widehat{\xi_i})( x) \qquad \mbox{ (by  Lemma \ref{lem-par}(i))} \\
&=\E_m(\chi_{ A}  |\sigma^{-n} \widehat{\xi_i})( x).
\end{align*}
It follows that $x\mapsto \mu_x(A)$ is $\sigma^{-n}\widehat{\xi_i}$-measurable and $m(A)=\int \mu_x(A) dm(x)$.  Therefore, $\{\mu_x\}$ is a canonical system of conditional measures associated with $\sigma^{-n}\xi_i$. By the uniqueness of conditional measures, we have
$\mu_x=m^{\sigma^{-n}\xi_i}_x$ for $m$-a.e.~$x$. This proves (1).

To see (2), let $n\in \N$ and notice that $\sigma^{-n} \xi_i=\xi_i\vee \P_{0}^{n-1}$ by Lemma \ref{lem-3.4}(1). Similar to the proof of (1), we  define a family of probability measures $\{\nu_x\}_{x\in \Sigma'}$ such that $\nu_x$ is supported on $(\sigma^{-n}\xi_i)(x)=\xi_i(x)\cap \P_0^{n-1}(x)$ and satisfies
$$
\nu_x(A)=
\frac {m_x^{\xi_i}(A\cap \P_0^{n-1}(x))} {m_x^{\xi_i}(\P_0^{n-1}(x))}
 \quad \mbox{ for any measurable }  A\subset \Sigma'.
$$
Then by Theorem \ref{thm-2.1}, for every measurable $A\subset  \Sigma'$ and $m$-a.e.~$x$,
\begin{align}
\label{e-chi}
\nu_x(A)&=\sum_{B\in \P_0^{n-1}} \chi_B(x)\cdot h_B(x),
\end{align}
where $h_B:=\E_m(\chi_{A\cap B}|\widehat{\xi_i})/\E_m(\chi_{ B}|\widehat{\xi_i})$. Since $h_B$ is $\widehat{\xi_i}$-measurable, the mapping $x\mapsto \nu_x(A)$ is  $\widehat{\xi_i}\vee \widehat{\P_0^{n-1}}$-measurable (i.e. $\sigma^{-n}\widehat{\xi_i}$-measurable).
Moreover  by \eqref{e-chi},
\begin{align*}
\int \nu_x(A) \;dm(x)&= \sum_{B\in \P_0^{n-1}}\int  \chi_B h_B \;dm\\
&=\sum_{B\in \P_0^{n-1}}\int  \E_m(\chi_B h_B|\widehat{\xi_i}) \;dm\\
&=\sum_{B\in \P_0^{n-1}}\int  \E_m(\chi_B|\widehat{\xi_i} ) h_B \;dm\\
&=\sum_{B\in \P_0^{n-1}}\int \E_m( \chi_{A\cap B}|\widehat{\xi_i})  \;dm\\
&=\sum_{B\in \P_0^{n-1}}m(A\cap B)=m(A).
\end{align*}
Hence the family $\{\nu_x\}$ is a  canonical system of conditional measures associated with $\sigma^{-n}\xi_i$,  and so (2) follows by the uniqueness of conditional measures.

Finally we prove (3). By (1), we have $$\displaystyle
m_{\sigma^n x}^{\xi_i}(A)=
m_{\sigma^n x}^{\xi_i}(\sigma^n(\sigma^{-n} A))=m_{x}^{\sigma^{-n} \xi_i}(\sigma^{-n} A).$$
Applying (2) to   $\sigma^{-n} A$ (instead of $A$)  yields that
$$
m_{\sigma^n x}^{\xi_i}(A)=m_{x}^{\sigma^{-n} \xi_i}(\sigma^{-n} A)=\frac {m_x^{\xi_i} (\sigma^{-n} A\cap \P_0^{n-1}(x))} {m_x^{\xi_i}(\P_0^{n-1}(x))},
$$
which implies (3).
\end{proof}

Now for $i\in \{0,1,\ldots, s\}$,  define
\begin{equation}
\label{e-hi}
h_i(x)=\E_m(f_i|\I)(x),   \quad x\in \Sigma',
\end{equation}
where  $f_i:={\bf I}_m(\P|\widehat{\xi_i})$ and $\I=\{A\in \B(\Sigma'):\; \sigma^{-1}A=A\}$.  Clearly  $f_i\geq 0$ a.e. By Lemma \ref{lem-par}(v),  $f_i\in L^1$. It follows that  $h_i\geq 0$ a.e.~and  $h_i\in L^1$.

\begin{lem}
\label{lem-3.5}
Let $i\in \{0,1,\ldots, s\}$. Then for  $m$-a.e.~$x\in \Sigma'$,
\begin{equation}
\label{e-h1}
\begin{split}
&\log m_x^{\xi_i}(\P_{0}^{n-1}(x))=-\sum_{j=0}^{n-1}{\bf I}_m(\P |\widehat{\xi}_i)(\sigma^jx)\quad\mbox{ and }\\
 &-\lim_{n\to \infty} \frac{1}{n}\log m_x^{\xi_i}(\P_{0}^{n-1}(x))=h_i(x).
 \end{split}
\end{equation}
Furthermore,
\begin{equation}
\label{e-h2}
-\lim_{n\to \infty} \frac{1}{n}\log m(\P_{0}^{n-1}(x))=  h_0(x) \quad  \mbox{ for  $m$-a.e. }x\in \Sigma'.
\end{equation}

\end{lem}

\begin{proof}
Let $i\in \{0,1,\ldots, s\}$.  By Theorem \ref{thm-2.1},
$$\log  m_x^{\xi_i}(\P_{0}^{n-1}(x))=\sum_{A\in \P_{0}^{n-1}}\chi_A(x) \log m_x^{\xi_i}(A)=\sum_{A\in \P_{0}^{n-1}}\chi_A(x) \log \E_m(\chi_A|\widehat{\xi_i})(x)$$
and hence $-\log m_x^{\xi_i}(\P_{0}^{n-1}(x))={\bf I}_m(\P_0^{n-1}|\widehat{\xi_i})(x)$ for $m$-a.e.~$x$. By Lemma \ref{lem-par},
\begin{align*}
{\bf I}_m(\P_0^{n-1}|\widehat{\xi_i})&={\bf I}_m(\P|\widehat{\xi_i})+{\bf I}_m\left(\bigvee_{j=1}^{n-1}\sigma^{-j} \P\big|\widehat{\xi}_i\vee \widehat{\P}\right)\\
&={\bf I}_m(\P|\widehat{\xi}_i)+{\bf I}_m\left(\bigvee_{j=1}^{n-1}\sigma^{-j} \P\big|\sigma^{-1} \widehat{\xi}_i \right)\qquad (\mbox{by Lemma \ref{lem-3.4}(1)})\\
&={\bf I}_m(\P|\widehat{\xi}_i)+{\bf I}_m(\P_{0}^{n-2} |\widehat{\xi}_i)\circ \sigma \qquad \qquad (\mbox{by Lemma \ref{lem-par}(ii)}).
\end{align*}
Therefore by induction we have
\begin{equation}
\label{e-xi}
{\bf I}_m(\P_0^{n-1}|\widehat{\xi}_i)=\sum_{j=0}^{n-1}{\bf I}_m(\P |\widehat{\xi}_i)\circ \sigma^j.
\end{equation}
 Now \eqref{e-h1} follows from \eqref{e-xi} and  the Birkhoff ergodic theorem.

To see \eqref{e-h2}, applying the Shannon-McMillian-Breiman theorem (see e.g. \cite[p.~39] {Parry1981})   to the transformations $\sigma$ and $\sigma^{-1}$ respectively, we have the following convergences (pointwise and in $L^1$):
\begin{equation}
\label{e-h4}
\begin{split}
-\lim_{n\to +\infty} \frac{1}{n}\log m(\P_{0}^{n-1}(x))&=\E_m(g_1|\I)(x),\\
-\lim_{n\to +\infty} \frac{1}{n}\log m(\P_{-(n-1)}^{0}(x))&=\E_m(g_2|\I)(x),
\end{split}
\end{equation}
where  $g_1:={\bf I}_m(\P|\bigvee_{j=1}^\infty \sigma^{-j}\widehat\P)$,  $g_2:={\bf I}_m(\P|\bigvee_{j=1}^\infty \sigma^{j}\widehat\P)$.
Noticing that $\widehat \xi_0=\bigvee_{j=1}^\infty \sigma^{j}\widehat\P$, we have $g_2={\bf I}_m(\P|\widehat \xi_0)=f_0$ and so
$\E_m(g_2|\I)=h_0$.  To prove \eqref{e-h2}, by \eqref{e-h4} it suffices to show that \begin{equation}
\label{e-h3}
\E_m(g_1|\I)(x)=\E_m(g_2|\I)(x)\quad \mbox{  for $m$-a.e.~$x$}.
\end{equation}

To see \eqref{e-h3}  first observe that for $x\in \Sigma'$,  $\P_{-(n-1)}^{0}(\sigma^n x)=\sigma^n(\P_0^{n-1}( x))$ and hence
$m(\P_{-(n-1)}^{0}(\sigma^n x))=m(\sigma^n(\P_0^{n-1}( x)))=m(\P_0^{n-1}( x))$.
For any $B\in \I$,  we have
\begin{align*}
\int_B & \log m(\P_{-(n-1)}^{0}(x)) \; dm(x)\\
&=\int \chi_B(x) \log m(\P_{-(n-1)}^{0}(x)) \; dm(x)\\
&=\int \chi_B(\sigma^n x) \log m(\P_{-(n-1)}^{0}(\sigma^n x)) \; dm(x)  \qquad(\mbox{by the $\sigma$-invariance of $m$}) \\
&=\int \chi_B(\sigma^nx) \log m(\P_0^{n-1}(x)) \; dm(x)\\
&=\int \chi_B(x) \log m(\P_0^{n-1}(x)) \; dm(x) \qquad\qquad(\mbox{by $\chi_B=\chi_B\circ \sigma^n$ as $B\in {\mathcal I}$}) \\
&=\int_B  \log m(\P_0^{n-1}(x)) \; dm(x).
\end{align*}
Dividing both sides by $n$, letting $n\to \infty$ and applying \eqref{e-h4}, we have
$$\int_B \E_m(g_1|\I)\; dm= \int_B \E_m(g_2|\I)\; dm\quad \mbox{ for all } B\in \I.$$
Therefore  $\E_m(g_1|\I)= \E_m(g_2|\I)$ almost everywhere.
This completes the proof of the lemma.
\end{proof}

Below we give an interesting corollary of Lemma \ref{lem-3.5}, although we will not use it in the rest part of the paper.

\begin{cor}
Let $i\in \{0,1,\ldots, s\}$. Then $h_i=0$ a.e.~if and only if $m_x^{\xi_i}=\delta_x$  (i.e.~$m_x^{\xi_i}(\{x\})=1$) for $m$-a.e.~$x\in \Sigma'$.
\end{cor}
\begin{proof}
By Lemma \ref{lem-par}(v),  $f_i:={\bf I}_m(\P|\widehat{\xi_i})\geq 0$ a.e.~ and $f_i\in L^1$. Hence by \eqref{e-hi}, $h_i=0$ a.e.~if and only if $f_i=0$ a.e. However according to the first equality in \eqref{e-h1},  the condition $f_i=0$ a.e.~ implies that for $m$-a.e.~$x$, $m_x^{\xi_i}(\mathcal P_0^{n-1}(x))=1$ for every $n\geq 1$ and hence
$$m_x^{\xi_i}(\{x\})=m_x^{\xi_i}\left(\xi_0(x)\cap \P_0^{\infty}(x)\right)=m_x^{\xi_i}\left( \P_0^{\infty}(x)\right)=1,$$
using the fact that  $m_x^{\xi_i}$ is supported on $\xi_i(x)\subset \xi_0(x)$.
Conversely, by the first equality in \eqref{e-h1} (applied to $n=1$), we obtain that $f_i(x)=-\log m_x^{\xi_i}(\P(x))$; hence the condition $$m_x^{\xi_i}(\{x\})=1 \mbox{ a.e.}$$  implies that $f_i=0$ a.e. This completes the proof of the corollary.
\end{proof}

We end the section by the following.

\begin{lem}
\label{lem-4.6}
Let $\epsilon>0$ and define $Q_{n,\epsilon}$ as in \eqref{e-qn} for $n\in \N$. Then for $m$-a.e.~$x\in \Sigma'$,
$$
\lim_{n\to \infty} \frac{m_x^{\xi_i}(Q_{n,\epsilon}\cap \P_0^{n-1}(x))}{m_x^{\xi_i}(\P_0^{n-1}(x))}=1 \qquad (i=0,1,\ldots, s)
$$
and
$$
\lim_{n\to \infty} \frac{m(Q_{n,\epsilon}\cap \P_0^{n-1}(x))}{m(\P_0^{n-1}(x))}=1.
$$
\end{lem}
\begin{proof}
The equalities follow from the   Lebesgue density lemma for Polish ultrametric spaces (see, e.g.~ \cite[Proposition 2.10]{Miller2008}) and the facts that the sequence  $(Q_{n,\epsilon})$ of sets is  monotone increasing as $n$ increases, and  $\bigcup_n Q_{n,\epsilon}$ is of  full $m$-measure by Proposition \ref{pro-3.1}(iii).
\end{proof}

\section{Transverse dimensions}
\label{S-5}

In this section, we prove an inequality for the transverse dimensions of the conditional measures that we constructed in Section~\ref{S-4}.

Recall that  ${\mathcal S}$ is an affine IFS on $\R^d$ of the form \eqref{e-form},   average contracting with respect to some $m\in \M_\sigma(\Sigma)$. Let  $\pi$ be the associated coding map.   Let $M:\;\Sigma\to {\rm Mat}_d(\R)$ be the matrix cocycle given by $M(x)=M_{x_{-1}}$, and  $\{(\lambda_i(x), k_i(x))\}_{1\leq i\leq s(x),\; x\in \Sigma'}$  the Lyapunov spectrum for  $M$ with respect to the transformation $\sigma^{-1}$.   Suppose that
\eqref{e-assump} holds, i.e. there exist  $s, k_1,\ldots, k_s$ so that $s(x)=s$, $k_i(x)=k_i$ ($i=1,\ldots, s$) for $m$-a.e.~$x\in \Sigma'$.   Let $\oplus_{i=1}^s E_x^i$ be the Oseledets splitting of $\R^d$, and $\{0\}=V_x^s\subset \cdots \subset V^0_x=\R^d$ the associated filtration.

Let $\xi_0,\xi_1,\ldots, \xi_s$ be the measurable partitions of $\Sigma'$ that we constructed in Section~\ref{S-4}. For $x\in \Sigma'$ and $r>0$, set
\begin{equation}
\label{gamma}
{\Gamma_i}(x,r)=\{y\in \Sigma':\; \mbox{dist}(\pi y+V_x^i, \; \pi x+V_x^i)\leq r\}, \quad i=1,\ldots, s
\end{equation}
and  define
    $$
\vartheta_{i-1}(x)=\liminf_{r\to 0}\frac{\log
m^{\xi_{i-1}}_x({\Gamma_{i}}(x,r))}{\log r}, \quad i=1,\ldots, s.
$$
We call $\vartheta_0$,\ldots, $\vartheta_{s-1}$ the {\it transverse dimensions} of $m$. Intuitively we may view $\vartheta_i(x)$ as the dimension of $m$ along the direction $E_x^{i+1}$.

The main result of this section is the following, which plays a key role in the proof of Theorem \ref{thm-1.0}.
\begin{pro}
\label{pro-4.1}
For $m$-a.e.~$x\in \Sigma'$,
$$
\vartheta_{i-1}(x)\geq \frac{h_{i}(x)-h_{i-1}(x)}{\lambda_{i}(x)}, \quad i=1,\ldots, s,
$$
where $h_i$ are defined as in \eqref{e-hi}.
\end{pro}

This result can be viewed as an analogue of Proposition 11.2 in \cite{LedrappierYoung1985}. A stronger version of the result, with the inequality being replaced by the equality, was  proved earlier in \cite[Theorem 6.2]{FengHu2009}, \cite[Theorem 3.3]{Barany2015}, and \cite[Propositions 5.3 and 7.3]{BaranyKaenmaki2015} under various additional assumptions.

The proof of Proposition \ref{pro-4.1} is quite long and delicate. Besides extending some ideas from the previous works \cite{LedrappierYoung1985, FengHu2009, BaranyKaenmaki2015}, we need to employ certain new strategy as well.

We first introduce some notation and give several lemmas.

For $x\in \Sigma'$, $i\in \{1,\ldots, s\}$ and $r>0$, set
$$
 B^i_x(r)=\{v\in E^i_x:\; \|v\|\leq r\}
$$
and
$$
T_i(x,r)=\{y\in \xi_{i-1}(x):\; \pi y-\pi x\in V_x^i\oplus B^i_x(r)\}.
$$

\begin{lem}
\label{lem-5.2}
Let $x\in \Sigma'$, $i\in \{1,\ldots, s\}$, $n\in \N$ and  $r>0$. For $$0\leq a\leq \min_{v\in E_{\sigma^n x}^i,\; \|v\|=1} \|M_{x_0\cdots x_{n-1}}v\|,$$ we have
\begin{equation}
\label{e-i2}
T_i(x, a r)\cap \P_0^{n-1}(x)\subset \sigma^{-n} T_{i}(\sigma^nx, r).
\end{equation}
\end{lem}
\begin{proof}
 Let $y\in T_i(x, ar)\cap \P^{n-1}_0(x)$.
By definition,  \begin{align}
& y \in \xi_{i-1}(x)\cap \P^{n-1}_0(x) \quad \mbox{ and } \label{e-i3}\\
& \pi y - \pi x\in V_{x}^i\oplus B^i_{x}(ar). \label{e-i4}
 \end{align}
By \eqref{e-i3} and Lemma \ref{lem-3.4}(1),  $y\in \sigma^{-n}(\xi_{i-1}(\sigma^n x))$. Moreover since $y\in \P_{0}^{n-1}(x)$, by \eqref{e-pi},
\begin{equation}
\label{e-i5}\pi y-\pi x =  M_{x_0\ldots x_{n-1}}(\pi \sigma^n y-\pi\sigma^n x).
\end{equation}
 Since $y\in \xi_{i-1}(x)$,  by definition $\pi y-\pi x\in V_{x}^{i-1}=V_x^i\oplus E_x^i$. Applying \eqref{e-v1} to $V_x^{i-1}$ yields
 $$\limsup_{k\to \infty}\frac{1}{k}\log \|M_{x_{-k}\ldots x_{-1}}(\pi y-\pi x)\|\leq \lambda_i(x)=\lambda_i(\sigma^nx),$$
 where the last equality follows from  Theorem \ref{thm-2}(vii).
Hence by \eqref{e-i5},
 $$\limsup_{k\to \infty}\frac{1}{n+k}\log \|M_{x_{-k}\ldots x_{-1}x_0\ldots x_{n-1}}(\pi \sigma^n y-\pi \sigma^n x)\|\leq \lambda_i(\sigma^nx).$$
Applying  \eqref{e-v1} to $V_{\sigma^nx}^{i-1}$  gives $\pi \sigma^n y-\pi\sigma^n x\in V_{\sigma^n x}^{i-1}=V_{\sigma^nx}^i\oplus E_{\sigma^n x}^i$. Write
\begin{align*}
\pi y-\pi x &=v_1+ w_1 \quad \mbox{ with $v_1\in V_{x}^i$ and $w_1\in E_x^i$},\\
\pi \sigma^n y-\pi \sigma^n x& =v_2+ w_2 \quad \mbox{ with $v_2\in V_{\sigma^n x}^i$ and $w_2\in E_{\sigma^nx}^i$}.
\end{align*}
By \eqref{e-i4}, $w_1\in B^i_{x}(ar)$ and hence $\|w_1\|\leq ar$. Since $M_{x_0\cdots x_{n-1}}V_{\sigma^nx}^i\subset V_x^i$ and $M_{x_0\cdots x_{n-1}}E_{\sigma^nx}^i\subset E_x^i$, by \eqref{e-i5} we see that $w_1=M_{x_0\ldots x_{n-1}}w_2$ and so
$$
ar\geq  \|w_1\|=\|M_{x_0\ldots x_{n-1}}w_2\|\geq a\|w_2\|.
$$
It follows that $\|w_2\|\leq r$.
 Hence
$\pi \sigma^n y-\pi\sigma^n x\in V_{\sigma^nx}^i\oplus B_{\sigma^nx}^i(r)$. This together with $y\in \sigma^{-n}(\xi_{i-1}(\sigma^n x))$ yields that $y\in \sigma^{-n}T_i(\sigma^n x, r)$. Therefore
$$T_i(x, ar)\cap \P_0^{n-1}(x)\subset \sigma^{-n}T_i(\sigma^n x, r)$$
 and we are done.
\end{proof}

Let $\theta(x)$ denote the smallest angle between the Oseledets subspaces, i.e.
$$
\theta(x)=\min_{I\cap J=\emptyset} \measuredangle \Big( \oplus_{i\in I}E^i_{x},\; \oplus_{j\in J}E^j_{x}\Big).
$$
 We have the following.
\begin{lem}
\label{lem-4.2} For $x\in \Sigma'$,   $i\in \{1,\ldots, s\}$ and  $r>0$,
$$T_{i}(x,r)\subset \xi_{i-1}(x)\cap {\Gamma_i}(x,r)\subset T_i(x, r/\sin\theta(x)).$$
\end{lem}
\begin{proof}
We first prove that $T_{i}(x,r)\subset \xi_{i-1}(x)\cap {\Gamma_i}(x,r)$. Let $y\in T_i(x,r)$. Then by definition, $y\in \xi_{i-1}(x)$  and $\pi y-\pi x=v+w$ for some
$v\in V^i_x$, $w\in E^i_x$ with $\|w\|\leq r$, which implies that $$\mbox{dist}(\pi y+V_x^i, \; \pi x+V_x^i)\leq \|w\|\leq r.$$
Hence  $y\in \xi_{i-1}(x)\cap {\Gamma_i}(x,r)$. This proves the relation  $T_{i}(x,r)\subset \xi_{i-1}(x)\cap {\Gamma_i}(x,r)$.

Next we prove that $\xi_{i-1}(x)\cap {\Gamma_i}(x,r)\subset T_i(x, r/\sin(\theta(x)))$. Let $U_x^i:=V^{i-1}_x\ominus V^i_x$ denote the orthogonal complement of $V_x^i$ in $V^{i-1}_x$.
Let
$z\in \xi_{i-1}(x)\cap {\Gamma_i}(x,r)$. Then $\pi z-\pi x\in V_x^{i-1}$ and $\mbox{dist}(\pi z+V_x^i, \pi x +V_x^i)\leq r$. Hence
$\pi z-\pi x=v+u$ for some $v\in V_x^i$ and $u\in U_x^i$ with $\|u\|\leq r$. Since $v+u\in V_x^{i-1}=V_x^i\oplus E_x^i$,
$v+u=v_1+w_1$ for some  $v_1\in V^i_x$ and $w_1\in E_x^i$. Notice that $w_1= (v-v_1)+u$ with $u\perp (v-v_1)$. We have
$$
\|w_1\|=\frac{\|u\|}{\sin \measuredangle(w_1, v-v_1)}\leq \frac{\|u\|}{\sin \theta(x)}\leq \frac{r}{\sin \theta(x)}.
$$
  Thus $\pi z-\pi x=v_1+w_1$, where $v_1\in V_x^i$ and $w_1\in E_x^i$ with $\|w_1\|\leq r/\sin \theta(x)$.  Therefore,  $z\in T_i(x, r/\sin \theta(x))$ and we are done.
\end{proof}

Now we turn back to the proof of Proposition \ref{pro-4.1}. Clearly, to prove the proposition it is sufficient to show that for any $\epsilon>0$, there exists $F(\epsilon)\subset \Sigma'$  so that
\begin{equation}
\label{e-import}
\vartheta_{i-1}(x)\geq  \frac{h_{i-1}(x)-h_{i}(x)}{-\lambda_{i}(x)+\epsilon}\quad  \mbox{ for $m$-a.e.~$x\in F(\epsilon)$ and }i\in \{1,\ldots, s\}.
\end{equation}
and  $\lim_{\epsilon\to 0}m(F(\epsilon))=1$.

 Here and afterwards in this section, we may assume that $\lambda_s\neq -\infty$ a.e., since Proposition \ref{pro-4.1}  holds automatically when $i=s$ and $\lambda_s(x)=-\infty$.

 We  first construct $F(\epsilon)$ for $\epsilon>0$. Set
\begin{equation}
\label{e-t1}
F_0(\epsilon):=\{x\in \Sigma':\;\sin \theta(x)>\epsilon\}.
\end{equation}
 By \eqref{e-e?},  there exist a large integer $N(\epsilon)$ and a Borel set $F(\epsilon)\subset F_0(\epsilon)$ with $m(F(\epsilon))>(1-\epsilon)m(F_0(\epsilon))$ so that for $i\in \{1,\ldots, s\}$,
\begin{equation}\label{e-t0}
 \|M_{x_0\cdots x_{n-1}}v\| \geq \epsilon^{-1} e^{n(\lambda_i(x)-\epsilon)}\|v\|\end{equation}
for $x\in F(\epsilon)$,  $n\geq N(\epsilon)$ and $v\in E_{\sigma^nx}^i$. Clearly, $m(F(\epsilon))\to 1$ as $\epsilon\to 0$.

In the remaining part of this section we prove \eqref{e-import} for the constructed $F(\cdot)$. From now on, we fix $\epsilon>0$ and write simply $F=F(\epsilon)$ and $N=N(\epsilon)$.

Let $\sigma_F: F\to F$ be the transformation induced by $\sigma^N$ on the set $F$ (cf. Section~\ref{Sect-2.3}). That is, $\sigma_F(x)=\sigma^{Nr_F(x)}(x)$, where
$$
r_F(x):=\inf\{ n\geq 1: \; \sigma^{nN}x\in F\}.
$$
The map $\sigma_F$ is well-defined on $F$ up to a set of zero $m$-measure.  Let $m_F$ be the Borel probability measure on $F$ defined by
$$m_F(D)=\frac{m(F\cap D)}{m(F)} \quad  \mbox{ for any Borel set }D\subset F.$$
Recall that $m_F$ is $\sigma_F$-invariant.

For $x\in F$,  set
\begin{equation}
\label{e-rho}
\begin{split}
\ell (x)&=N r_F(x)\qquad \mbox{ and } \\
 {\rho}(i, x)&=e^{\ell (x) (\lambda_i(x)-\epsilon)}, \qquad i=1,\ldots, s.
\end{split}
\end{equation}
Then we have

\begin{lem}
For $x\in F$, $i\in \{1,\ldots, s\}$ and $r>0$,
\begin{align}
\xi_{i-1}(x)\cap {\Gamma_{i}}\left(x,\;  {\rho}(i, x) r\right)&\cap \P_0^{\ell (x)-1}(x) \subset  \sigma^{-\ell(x)} \left({\Gamma_i}\left(\sigma_F x,r\right)
\cap \xi_{i-1}(\sigma_Fx)\right).
\label{e-t2}
\end{align}
\end{lem}
\begin{proof}
Fix $x\in F$, $i\in \{1,\ldots, s\}$ and $r>0$. Set $a=\epsilon^{-1}{\rho}(i, x)$.  Since $\ell(x)=Nr_F(x)\geq N$, by \eqref{e-t0},
\begin{equation}
\label{e-ttd}
a=\epsilon^{-1}e^{\ell (x)(\lambda_i(x)-\epsilon)}\leq \inf\{\|M^{\ell (x)}(x)v\|:\; v\in E_{\sigma^{\ell(x)}x}^i,\; \|v\|=1\},
\end{equation}
where $M^n(x):=M_{x_0\cdots x_{n-1}}$.
Observe that
\begin{equation*}
\begin{split}
\xi_{i-1}&(x)\cap {\Gamma_{i}}\left(x,\;  {\rho}(i, x) r\right)\\
  & \subset T_i\left(x, {\rho}(i, x) r/\sin\theta(x)\right) \qquad (\mbox{by Lemma \ref{lem-4.2}})\\
& \subset T_i\left(x, \epsilon^{-1}{\rho}(i, x) r\right)\qquad\quad (\mbox{since $\sin \theta(x)\geq \epsilon$})\\
&= T_i\left(x, a r\right).\\
\end{split}
\end{equation*}
Hence
\begin{equation*}
\begin{split}
\xi_{i-1}&(x)\cap {\Gamma_{i}}\left(x,\;  {\rho}(i, x) r\right) \cap \P_0^{\ell (x)-1}(x)\\
  & \subset  T_i\left(x, a r\right)\cap \P_0^{\ell (x)-1}(x)\\
& \subset \sigma^{-\ell (x)} T_i\left(\sigma^{\ell (x)} x,  r\right)\qquad\qquad (\mbox{by \eqref{e-ttd} and Lemma \ref{lem-5.2}})\\
& =  \sigma^{-\ell (x)} T_i\left(\sigma_F x,  r\right)  \\
&\subset \sigma^{-\ell(x)} ({\Gamma_i}(\sigma_F x,  r)\cap \xi_{i-1}(\sigma_Fx)) \qquad (\mbox{by Lemma \ref{lem-4.2}}).
\end{split}
\end{equation*}
This completes the proof of the lemma.
\end{proof}

Now write
\begin{equation}
\label{e-l32}
F_n:=\{x\in F:\; r_F(x)=n\},\qquad n=1,2,\ldots.
\end{equation}
Recall that $\{m_x^{\xi_i}\}$ is the canonical system of conditional measures associated with $\xi_i$, $i=0,\ldots, s$. The following result is an induced version of Lemma \ref{lem-2.4}.


 \begin{pro}
 \label{pro-4.5}
Let $i\in \{1,\ldots, s\}$. Then for $m$-a.e.~$x\in F$,
\begin{equation}\label{e-4.13}
\begin{split}
\lim_{r\to 0}\log \frac
{m^{\xi_{i-1}}_x\left({\Gamma_i}(x,r)\cap \P_{0}^{\ell (x)-1}(x)\right)}
{m^{\xi_{i-1}}_x\left({\Gamma_i}(x,r)\right)}
&=-\sum_{k=1}^\infty \chi_{F_k}(x)\sum_{j=0}^{kN-1} {\bf I}_m(\P|\widehat{\xi_i})(\sigma^j x).
\end{split}
\end{equation}
Furthermore, set
\begin{equation}\label{e-2.5''}
g(x)=-\inf_{r>0}
\log \frac
{m^{\xi_{i-1}}_x\left({\Gamma_i}(x,r)\cap \P_{0}^{\ell (x)-1}(x)\right)}
{m^{\xi_{i-1}}_x\left({\Gamma_i}(x,r)\right)}.
\end{equation}
 Then $g\geq 0$ and $g\in L^1(F,{\B}|_F,m_F)$.
\end{pro}
\begin{proof}
Fix $i\in \{1,\ldots, s\}$. Write $d_i=\sum_{j=i+1}^s{k_j}$.
Define $\phi_i: \Sigma'\to Y_i:=G(d, d_i)\times \R^d$ by
$$
\phi_i(x)=\left(V_x^i, \; P_{(V_x^i)^\perp}(\pi x)\right).
$$
Then $\phi_i$ is measurable.  Moreover,
\begin{equation}
\xi_i(x)=\{y\in \xi_0(x):\; \phi_i(y)=\phi_i(x)\},\quad x\in \Sigma'.
\end{equation}

 Endow  $Y_i$ with the following  product metric $\rho_i$:
$$
\rho_i\left((V, a), (W, b)\right)=\max\{\|P_V-P_W\|, \|a-b\|\}.
$$
It is not hard to see that  $Y_i$ is a Besicovitch space. For $x\in \Sigma'$ and $r>0$, set
$$
B^{\phi_i}(x,r):=\{y\in \Sigma':\; \rho_i(\phi_iy, \phi_i x)\leq r\}.
$$
Then by definition,
\begin{equation}
\label{e-phigamma}
\xi_0(x)\cap B^{\phi_i}(x,r)=\xi_0(x)\cap \Gamma_i(x,r),\qquad x\in \Sigma',\; r>0.
\end{equation}
Hence for $x\in F$ and $r>0$,
\begin{equation}
\label{e-4.14}
\begin{split}
\log \frac
{m^{\xi_{i-1}}_x\left({\Gamma_i}(x,r)\cap \P_{0}^{\ell (x)-1}(x)\right)}
{m^{\xi_{i-1}}_x\left({\Gamma_i}(x,r)\right)}&=\log \frac
{m^{\xi_{i-1}}_x\left(\xi_0(x)\cap {\Gamma_i}(x,r)\cap \P_{0}^{\ell (x)-1}(x)\right)}
{m^{\xi_{i-1}}_x\left(\xi_0(x)\cap {\Gamma_i}(x,r)\right)}\\
&=\log \frac
{m^{\xi_{i-1}}_x\left(\xi_0(x)\cap B^{\phi_i}(x,r)\cap \P_{0}^{\ell (x)-1}(x)\right)}
{m^{\xi_{i-1}}_x\left(\xi_0(x)\cap B^{\phi_i}(x,r)\right)}\\
&=\log \frac
{m^{\xi_{i-1}}_x\left(B^{\phi_i}(x,r)\cap \P_{0}^{\ell (x)-1}(x)\right)}
{m^{\xi_{i-1}}_x\left(B^{\phi_i}(x,r)\right)}\\
&= \sum_{k=1}^\infty\sum_{A\in \P_{0}^{kN-1}} \chi_{ F_k\cap A}(x) \log \frac{m_x^{\xi_{i-1}}\left(B^{\phi_i}(x,r)\cap A\right)}{m_x^{\xi_{i-1}}\left(B^{\phi_i}(x,r)\right)}.
\end{split}
\end{equation}

By \eqref{e-4.14} and applying Lemma \ref{lem-2.4}(1) to $\phi_i: \Sigma'\to Y_i$, we have for $m$-a.e.~$x\in F$,
\begin{equation*}
\begin{split}
\lim_{r\to 0} & \log \frac
{m^{\xi_{i-1}}_x\left({\Gamma_i}(x,r)\cap \P_{0}^{\ell (x)-1}(x)\right)}
{m^{\xi_{i-1}}_x\left({\Gamma_i}(x,r)\right)}\\
& =\sum_{k=1}^\infty\sum_{A\in \P_{0}^{kN-1}} \chi_{A\cap F_k}(x)\log  \E_m\left(\chi_A|\widehat{\xi_{i-1}} \vee \phi_{i}^{-1}\B(Y_i)\right)(x)\\
& =\sum_{k=1}^\infty\sum_{A\in \P_{0}^{kN-1}} \chi_{A\cap F_k}(x)\log  \E_m\left(\chi_A|\widehat{\xi_{i}} \right)(x)  \\
& =\sum_{k=1}^\infty \chi_{F_k}(x)  \sum_{A\in \P_{0}^{kN-1}} \chi_{A}(x)\log  \E_m\left(\chi_A|\widehat{\xi_{i}} \right)(x)\\
&=-\sum_{k=1}^\infty \chi_{F_k}(x) {\bf I}_m(\P_0^{kN-1} |\widehat{\xi_i})(x)\\
&=-\sum_{k=1}^\infty \chi_{F_k}(x)\sum_{j=0}^{kN-1} {\bf I}_m(\P|\widehat{\xi_i})(\sigma^j x)    \qquad\mbox{(by \eqref{e-xi})}.
\end{split}
\end{equation*}
This proves \eqref{e-4.13}.

Next we prove that $g\in L^1(m_F)$. We mainly follow the arguments in \cite[Lemma 3.3 and Proposition 3.5]{FengHu2009}. By Theorem \ref{thm-2.1}, for any given $C\in \xi_{i-1}$, the conditional measures $m^{\xi_{i-1}}_x$ ($x\in C$) represent the same measure supported on $C$, which we rewrite as $m_C$.  Fix  $C\in \xi_{i-1}$, $k\in \N$ and   $A\in \P_{0}^{kN-1}$. We define measures $\mu_C$ and $\nu_C$ on $Y_i$ by $\mu_C(E)=m_C(\phi_i^{-1}E\cap A)$ and $\nu_C(E)=m_C(\phi_i^{-1}E)$ for all $E\in \B(Y_i)$.
By the Hardy-Littlewood maximal inequality (see, e.g.~Theorem 2.19 in \cite{Mattila1995}), there exists a positive constant $a$ (which depends on $Y_i$) such that
$$
\mu_C\left\{z\in Y_i:\;
\inf_{r>0}\frac{\mu_C(B(z,r))}{\nu_C(B(z,r))}<u \right\}\leq a u\qquad (u>0).
$$
Hence for any $u>0$,
$$
m_C\left(\
\left\{
x\in \Sigma':\;
\inf_{r>0} \frac{m_C\left(B^{\phi_i}(x,r)\cap A\right)}
{m_C\left(B^{\phi_i}(x,r)\right)}<u
\right\}
\cap A \right)
\leq a u.
$$
Integrating $C$ over $\xi_{i-1}$, we obtain
$$
m\left(\
\left\{
x\in \Sigma':\;
\inf_{r>0} \frac{m^{\xi_{i-1}}_x\left(B^{\phi_i}(x,r)\cap A\right)}
{m^{\xi_{i-1}}_x\left(B^{\phi_i}(x,r)\right)}<u \right\}
\cap A \right)
\leq a u.
$$
Write $\displaystyle g^A(x)=\inf_{r>0}\frac
 {m^{\xi_{i-1}}_x\left(B^{\phi_i}(x,r)\cap A\right)}
 {m^{\xi_{i-1}}_x\left(B^{\phi_i}(x,r) \right)}$.
Then the above inequality can be rewritten as
\begin{equation}
\label{e-alambda}
m(A\cap \{g^A<u\})\leq a u.
\end{equation}
Note that by (\ref{e-2.5''}) and \eqref{e-4.14},  $g(x)=-\sum_{k=1}^\infty \sum_{A\in \P_{0}^{kN-1} }\chi_{F_k\cap A} (x)\log g^A(x)$. Since $g$ is non-negative,
\begin{eqnarray*}
\int g\; dm &=&\int_0^\infty m\{g>t\}\;dt\\
&=& \int_0^\infty\sum_{k=1}^\infty \sum_{ A\in \P_{0}^{kN-1} }  m(F_k\cap A\cap \{g^A<e^{-t}\})\; dt \\
&\leq& \sum_{k=1}^\infty \sum_{ A\in \P_{0}^{kN-1} }  \int_0^\infty  \min \{m(F_k\cap A),
a e^{-t}\}\;dt \qquad\quad  (\mbox{by \eqref{e-alambda}})\\
&\leq& \sum_{k=1}^\infty \sum_{ A\in \P_{0}^{kN-1} }   \left(-m(F_k\cap A)\log m(F_k\cap A)+ m(F_k\cap A) (1+\log a) \right)\\
&\leq& 1+\log a + \sum_{k=1}^\infty \sum_{ A\in \P_{0}^{kN-1} }   \left(-m(F_k\cap A)\log m(F_k\cap A) \right)\\
&\leq& 1+\log a+ \sum_{k=1}^\infty m(F_k) \left[ \left(\sum_{ A\in \P_{0}^{kN-1} }
  -\frac{m(F_k\cap A)}{m(F_k)} \log \frac{m(F_k\cap A)}{m(F_k)}\right)\right.\\
  &\mbox{}&\qquad \qquad\qquad \qquad \qquad \qquad \qquad \left. + \log  \frac{1}{m(F_k)} \right] \\
&\leq & 1+\log a+  \sum_{k=1}^\infty m(F_k) \left(kN\log (\#\Lambda) + \log  \frac{1}{m(F_k)}\right)\\
&<&\infty \qquad  \qquad \mbox{(by Lemma \ref{lem-induce}(ii)-(iii))}.
\end{eqnarray*}
 This finishes the proof of the proposition.
\end{proof}

Finally we are ready to prove \eqref{e-import}, the last step in the proof of  Proposition \ref{pro-4.1}.

\begin{proof}[Proof of  \eqref{e-import}]
Fix $\epsilon>0$ and write $F=F(\epsilon)$. Let $i\in \{1,\ldots, s\}$.

  For $x\in F$ and $n\in \N$, define $$
{\rho}_n(i,x)=\prod_{k=0}^{n-1}{\rho}(i,\sigma_F^kx),
$$
where $\sigma_F^k:=(\sigma_F)^k$, and ${\rho}(i,x)=e^{\ell (x)(\lambda_i(x)-\epsilon)}$ (as defined  in \eqref{e-rho}). Moreover, write
\begin{align*}
H_n(x)&:=\log \frac{m_x^{\xi_{i-1}}\left({\Gamma_{i}}\left(x, {\rho}_n(i,x)\right) \right)}
{m_{\sigma_Fx}^{\xi_{i-1}} \left( {\Gamma_{i}}\left(\sigma_F x, {\rho}_{n-1}(i,\sigma_F x)\right) \right)},\\
G_n(x)&:=\log \frac{m_x^{\xi_{i-1}}\left({\Gamma_{i}}\left(x, {\rho}_n(i,x)\right)\cap \P_0^{\ell (x)-1}(x) \right)}
{m_{x}^{\xi_{i-1}} \left({\Gamma_{i}}\left( x, {\rho}_{n}(i, x)\right) \right)}.
\end{align*}
Then  for $m$-a.e.~$x\in F$,
\begin{align*}
H_n(x)+G_n(x)&=\log \frac{m_x^{\xi_{i-1}}\left({\Gamma_{i}}\left(x, {\rho}_n(i,x)\right)\cap \P_0^{\ell (x)-1}(x) \right)}
{m_{\sigma_Fx}^{\xi_{i-1}} \left( {\Gamma_{i}}\left(\sigma_F x, {\rho}_{n-1}(i,\sigma_F x)\right) \right)}\\
&=\log \frac{m_x^{\xi_{i-1}}\left(\xi_{i-1}(x)\cap {\Gamma_{i}}\left(x, {\rho}_n(i,x)\right)\cap \P_0^{\ell (x)-1}(x) \right)}
{m_{\sigma_Fx}^{\xi_{i-1}} \left( {\Gamma_{i}}\left(\sigma_F x, {\rho}_{n-1}(i,\sigma_F x)\right) \right)}\\
&\leq \log \frac{m_x^{\xi_{i-1}}\left(  \sigma^{-\ell(x)} ({\Gamma_{i}}  \left(\sigma_F x, {\rho}_{n-1}(i,\sigma_F x))
\cap \xi_{i-1}(\sigma_Fx)
\right)\cap \P_0^{\ell (x)-1}(x) \right)}
{m_{\sigma_Fx}^{\xi_{i-1}}\left( {\Gamma_{i}}\left(\sigma_F x, {\rho}_{n-1}(i,\sigma_F x)\right) \right)} \quad (\mbox{by \eqref{e-t2}})\\
&
\leq \log \frac{m_x^{\xi_{i-1}}\left(  \sigma^{-\ell(x)} ({\Gamma_{i}}  \left(\sigma_F x, {\rho}_{n-1}(i,\sigma_F x))\right)\cap \P_0^{\ell (x)-1}(x) \right)}
{m_{\sigma_Fx}^{\xi_{i-1}}\left( {\Gamma_{i}}\left(\sigma_F x, {\rho}_{n-1}(i,\sigma_F x)\right) \right)}
\\
&
= \log \frac{m_x^{\xi_{i-1}}\left(  \sigma^{-\ell(x)}({\Gamma_{i}}  \left(\sigma_F x, {\rho}_{n-1}(i,\sigma_F x))\right)\cap \P_0^{\ell (x)-1}(x) \right)}
{m_{\sigma^{\ell(x)}x}^{\xi_{i-1}}\left( {\Gamma_{i}}\left(\sigma_F x, {\rho}_{n-1}(i,\sigma_F x)\right) \right)}
\\
&=\log m^{\xi_{i-1}}_x(\P_0^{\ell (x)-1}(x)) \qquad \qquad \qquad \qquad  (\mbox{by Lemma \ref{lem-3.6}(3)}) \\
&=-\sum_{k=1}^\infty \chi_{F_k}(x) \sum_{j=0}^{kN-1} {\bf I}_m(\P|\widehat{\xi_{i-1}})(\sigma^j x)=:Q_{i-1}(x) \qquad (\mbox{by \eqref{e-h1}}),
\end{align*}
that is, $H_n(x)+G_n(x)\leq Q_{i-1}(x)$.
Therefore for $m$-a.e.~$x\in F$,
\begin{align*}-\log m_x^{\xi_{i-1}}\left({\Gamma_{i}}(x, {\rho}_n(i, x)) \right)&=
-\left(\sum_{j=0}^{n-1}H_{n-j}(\sigma^j_Fx)\right)-
\log m^{\xi_{i-1}}_{\sigma^n_F x}\left({\Gamma_{i}}(\sigma^n_F x, 1)\right)\\
&\geq -\sum_{j=0}^{n-1}H_{n-j}(\sigma^j_Fx)\\
&\geq \sum_{j=0}^{n-1}\left(G_{n-j}(\sigma^j_Fx)-Q_{i-1}(\sigma_F^jx)\right),\\
\end{align*}
and thus
\begin{eqnarray*}\label{e-5.8}
\frac{-\log m_x^{\xi_{i-1}}\left(\Gamma_{i}(x, {\rho}_n(i,x)) \right)}{n}&\geq& \frac{1}{n}\sum_{j=0}^{n-1}\left(G_{n-j}(\sigma^j_F x)- Q_{i-1}(\sigma^j_Fx)\right).
\end{eqnarray*}
Notice that by Proposition \ref{pro-4.5}, when $n\to +\infty$,
\begin{eqnarray*}
 G_n\to Q_{i}:=-\sum_{k=1}^\infty\chi_{F_k}\sum_{j=0}^{kN-1}  {\bf I}_m(\P|\widehat{\xi_{i}})\circ \sigma^j
\end{eqnarray*}
pointwise and in $L^1$. By Lemma \ref{lem-3.15}, for $m$-a.e.~$x\in F$,
\begin{eqnarray*}
\liminf_{n\to \infty}\frac{-\log
m_x^{\xi_{i-1}}\left(\Gamma_{i}\left(x, {\rho}_n(i, x)\right)
\right)}{n}&\geq&  \E_{m_F}((Q_{i}-Q_{i-1})|\I_F)(x),
\end{eqnarray*}
where $\I_F:=\{B\in \B|_F:\; \sigma_F^{-1}(B)=B\}$.
In the meantime, by the Birkhoff ergodic theorem,
 \begin{equation*}
\begin{split}
 \lim_{n\to
\infty} \frac{-1}{n}\log ({\rho}_n(i, x)) &= \E_{m_F}(N(-\lambda_i+\epsilon)r_F|\I_F)(x)\\
&=N(-\lambda_i(x)+\epsilon)\E_{m_F}(r_F|\I_F)(x) \quad \mbox{$m_F$-a.e.},
\end{split}
\end{equation*}
where we use the fact that $\lambda_i$ is $\sigma$-invariant and thus $\sigma_F$-invariant.
Hence for $m$-a.e.~$x\in F$,
 \begin{eqnarray*}
  \liminf_{r\to 0}\frac{\log m_x^{\xi_{i-1}}\left(\Gamma_{i}(x, r)
\right)} {\log r}
  &=&\liminf_{n\to \infty}\frac{\log m_x^{\xi_{i-1}}\left(\Gamma_{i}\left(x, {\rho}_n(i, x)\right)
\right)}
  {\log ({\rho}_n(i, x))}\\
  &\geq &
\frac{ \E_{m_F}((Q_{i}-Q_{i-1})|\I_F)(x)}{N(-\lambda_i(x)+\epsilon)\E_{m_F}(r_F|\I_F)(x)}\\
&=& \frac{\E_m\left(\big({\bf I}_m(\P|\widehat{\xi_{i-1}})-{\bf I}_m(\P|\widehat{\xi_{i}})\big)
|{\mathcal I}
\right)(x)}{-\lambda_i(x)+\epsilon}\quad\mbox{(by Lemma \ref{lem-inderg})}\\
&=& \frac{h_{i-1}(x)-h_{i}(x)}{-\lambda_i(x)+\epsilon}.
  \end{eqnarray*}
That is, \eqref{e-import} holds.  This completes the proof of Proposition  \ref{pro-4.1}.
\end{proof}

\section{Local dimensions of invariant measures for affine IFSs}
\label{S-6}

In this section, we prove Theorems \ref{thm-1.0}-\ref{thm-1.2} and  \ref{cor-1.0}-\ref{thm-1.7'}.

Let $M:\; \Sigma\to {\rm Mat}_d(\R)$ be the matrix-valued function defined by $$M(x)=M_{x_{-1}},\quad x=(x_j)_{j=-\infty}^{+\infty}.$$
Let $m\in \M_\sigma(\Sigma)$. Let $$\R^d=\oplus_{i=1}^{s(x)}E_x^i\qquad (x\in \Sigma')$$ be the Oseledets splittings of $\R^d$  associated with $(\Sigma, \sigma^{-1},m)$ and  $M$ (see Section~\ref{S-4}), and
$0>\lambda_1(x)>\cdots>\lambda_{s(x)}(x)\geq -\infty$ the corresponding Lyapunov exponents. Below we prove parts (i) and (ii) of Theorem \ref{thm-1.0} separately.

\begin{proof}[Proof of Theorem \ref{thm-1.0}(i)]
In the beginning  we assume that the condition \eqref{e-assump} holds, that is,
for all $x\in \Sigma'$,
$$s(x)=s, \quad \mbox{ and }\quad \dim E_x^i=k_i \mbox{ for }i=1, \ldots, s.$$
(Just keep in mind that we don't assume that $m$ is ergodic at this moment.)

Write  $\displaystyle V_x^i=\oplus_{j={i+1}}^s E_x^j$ for $i=0,\ldots, s-1$, and $V_x^s=\{0\}$.
Clearly
$$ \{0\}=V_x^{s}\subset V_x^{s-1}\subset\cdots\subset V_x^0=\R^d.$$
Let $\xi_0,\xi_1,\ldots, \xi_s$ be the measurable partitions of $\Sigma'$  constructed as in Section~\ref{S-4}.   Furthermore, we set
\begin{equation}
\label{e-conv}
\xi_{-1} =\{\Sigma',\emptyset\}\qquad \mbox{and} \quad \lambda_{0}(x)=\lambda_1(x)  \mbox{ for $x\in \Sigma'$}.
\end{equation}
Clearly $\xi_{-1}(x)=\Sigma'$ for any $x\in \Sigma'$. By Lemma \ref{lem-3.4}(2), we have
\begin{equation}
\label{e-t301}
Q_{n,\epsilon}\cap \xi_i(x) \cap \P_0^{n-1}(x)\subset B^\pi(x, e^{n (\lambda_{i+1}(x)+\epsilon)}),\quad i=-1, 0,\ldots, s-1
\end{equation}
when $n$ is large enough. Here $B^\pi(x,r)$ is defined as in \eqref{e-ball}.

For $i=-1, 0,\ldots, s$, let  $\{m_x^{\xi_i}\}$ be the canonical system of conditional measures associated with  $\xi_i$.
By the definition of $\xi_{-1}$, we see that $m_x^{\xi_{-1}}= m$ for any $x\in \Sigma'$.

For $x\in \Sigma'$ and $i\in \{0,1,\ldots, s\}$, let $h_i(x)$  be defined as in \eqref{e-hi}.  Due to \eqref{e-h2} we write  $$h_{-1}(x)=h_0(x).$$
 According to  Lemmas \ref{lem-3.5}  and \ref{lem-4.6},
 \begin{equation}
\label{e-h}
\lim_{n\to \infty} \frac{-\log m_x^{\xi_i}(Q_{n,\epsilon} \cap \P_0^{n-1}(x))}{n}=h_i(x) \quad \mbox{for $m$-a.e.$~x\in \Sigma'$},\qquad i=-1,0,\ldots, s.
\end{equation}

 For $x\in \Sigma'$ and $r>0$, let $\Gamma_i(x,r)$ be defined as in \eqref{gamma}, that is,
$$
{\Gamma_i}(x,r)=\{y\in \Sigma':\; \mbox{dist}(\pi y+V_x^i, \; \pi x+V_x^i)\leq r\}, \quad i=1,\ldots, s.
$$
Write for convention that
$$
{\Gamma_0}(x,r)=\Sigma'.
$$
It is easy to see that for $i=0,1, \ldots, s$,
\begin{equation}
\label{e-5.4}
{\Gamma_i}(x,r)=\{y\in \Sigma':\; \|P_{({V_x^i})^\perp}(\pi y-\pi x)\|\leq r\},
\end{equation}
where $({V_x^i})^\perp$ stands for the orthogonal complement of  the space  $V_x^i$ in $\R^d$, and $P_{W}$ is the orthogonal projection from $\R^d$ to $W$.

Moreover,
define
\begin{equation}
\vartheta_i(x)=\liminf_{r\to 0}\frac{\log
m^{\xi_i}_x({\Gamma_{i+1}}(x,r))}{\log r},     \qquad i=-1,0,\ldots, s-1.
\end{equation}
Clearly $\vartheta_{-1}(x)=0$ for every $x\in \Sigma'$ since ${\Gamma_0}(x,r)=\Sigma'$.  Combining this with Proposition \ref{pro-4.1} yields
\begin{equation}
\label{e-h'}
\vartheta_i(x)\geq \frac{h_{i+1}(x)-h_{i}(x)}{\lambda_{i+1}(x)}\qquad (i=-1, 0, \ldots, s-1)
\end{equation}
for $m$-a.e.~$x\in \Sigma'$.

For $i=-1, 0,\ldots, s$ and $x\in \Sigma'$, define
$$
\overline{\delta}_i(x)=\limsup_{r\to 0}\frac{\log m^{\xi_i}_x(B^\pi(x,r))}{\log r},\quad
\underline{\delta}_i(x)=\liminf_{r\to 0}\frac{\log m^{\xi_i}_x(B^\pi(x,r))}{\log r}.
$$
We claim that for $m$-a.e.~$x\in \Sigma'$,
\begin{itemize}

\item[(C1)] $\overline{\delta}_{s}(x)=\underline{\delta}_{s}(x)=0$.
\item[(C2)]
 $\displaystyle \frac{h_{i+1}(x)-h_{i}(x)}{\lambda_{i+1}(x)}\geq \overline{\delta}_i(x)-\overline{\delta}_{i+1}(x)$ for  $i=-1, 0, \ldots, s-1$.
\item[(C3)]
 $\underline{\delta}_{i+1}(x)+\vartheta_i(x)\leq \underline{\delta}_i(x)$ for  $i=-1,0,\ldots, s-1$.
\end{itemize}

It is easy to see that (C1)-(C3) together with  (\ref{e-h'}) force inductively that for $m$-a.e.~$x\in \Sigma'$,
\begin{eqnarray}
\vartheta_i(x)&=& \frac{h_{i+1}(x)-h_{i}(x)}{\lambda_{i+1}(x)}  \quad \mbox{ for }
i=s-1,  \ldots, 0, -1,\label{e-h'1}\\
\underline{\delta}_i(x)&=&\overline{\delta}_i(x)
\quad \mbox{ for } i=s, s-1,\ldots,0,-1  \nonumber
\end{eqnarray}
 (we write the common value as $\delta_i(x)$),  and furthermore
\begin{equation}
\label{e-key11}
\delta_{i}(x)=\sum_{j=i}^{s-1} \vartheta_j(x)=
\sum_{j=i}^{s-1} \frac{h_{j+1}(x)-h_{j}(x) }{\lambda_{j+1}(x)}
\end{equation}
for $i=-1, 0,\ldots,s$.  In particular,
\begin{equation}
\label{e-key1}
\dim_{\rm loc}(\pi_*m,\pi x)=\delta_{-1}(x)=\delta_0(x)=\sum_{i=0}^{s-1} \vartheta_i(x)=
\sum_{i=0}^{s-1} \frac{h_{i+1}(x)-h_{i}(x) }{\lambda_{i+1}(x)}
\end{equation}
for $m$-a.e.~$x$,
which proves  Theorem \ref{thm-1.0}(i) under the additional assumption \eqref{e-assump}. In the
following we prove (C1)-(C3) respectively.

\bigskip

\begin{proof}[Proof of (C1)]  Since $\xi_s(x)=\pi^{-1}(\pi x)\cap \xi_0(x)\subset B^\pi(x,r)$ for any $x\in \Sigma'$ and $r>0$, we have
 $$m^{\xi_s}_x(B^\pi(x,r))\geq m^{\xi_s}_x(\xi_s(x))=1$$
 and so $m^{\xi_s}_x(B^\pi(x,r))=1$ for all $x\in \Sigma'$. Thus $\overline{\delta}_s(x)=\underline{\delta}_s(x)=0$ for all $x\in \Sigma'$.
\end{proof}

\bigskip

\begin{proof}[Proof of (C2)]  We give a proof by contradiction, which is  modified from \cite[\S10.2]{LedrappierYoung1985} and the proof of \cite[Theorem 2.11]{FengHu2009}.
Assume that (C2) is not true. Then there exists $i\in \{-1,0,\ldots, s-1\}$  such
that
$$
\frac{h_{i+1}(x)-h_i(x)}{\lambda_{i+1}(x)}<\overline{\delta}_i(x)-\overline{\delta}_{i+1}(x)
$$
on a set $U=U_i\subset \Sigma'$ with positive measure. Fix  such $i$. Removing a suitable subset from $U$ if necessary, we may assume that  one of the following holds: (a) $\lambda_{i+1}(x)\neq -\infty$ for all $x\in U$; or (b) $\lambda_{i+1}(x)=-\infty$ and $\overline{\delta}_i(x)>\overline{\delta}_{i+1}(x)$  for all $x\in U$. Notice that  (b) can not occur unless $i=s-1$, since $\lambda_{i+1}(x)\neq -\infty$ for $i<s-1$.

Now we first assume that the scenario (a) occurs. Then there exist $\alpha>0$ and real numbers $h_i,h_{i+1},\lambda_{i+1}, \overline{\delta}_i,\overline{\delta}_{i+1}$ with $\lambda_{i+1}<0$ such that
\begin{equation}
\label{e-ass}
\frac{h_{i+1}-h_i}{\lambda_{i+1}}< \overline{\delta}_i-\overline{\delta}_{i+1}-\alpha
\end{equation}
and for any $\epsilon>0$, there exists $B_\epsilon\subset U$ with $m(B_\epsilon)>0$ so that for $x\in B_\epsilon$,
$$|h_i(x)-h_i|<\epsilon/2, \quad |h_{i+1}(x)-h_{i+1}|<\epsilon/2, \quad \lambda_{i+1}(x)<\lambda_{i+1}+\epsilon/2$$
and
$$\overline{\delta}_i(x)\geq \overline{\delta}_i-\epsilon/2, \quad \overline{\delta}_{i+1}(x)<\overline{\delta}_{i+1}+\epsilon/2.$$
Fix $\epsilon\in (0, -\lambda_{i+1}/3)$. There exists $n_0\colon B_\epsilon\to \N$ such
that for $m$-a.e.~$x\in B_\epsilon$ and $n>n_0(x)$, we have
\begin{itemize}
\item[(1)] $\displaystyle
\frac{\log m^{\xi_{i+1}}_x\left(B^\pi(x,e^{n(\lambda_{i+1}+2\epsilon)})\right)}
 {n(\lambda_{i+1}+2\epsilon)}<
 \overline{\delta}_{i+1}+\epsilon;$
\item[(2)] $\displaystyle
-\frac{1}{n} \log m^{\xi_{i+1}}_x(\P_0^{n-1}(x))> h_{i+1}-\epsilon$\qquad
(by (\ref{e-h1}));

\item[(3)] $\displaystyle
Q_{n,\epsilon} \cap \xi_i(x)\cap \P_0^{n-1}(x)\subset B^\pi(x,e^{n(\lambda_{i+1}+2\epsilon)})$ \qquad (by (\ref{e-t301}));

\item[(4)] $\displaystyle
-\frac{1}{n}\log m^{\xi_i}_x(Q_{n,\epsilon} \cap \P_0^{n-1}(x))< h_i+\epsilon$\qquad (by
(\ref{e-h})).
\end{itemize}

 Take $N_0$ such that $$
\Delta:=\{x\in B_\epsilon\colon n_0(x)\leq N_0\}$$
 has the positive measure.  By Lemma \ref{lem-2.4}(1) and Lemma \ref{lem-2.10}, there exist $c>0$ and $\Delta'\subset \Delta$ with $m(\Delta')>0$ such that for $x\in \Delta'$, there exists $n=n(x)\geq N_0$ such that

\begin{itemize}
\item[(5)]  $\displaystyle \frac{m^{\xi_{i+1}}_x (L\cap \Delta)}{m^{\xi_{i+1}}_x(L)}> c$, where
$$
L:=B^\pi(x,e^{n(\lambda_{i+1}+2\epsilon)});
$$
\item[(6)]
$\displaystyle
\frac{\log m^{\xi_i}_x\left(B^\pi(x,2e^{n(\lambda_{i+1}+2\epsilon)})\right)} {n(\lambda_{i+1}+2\epsilon)}> \overline{\delta}_{i}-\epsilon;$

\item[(7)] $\displaystyle\frac{\log (1/c)}{n}<\epsilon$.

\end{itemize}

Take $x\in \Delta'$ such that (1)--(7) are satisfied with $n=n(x)$.
Write $C=\xi_{i+1}(x)$ and $C'=\xi_i(x)$. Then by (5) and (1),
$$m_x^{\xi_{i+1}}(L\cap \Delta )\geq cm_x^{\xi_{i+1}}(L)\geq
ce^{n(\lambda_{i+1}+2\epsilon)(\overline{\delta}_{i+1}+\epsilon)}.$$
But for each $y\in L\cap \Delta$,  by (2),
$m_y^{\xi_{i+1}}(\P_0^{n-1}(y))< e^{-n(h_{i+1}-\epsilon)}$. It follows that
the number of distinct $\P_0^{n-1}$-atoms intersecting $C\cap L\cap
\Delta$ is larger than $$m_x^{\xi_{i+1}}(L\cap \Delta)
e^{n(h_{i+1}-\epsilon)}.$$ However each such a $\P_0^{n-1}$-atom, say
$\P_0^{n-1}(y)$,  intersects $C'\cap L\cap \Delta$. This  implies that $Q_{n,\epsilon} \cap C'\cap \P_0^{n-1}(y)$ is contained in $C'\cap
B^\pi(x,2e^{n(\lambda_{i+1}+2\epsilon)})$. To see this implication,
let $z\in \P_0^{n-1}(y)\cap C^\prime\cap L\cap \Delta$; since $z\in
L\cap \Delta$, we have $d(\pi z, \pi x)\leq
e^{n(\lambda_{i+1}+2\epsilon)}$ and thus
\begin{align*}
Q_{n,\epsilon}\cap C'\cap\P_0^{n-1}(y)&=Q_{n,\epsilon}\cap \xi_i(z)\cap \P_0^{n-1}(z)\\
&\subset  B^\pi(z,e^{n(\lambda_{i+1}+2\epsilon)})\quad (\mbox{by (3)})\\
&\subset B^\pi(x,2e^{n(\lambda_{i+1}+2\epsilon)}),
\end{align*}
so $Q_{n,\epsilon}\cap C'\cap\P_0^{n-1}(y)\subset C'\cap B^\pi(x,2e^{n(\lambda_{i+1}+2\epsilon)})$, as desired.
In the meantime,  by (4), $m^{\xi_i}_x(Q_{n,\epsilon}\cap \P_0^{n-1}(y))\geq e^{-n(h_i+\epsilon)}$.  (To see it, picking $w\in \P_0^{n-1}(y)\cap  C'\cap L\cap \Delta$, we have $\xi_i(x)=\xi_i(w)$ and thus $m^{\xi_i}_x(Q_{n,\epsilon}\cap \P_0^{ n-1}(y))=m_w^{\xi_i}(Q_{n,\epsilon}\cap \P_0^{n-1}(w))\geq e^{-n(h_i+\epsilon)}$.) Hence
\begin{eqnarray*}
m^{\xi_i}_{x}(B^\pi(x,2e^{n(\lambda_{i+1}+2\epsilon)})) &\geq &
 \#  \{ \mbox{$\P_0^{n-1}$-atoms intersecting } C'\cap L\cap\Delta\} \cdot e^{-n(h_i+\epsilon)}\\
 &\geq & m_x^{\xi_{i+1}}(L\cap \Delta)  e^{n(h_{i+1}-\epsilon)}e^{-n(h_i+\epsilon)}\\
&\geq & ce^{n(\lambda_{i+1}+2\epsilon)
(\overline{\delta}_{i+1}+\epsilon)}e^{n(h_{i+1}-\epsilon)}e^{-n(h_i+\epsilon)}.
\end{eqnarray*}
Combining the above inequality  with (6) yields
\begin{equation}
\label{e-eq1}
\begin{split}
\mbox{}&(\lambda_{i+1}+2\epsilon)(\overline{\delta}_{i}-\epsilon)\\
& \quad \geq (\lambda_{i+1}+2\epsilon)(\overline{\delta}_{i+1}+\epsilon)+\frac{\log c}{n}+h_{i+1}-h_{i}-2\epsilon\\
& \quad \geq (\lambda_{i+1}+2\epsilon)(\overline{\delta}_{i+1}+\epsilon)+h_{i+1}-h_{i}-3\epsilon.
\end{split}
\end{equation}
Taking $\epsilon\to 0$ yields $h_{i+1}-h_i\leq \lambda_{i+1} (\overline{\delta}_i-\overline{\delta}_{i+1})$, which leads to a contradiction with (\ref{e-ass}) (keep in mind that $\lambda_{i+1}<0$).

Next we assume that the scenario (b) occurs, that is, $\lambda_{i+1}(x)=-\infty$ and $\overline{\delta}_i(x)>\overline{\delta}_{i+1}(x)$  for all $x\in U$. In this case,  $i=s-1$, and thus by (C1), $\underline{\delta}_{i+1}(x)=\overline{\delta}_{i+1}(x)=0$ for all $x\in U$. So $\overline{\delta}_i(x)>0$ for all $x\in U$. Hence there exist  real numbers $h_i, h_{i+1}, \overline{\delta}_i$ with $
\overline{\delta}_i>0$, so that for any $\epsilon>0$, there exists   $B_\epsilon\subset U$ with $m(B_\epsilon)>0$ such that for $x\in B_\epsilon$,
$$|h_i(x)-h_i|<\epsilon/2, \quad |h_{i+1}(x)-h_{i+1}|<\epsilon/2, \quad \overline{\delta}_i(x)\geq \overline{\delta}_i-\epsilon/2.$$
Set $\lambda_{i+1}=(-1/\epsilon)-2\epsilon$ and  $\overline{\delta}_{i+1}=0$. Then an argument similar to that for the scenario (a) shows that the previous estimates (1)-(7) hold, and moreover, the inequality \eqref{e-eq1} still holds. Taking $\epsilon\to 0$ gives  $\overline{\delta}_i\leq  \overline{\delta}_{i+1}=0$, leads to a contradiction with $\overline{\delta}_i>0$.
\end{proof}

\bigskip

\begin{proof}[Proof of (C3)] Here  we give a proof by contradiction, following the lines of the proof of \cite[Theorem 2.11]{FengHu2009}, in which the arguments  were adapted from the original proof of \cite[Lemma 11.3.1]{LedrappierYoung1985}.
Assume that (C3) is not true. Then there exists $i\in \{-1,0,\ldots, s-1\}$
such that $
\underline{\delta}_{i+1}(x)+\vartheta_i(x)>\underline{\delta}_i(x) $
on a subset of $\Sigma'$ with positive measure. Hence there exist
$\beta>0$ and real numbers
$\underline{\delta}_i,\underline{\delta}_{i+1}, \vartheta_i$   such
that
\begin{equation}
\label{e-6.12}
\underline{\delta}_{i+1}+\vartheta_i>\underline{\delta}_i+\beta,
\end{equation}
and for any $\epsilon>0$, there exists $A_\epsilon\subset \Sigma'$ with $m(A_\epsilon)>0$ so that for $x\in A_\epsilon$,
\begin{equation}\label{e-as2}
|\underline{\delta}_i(x)-\underline{\delta}_i|<\epsilon/2, \quad
|\underline{\delta}_{i+1}(x)-\underline{\delta}_{i+1}|<\epsilon/2,
\quad |\vartheta_i(x)-\vartheta_{i}|<\epsilon/2.
\end{equation}

Let  $0<\epsilon<\beta/4$. Find $N_1$ and a set $A_\epsilon'\subset A_\epsilon$ with $m(A_\epsilon')>0$ such that
\begin{equation}
\label{e-6.2} m_x^{\xi_{i+1}}\left(B^\pi(x,2e^{-n})\right)\leq
e^{-n(\underline{\delta}_{i+1}-\epsilon)}\qquad \mbox{ for }\; x\in A_\epsilon' \mbox{  and }n>N_1.
\end{equation}
By Lemma  \ref{lem-2.4}(1) and Lemma \ref{lem-2.10}, we can find $c>0$ and   $A_\epsilon''\subset A_\epsilon'$ with $m(A_\epsilon'')>0$ and
$N_2>N_1$ such that for all $x\in A_\epsilon''$ and $n\geq N_2$,
$$
\frac{m_x^{\xi_i}(A_\epsilon'\cap B^\pi(x,e^{-n}))}{m_x^{\xi_i}(B^\pi(x,e^{-n}))}>c.
$$
For $x\in A_\epsilon''$ and $n\geq N_2$, we have
\begin{equation}
\label{e-6.3}
\begin{split}
m_x^{\xi_i}(B^\pi(x,e^{-n}))&\leq c^{-1}m_x^{\xi_i}(A_\epsilon'\cap B^\pi(x,e^{-n}))\\
&= c^{-1}\int m_y^{\xi_{i+1}}(A_\epsilon'\cap B^\pi(x,e^{-n}))\; dm^{\xi_i}_x(y)\\
&= c^{-1}\int_{{\Gamma_{i+1}}(x,e^{-n})} m_y^{\xi_{i+1}}(A_\epsilon'\cap B^\pi(x,e^{-n}))\; dm^{\xi_i}_x(y),\\
\end{split}
\end{equation}
where in the last equality, we use the fact that $y\in {\Gamma_{i+1}}(x,e^{-n})$,  if  $y\in \xi_i(x)$ and  $\xi_{i+1}(y)\cap A_\epsilon'\cap B^\pi(x,e^{-n})\neq \emptyset$.
To see this fact, let $y\in \xi_i(x)$ such that $\xi_{i+1}(y)\cap A_\epsilon'\cap B^\pi(x,e^{-n})\neq \emptyset$.
Take  $w\in \xi_{i+1}(y)\cap  A_\epsilon'\cap B^\pi(x,e^{-n})$. Then  $\pi w-\pi y\in V_{y}^{i+1}$,  $\|\pi w-\pi x\|\leq e^{-n}$ and $w^-=y^{-}=x^{-}$ which implies $V_{w}^{i+1}=V_{y}^{i+1}=V_{x}^{i+1}$.  Hence
$$
\mbox{dist}(\pi y+V_x^{i+1}, \pi x+V_{x}^{i+1})= \mbox{dist}(\pi w+V_x^{i+1}, \pi x+V_{x}^{i+1})\leq \|\pi w-\pi x\|\leq e^{-n},
$$
and thus $y\in \Gamma_{i+1}(x,e^{-n})$. This completes the proof of  the fact. In the above argument, since $\|\pi w-\pi x\|\leq e^{-n}$, we have
$A_\epsilon'\cap B^\pi(x,e^{-n}) \subset  B^\pi(w,2e^{-n})$ and thus
\begin{eqnarray*}
m^{\xi_{i+1}}_y(A_\epsilon'\cap B^\pi(x,e^{-n}))&=&m^{\xi_{i+1}}_w(A_\epsilon'\cap B^\pi(x,e^{-n}))\\
&\leq& m^{\xi_{i+1}}_w(B^\pi(w,2e^{-n}))\\
&\leq& e^{-n(\underline{\delta}_{i+1}-\epsilon)} \qquad \mbox{(by \eqref{e-6.2})}.
\end{eqnarray*}
Combining  the above inequality  with (\ref{e-6.3}) yields
$$
m_x^{\xi_i}(B^\pi(x,e^{-n})) \leq c^{-1}e^{-n(\underline{\delta}_{i+1}-\epsilon)}
m^{\xi_i}_x({\Gamma_{i+1}}(x,e^{-n}))\qquad (x\in A_\epsilon'',\; n\geq N_2).
$$
Letting $n\to \infty$, we obtain $\underline{\delta}_i(x)\geq
\underline{\delta}_{i+1}-\epsilon+\vartheta_i(x)$ for $x\in
A_\epsilon''$. Combining this with (\ref{e-as2}) yields
$$
\underline{\delta}_i\geq
\underline{\delta}_{i+1}+\vartheta_i-4\epsilon\geq
\underline{\delta}_{i+1}+\vartheta_i-\beta,
$$
which contradicts  \eqref{e-6.12}.
\end{proof}

So far we have proved Theorem \ref{thm-1.0}(i)  under the additional assumption \eqref{e-assump}. Now we consider the general case that the integer functions $s(x)$ and  $\dim V_x^i$, $1\leq i\leq  s(x)$,  may not be constant over $\Sigma'$. In such case, by  Theorem \ref{thm-2} there exists a finite Borel partition $$\Sigma'=\bigsqcup_{j=1}^k \Sigma_j$$ of $\Sigma'$ so that for each $j$,
$\Sigma_j$ is  $\sigma$-invariant,  and   $s(x)$ and  $\dim V_x^i$ are constant   restricted on  $\Sigma_j$.    Ignore those indices $j$ with $m(\Sigma_j)=0$. We define probability measures $m_j$ by $$m_j=\frac{m|_{\Sigma_j}}{m(\Sigma_j)}.$$
 Then  $m_j\in \M_\sigma(\Sigma_j)$. Since now \eqref{e-assump} holds for $m_j$ (in which $\Sigma'$ is replaced by $\Sigma_j$), we see that  \eqref{e-key1} holds when replacing $m$ by $m_j$. In particular, the local dimension $\dim_{\rm loc}(\pi_*(m_j), \pi x)$ exists for $m_j$-a.e.~$x\in \Sigma_j$.   Equivalently,
 \begin{equation}
 \label{e-end1}
\lim_{r\to 0}\frac{\log m(\Sigma_j\cap B^\pi(x,r))}{\log r} \quad \mbox{ exists for $m$-a.e.~$x\in \Sigma_j$}.
 \end{equation}
  By Lemma  \ref{lem-2.4}(1) and Lemma \ref{lem-2.10}, for $m$-a.e.~$x\in \Sigma_j$, the following
 $$
\lim_{r\to 0} \frac{m(\Sigma_j\cap B^\pi(x,r))}{m(B^\pi(x,r))}
 $$
 exists and takes  positive value. This together with \eqref{e-end1} yields that the local dimension $\dim_{\rm loc} ( \pi_*m,\pi x)$ exists for $m$-a.e.~$x\in \Sigma_j$. Since $j$ is arbitrarily taken,  $\dim_{\rm loc} (\pi_*m,\pi x)$ exists for  $m$-a.e.~$x\in \Sigma'$.  This completes the proof of Theorem \ref{thm-1.0}(i).
 \end{proof}

 \begin{proof}[Proof of Theorems \ref{thm-1.0}(ii) and  \ref{thm-1.1}]
Since now $m$ is assumed to be  ergodic,  the condition \eqref{e-assump} holds and the functions $\lambda_i(x)$,  $h_i(x)$  ($i=-1,\ldots, s$) considered in the proof of Theorem \ref{thm-1.0}(i) are all
constant, which we denote by $\lambda_i$, $h_i$ respectively.  The formula \eqref{ly-dim} just follows from \eqref{e-key1}.
\end{proof}

 \begin{proof}[Proof of Theorem \ref{thm-1.2}]
 It is based on the proof of Theorem \ref{thm-1.0}.  To see \eqref{e-con-dim}, let $i\in \{0,\ldots, s-1\}$. By \eqref{e-key11},  for $m$-a.e.~$x\in \Sigma'$,
\begin{equation*}
\dim_{\rm loc}\left(\pi_*\big(m_x^{\xi_i}\big), \pi x\right)=\delta_i=\sum_{k=i}^{s-1}\frac{h_{k+1}-h_{k}}{\lambda_{k+1}}.
\end{equation*}
Equivalently,   for  $m$-a.e.~$x\in \Sigma'$ and $m^{\xi_i}_x$-a.e.~$y\in \xi_i(x)$,
\begin{equation*}
\dim_{\rm loc}\left(\pi_*\big(m_x^{\xi_i}\big), \pi y\right)=\delta_i=\sum_{k=i}^{s-1}\frac{h_{k+1}-h_{k}}{\lambda_{k+1}}.
\end{equation*}
Hence for $m$-a.e.~$x\in \Sigma'$, $\pi_*\big(m_x^{\xi_i}\big)$ is exact dimensional with dimension given by \eqref{e-con-dim}.

Next we prove \eqref{e-con-proj-dim} and \eqref{e-proj-dim}. By \eqref{e-h'1},  for $m$-a.e.~$x\in \Sigma'$,  \begin{equation}
\label{e-h''}
\vartheta_i(x)= \vartheta_i:= \frac{h_{i+1}-h_{i}}{\lambda_{i+1}} \quad \mbox{ for } i=-1, 0, \ldots, s-1.
\end{equation}

Let $\Gamma_i(x,r)$ ($x\in \Sigma')$, $0\leq i\leq  s$,  be defined as in \eqref{e-5.4}.

Fix $j\in \{1,\ldots, s\}$. For $i=-1,0,\ldots, j$ and $x\in \Sigma'$, define
$$
\overline{\gamma}_{i, j}(x)=\limsup_{r\to 0}\frac{\log m^{\xi_i}_x({\Gamma_j}(x,r))}{\log r},\quad
\underline{\gamma}_{i,j}(x)=\liminf_{r\to 0}\frac{\log m^{\xi_i}_x({\Gamma_j}(x,r))}{\log r}.
$$
We claim that
\begin{equation}
\label{e-claim1}
\xi_i(x)\cap  {\Gamma_j}(x,r)= \xi_i(x)\cap g^{-1}(B(gx, r)),
\end{equation}
where $g:\; \xi_i(x)\to {(V_x^j)}^\perp$ is defined by $y\mapsto P_{{(V_x^j)}^\perp}(\pi y)$. To see this, let $y\in \xi_i(x)\cap  {\Gamma_j}(x,r)$. Then $\mbox{dist}
(\pi y+V_x^j, \pi x+ V_x^j)\leq r$, equivalently, $\| gy-gx\|\leq r$; hence $y\in g^{-1}(B(gx, r))$. This proves the direction $\xi_i(x)\cap  {\Gamma_j}(x,r)\subset \xi_i(x)\cap g^{-1}(B(gx, r))$. The other direction can be proved similarly. This completes the proof of \eqref{e-claim1}.

Now due to \eqref{e-claim1}, we have $m^{\xi_i}_x({\Gamma_j}(x,r))=m^{\xi_i}_x(g^{-1}(B(gx, r)))$,  and so
\begin{equation}
\label{e-fact}
\begin{split}
\overline{\gamma}_{i,j}(x)&=\overline{\dim}_{\rm loc}\left(\big(P_{{(V_x^j)}^\perp}\pi\big)_*\big(m^{\xi_i}_x\big), \;   P_{{(V_x^j)}^\perp} (\pi x)\right), \\
\underline{\gamma}_{i,j}(x)&=\underline{\dim}_{\rm loc}\left(\big(P_{{(V_x^j)}^\perp}\pi\big)_*\big(m^{\xi_i}_x\big), \; P_{{(V_x^j)}^\perp} (\pi x)\right).
\end{split}
\end{equation}

We claim that for $m$-a.e.~$x\in \Sigma'$, the following properties hold:
\begin{itemize}
\item[(D1)] $\overline{\gamma}_{j,j}(x)=\underline{\gamma}_{j,j}(x)=0$.
\item[(D2)]
 $h_{i}-h_{i+1}\geq -\lambda_{i+1} (\overline{\gamma}_{i,j}(x)-\overline{\gamma}_{i+1,j}(x))$ for  $i=-1, 0, \ldots, j-1$.
\item[(D3)]
 $\underline{\gamma}_{i+1,j}(x)+\vartheta_i\leq \underline{\gamma}_{i,j}(x)$ for  $i=-1,0,\ldots, j-1$.
\end{itemize}

Clearly  (D1)-(D3) together with  (\ref{e-h''}) force that for $m$-a.e.~$x\in \Sigma'$,
$$\underline{\gamma}_{i,j}(x)=\overline{\gamma}_{i,j}(x) \quad
\mbox{ for $i=j, j-1,\ldots,  0, -1$},$$
(we write the common value as  $\gamma_{i,j}(x)$), and furthermore
\begin{align}
\gamma_{-1,j}(x)&=\sum_{k=0}^{j-1} \vartheta_k=
\sum_{k=0}^{j-1} \frac{h_{k+1}-h_{k} }{\lambda_{k+1}} \quad \mbox{ and }   \label{e-key1*} \\
\gamma_{i,j}(x)&=\sum_{k=i}^{j-1} \frac{h_{k+1}-h_{k} }{\lambda_{k+1}} \quad \mbox{ for } i\in \{0,1,\ldots, j-1\}. \label{e-key1_*}
\end{align}
Now \eqref{e-proj-dim} just follows from \eqref{e-key1*} and the fact \eqref{e-fact}. To see \eqref{e-con-proj-dim}, let $i\in \{0,\ldots, j-1\}$. By
\eqref{e-key1_*} and \eqref{e-fact}, we have for $m$-a.e.~$x\in \Sigma'$ and $m_x^{\xi_i}$-a.e.~$y\in \xi_i(x)$,
$$
\dim_{\rm loc}\left( \big(P_{{(V_x^j)}^\perp}\pi\big)_*\big(m^{\xi_i}_x\big), \;   P_{{(V_x^j)}^\perp} (\pi y)\right)=\gamma_{i,j}(x)=\sum_{k=i}^{j-1} \frac{h_{k+1}-h_{k} }{\lambda_{k+1}},
 $$
where we use the fact that
$V_y^i=V_x^i$  for $y\in \xi_i(x)$, due to  $y\in \xi_0(x)$ (see Lemma \ref{lem-3.1}).  As a consequence, for $m$-a.e.~$x\in \Sigma'$,  $\big(P_{{(V_x^j)}^\perp}\pi\big)_*\big(m^{\xi_i}_x\big)$ is exact dimensional and \eqref{e-con-proj-dim} holds.  To complete the proof of  Theorem \ref{thm-1.2},
 in the following we prove (D1)-(D3) respectively.

\medskip

 By the definition of $\xi_j$, for  $x\in \Sigma'$ and $y\in \xi_j(x)$, we have
 $\pi y-\pi x\in V_x^j$ and thus $\pi y +V_x^j=\pi x +V_x^j$. It follows that $y\in  {\Gamma_j}(x,r)$. Hence  $\xi_j(x)\subset \Gamma_j(x,r)$ and thus $m^j_x({\Gamma_j}(x,r))=1$ for $x\in \Sigma'$ and any $r> 0$. Hence $\overline{\gamma}_{j,j}(x)=\underline{\gamma}_{j,j}(x)=0$ for all $x\in \Sigma'$.
This proves (D1).

The proofs of   (D2) and (D3) are almost identical  to that of (C2) and (C3), respectively. Indeed  we only need to modify the proofs  of (C2) and (C3) slightly. More precisely,  among other minor adjustments,  we  may simply  replace the terms $\delta_i$, $\delta_{i+1}$, $B^\pi(x, e^{n(\lambda_{i+1}+2\epsilon)})$, $B^\pi(x, 2e^{n(\lambda_{i+1}+2\epsilon)})$, $B^\pi(x, e^{-n})$ therein  by $\gamma_{i,j}$,  $\gamma_{i+1,j}$,
 ${\Gamma_j}(x, e^{n(\lambda_{i+1}+2\epsilon)})$,   ${\Gamma_j}(x, 2e^{n(\lambda_{i+1}+2\epsilon)})$, and  ${\Gamma_j}(x, e^{-n})$ respectively.
This completes the proof of Theorem \ref{thm-1.2}.
\end{proof}

As a corollary of Theorem \ref{thm-1.2}, we have
\begin{cor}
\label{cor-6.1}
Under the assumptions of Theorem \ref{thm-1.1}, for  $i\in \{0,\ldots, s-1\}$ and $m$-a.e.~$x\in \Sigma'$,
\begin{equation}
\label{e-theta_i}
\vartheta_{i}(x)=\lim_{r\to 0}\frac{\log m_{x}^{\xi_{i}}(\Gamma_{i+1}(x,r))}{\log r}=\frac{h_{i+1}-{h_{i}}}{\lambda_{i+1}}\leq k_{i+1}.
\end{equation}
\end{cor}
\begin{proof}
Fix $i\in \{0,\ldots, s-1\}$. As is proved in Theorem \ref{thm-1.2},  for $m$-a.e.~$x\in \Sigma'$,
\begin{eqnarray*}
\lim_{r\to 0}\frac{\log m_{x}^{\xi_{i}}(\Gamma_{i+1}(x,r))}{\log r}=\gamma_{i,i+1}(x)
=\frac{h_{i+1}-{h_{i}}}{\lambda_{i+1}}.
\end{eqnarray*}

To see \eqref{e-theta_i} it remains to prove that $\frac{h_{i+1}-{h_{i}}}{\lambda_{i+1}}\leq k_{i+1}$.   By Theorem \ref{thm-1.2}, for $m$-a.e.~$x\in \Sigma'$, the measure $\eta_x:= \big(P_{{(V_x^{i+1})}^\perp}\pi\big)_*\big(m^{\xi_i}_x\big)$ is exact dimensional with dimension $\frac{h_{i+1}-{h_{i}}}{\lambda_{i+1}}$. However, $\eta_x$ is supported on the affine subspace $\pi x+ (V_x^{i}\ominus V_x^{i+1})$ of dimension $k_{i+1}$, where $V_x^{i}\ominus V_x^{i+1}$ stands for the orthogonal complement of $V_x^{i+1}$ in $V_x^{i}$. Hence $\dim_{\rm H}\eta_x\leq k_{i+1}$, and so,  $\frac{h_{i+1}-{h_{i}}}{\lambda_{i+1}}\leq k_{i+1}$.
\end{proof}
\begin{lem}
\label{lem-quasi}
\begin{itemize}
\item[(i)] Let $m\in \M_\sigma(\Sigma)$ be  quasi-Bernoulli. Then for $m$-a.e.~$x\in \Sigma$, $\pi_*\big(m_x^{\xi_0}\big)$ is strongly equivalent to $\pi_*m$.
\item[(ii)] Let $m\in \M_\sigma(\Sigma)$ be  sub-multiplicative. Then for $m$-a.e.~$x\in \Sigma$, $\pi_*\big(m_x^{\xi_0}\big)$ is absolutely continuous with respect to $\pi_*m$.
\end{itemize}
\end{lem}
\begin{proof} We first prove (i).
Since $m$ is quasi-Bernoulli, by definition there exists a positive constant $C$ such that
$$
C^{-1}m([I])m([J])\leq m([IJ])\leq C m([I])m([J])
$$
for all finite words $I, J$ over $\Lambda$. Below we show that for $m$-a.e.~$x\in \Sigma$,
\begin{equation}
\label{e-quasi}
C^{-1}m([I])\leq m_x^{\xi_0}([I])\leq C m([I])
\end{equation}
for all finite words $I$ over $\Lambda$. This is enough to conclude the strong equivalence between $\pi_*\big(m_x^{\xi_0}\big)$ and $\pi_*m$, since $\pi x$ only depends on $x^+:=(x_n)_{n=0}^\infty$.

To see \eqref{e-quasi}, note that the measurable partition $\xi_0$ is induced by the mapping $\tau:\; \Sigma \to \Sigma^{-}$, $x\mapsto x^{-}=(x_n)_{-\infty}^{-1}$. That is,
$\xi_0(x)=\{y\in \Sigma:\; \tau y=\tau x\}$ for every $x$.
  Applying Lemma \ref{lem-2.4}(1) to $\tau:\Sigma\to \Sigma^{-}$ yields that for $m$-a.e.~$x$,
 \begin{equation}
 \label{e-inequality}
 m^{\xi_0}_x([I])= {\bf E}_m(\chi_{[I]}|\tau^{-1}(\B(\Sigma^{-})))(x)=\lim_{n\to \infty}
 \frac{m([x_{-n}\ldots x_{-1}I])}{m( [x_{-n}\ldots x_{-1}])}
 \end{equation}
 for all  finite words $I$ over $\Lambda$. \eqref{e-quasi} is then obtained from the quasi-Bernoulli property of $m$.

 Next we prove (ii). Here $m$ is assumed to be sub-multiplicative and we only have the one-sided inequality  $m([IJ])\leq C m([I])m([J])$. However this is enough to derive from \eqref{e-inequality}
 that for $m$-a.e.~$x$, $m^{\xi_0}_x([I])\leq Cm([I])$ for all  finite words $I$ over $\Lambda$.  As a consequence, $ \pi_*\big(m_x^{\xi_0}\big)$ is absolutely continuous with respect to $\pi_*m$, with a uniformly bounded Radon-Nikodym derivative.
\end{proof}
\begin{proof}[Proof of Theorem \ref{cor-1.0}] We first prove  (i).  Fix $i\in \{1,\ldots, s-1\}$. By Theorem \ref{thm-1.2}, for $m$-a.e.~$x\in \Sigma'$,  $\pi_*\big(m_x^{\xi_i}\big)$ and $\big(P_{({V_x^i})^\perp}\pi\big)_*\big(m_x^{\xi_0}\big)$ are exact dimensional with
$$
\dim_{\rm H}\big(\pi_*\big(m_x^{\xi_i}\big)\big)=\sum_{k=i}^{s-1}\frac{h_{k+1}-h_k}{\lambda_{k+1}}
$$
and
$$
\dim_{\rm H} \left(\big(P_{({V_x^i})^\perp}\pi\big)_*\big(m_x^{\xi_0}\big) \right)=\sum_{k=0}^{i-1}\frac{h_{k+1}-h_k}{\lambda_{k+1}},
$$
 hence,
\begin{equation}
\label{e-sum}
\dim_{\rm H} \big(\pi_*\big(m_x^{\xi_i}\big)\big)+\dim_{\rm H}\left(\big(P_{({V_x^i})^\perp}\pi\big)_*\big(m_x^{\xi_0}\big)\right)=\dim_{\rm H}\left(\pi_*\big(m_x^{\xi_0}\big)\right).
\end{equation}

Next let $x\in \Sigma'$ and write $W=V_x^i$,  $\nu=m^{\xi_0}_x$, $\eta= \pi_*\nu$.  Notice that $\nu$ is supported on $\xi_0(x)$.  Consider the measurable partition $\zeta$ of $\R^d$ given by
$$
\zeta:=\{W+a:\; a\in W^{\perp}\}.
$$
Set $\pi^{-1}\zeta:=\{\xi_0(x)\cap \pi^{-1}(W+a):\; a\in W^{\perp}\}$. Then $\pi^{-1}\zeta$ is a measurable partition of $\xi_0(x)$. Let $\{\nu_y^{\pi^{-1}\zeta}\}_{y\in \xi_0(x)}$ be the system of conditional measures of $\nu$ associated with $\pi^{-1}\zeta$, and $\{\eta_z^\zeta\}_{z\in \R^d}$ the system of conditional measures of $\eta$ associated with $\zeta$. Write $\eta_{W,z}:=\eta_z^\zeta$.  By the uniqueness of conditional measures, we have for $\nu$-a.e.~$y$, \begin{equation}
\label{e-cond}
\pi_*\big(\nu_y^{\pi^{-1}\zeta}\big)=\eta_{W, \pi y}.
\end{equation}
Notice also that for $y\in \xi_0(x)$, the atom $(\pi^{-1}\zeta)(y)$ is nothing but $\xi_i(y)$.  Hence we have $\nu_y^{\pi^{-1}\zeta}=m^{\xi_i}_y$ for $m$-a.e.~$x$ and $m_x^{\xi_0}$-a.e.~$y$. This combining with \eqref{e-cond} gives
\begin{equation}
\label{e-comp}
 \pi_*\big(m^{\xi_i}_x\big)= \big( \pi_*\big(m_x^{\xi_0}\big)\big)_{V_x^i, \pi x}
\end{equation}
for $m$-a.e.~$x$. Plugging the above equality into  \eqref{e-sum}, we see that $\pi_*\big(m^{\xi_0}_x\big)$ satisfies dimension conservation along $V_x^i$. This proves (i).

 Now we turn to the proof of (ii). Suppose that $m$ is quasi-Bernoulli. By Lemma \ref{lem-quasi}(i), for $m$-a.e.~$x\in \Sigma'$, $\pi_*\big(m_x^{\xi_0}\big)$ is strongly equivalent to $\mu=\pi_*m$; as a consequence,
  $\big(P_{({V_x^i})^\perp}\pi\big)_*\big(m^{\xi_0}_x\big)$ is strongly equivalent to $\big(P_{({V_x^i})^\perp}\big)_*\mu$.
 It follows that  for $m$-a.e.~$x\in \Sigma'$,  $\big(P_{({V_x^i})^\perp}\big)_*\mu$ is exact dimensional with dimension $\sum_{k=0}^{i-1} \frac{h_{k+1}-h_k}{\lambda_{k+1}}$. Equivalently, for  $\big(\Pi_i\big)_*m$-a.e.~$W$,  $\big(P_{W^\perp}\big)_*\mu$ is exact dimensional with dimension $\sum_{k=0}^{i-1} \frac{h_{k+1}-h_k}{\lambda_{k+1}}$.

 Again since  $\pi_*\big(m_x^{\xi_0}\big)$ is strongly equivalent to $\mu$ for $m$-a.e.~$x$,  applying Lemma \ref{lem-2.8} to the orthogonal projection
 $P_{(V^i_x)^\perp}:\; \R^d\to (V^i_x)^\perp$, we see that $m$-a.e.~$x$, $\mu_{V_x^i, \pi x}$ is equivalent to $\big(\pi_*\big(m_x^{\xi_0}\big)\big)_{V_x^i, \pi x}=\pi_*\big(m^{\xi_i}_x\big)$, and so $\mu_{V_x^i, \pi x}$ is exact dimensional with dimension $\sum_{k=i}^{s-1} \frac{h_{k+1}-h_k}{\lambda_{k+1}}$. Equivalently, for  $\big(\Pi_i\big)_*m$-a.e.~$W$ and $\mu$-a.e.~$z$, $\mu_{W, z}$ is exact dimensional with dimension $\sum_{k=i}^{s-1} \frac{h_{k+1}-h_k}{\lambda_{k+1}}$. Recall that we have proved that for  $\big(\Pi_i\big)_*m$-a.e.~$W$,
 $\big(P_{W^\perp}\big)_*\mu$ is exact dimensional with dimension $\sum_{k=0}^{i-1} \frac{h_{k+1}-h_k}{\lambda_{k+1}}$.  This is enough to conclude (ii).

 Finally, we prove (iii). Suppose that $m$ is sub-multiplicative. By Lemma \ref{lem-quasi}(ii), for $m$-a.e.~$x\in \Sigma'$, $\pi_*\big(m_x^{\xi_0}\big)$ is absolutely continuous with respect to $\mu$. Hence there exists $H\subset \Sigma'$ with full $m$-measure such that for any $x\in H$,  there exists a Borel set $F_x\subset \R^d$ with positive $\mu$-measure  such that    $(\pi_*\big(m_x^{\xi_0}\big))_{F_x}$ is strongly equivalent to $\mu_{F_x}$, where $\nu_A$ stands for the probability measure defined by $\displaystyle\nu_A(\cdot)={\nu(A\cap \cdot)}/{\nu (A)}$. As is proved in part (ii), when $m$ is quasi-Bernoulli, we can take $F_x= \Sigma'$.

 Now fix $x\in H$ and $i\in \{1,\ldots, s-1\}$. Set $W=V^i_x$ and write for convenience \begin{align*}
 \eta:=\pi_*\big(m_x^{\xi_0}\big),\quad
  \eta':=(\pi_*\big(m_x^{\xi_0}\big))_{F_x},\quad
   \mu':=\mu_{F_x}.
 \end{align*}
 Applying Lemma \ref{lem-2.9} to the projection $P_{W^\perp}:\R^d\to \R^d$ and using the Borel density lemma, we see that for $\mu$-a.e.~$z\in F_x$ (equivalently for $\eta$-a.e.~$z\in F_x$),
 \begin{equation}
 \label{e-ff1}
 \begin{split}
 \dim_{\rm loc}((\eta')_{W,z}, z) &= \dim_{\rm loc}(\eta_{W,z}, z), \\
 \dim_{\rm loc}((\mu')_{W,z}, z) &= \dim_{\rm loc}(\mu_{W,z}, z), \\
 \dim_{\rm loc}\big(\big(P_{W^\perp}\big)_*\eta', P_{W^\perp}(z)\big) &= \dim_{\rm loc}\big(\big(P_{W^\perp}\big)_*\eta, P_{W^\perp}(z)\big),\\
\dim_{\rm loc}\big(\big(P_{W^\perp}\big)_*\mu', P_{W^\perp}(z)\big) &= \dim_{\rm loc}\big( \big(P_{W^\perp}\big)_*\mu, P_{W^\perp}(z)\big).
 \end{split}
 \end{equation}
 Since $\eta'$ and $\mu'$ are strongly equivalent, by Lemma \ref{lem-2.8}, for $\mu$-a.e.~$z\in F_x$,
 \begin{align*}
 \dim_{\rm loc}((\eta')_{W,z}, z) &= \dim_{\rm loc}((\mu')_{W,z}, z), \\
 \dim_{\rm loc}\big(\big(P_{W^\perp}\big)_*\eta', P_{W^\perp}(z)\big) &= \dim_{\rm loc}\big( \big(P_{W^\perp}\big)_*\mu', P_{W^\perp}(z)\big).
 \end{align*}
 Combining the above equalities  with \eqref{e-ff1} yields that for $\mu$-a.e.~$z\in F_x$,
 \begin{align*}
 \dim_{\rm loc}(\mu_{W,z}, z) &= \dim_{\rm loc}(\eta_{W,z}, z), \\
 \dim_{\rm loc}\big( \big(P_{W^\perp}\big)_*\mu, P_{W^\perp}(z)\big) &= \dim_{\rm loc}\big( \big(P_{W^\perp}\big)_*\eta, P_{W^\perp}(z)\big).
 \end{align*}
 Now (iii) follows from (i). This completes the proof of the theorem.
\end{proof}

\begin{proof}[Proof of Theorem \ref{thm-1.7'}] Here we only give a sketched proof. It is based on \cite[Theorem 2.11]{FengHu2009} and its proof.

 Since the linear parts $M_j$ of ${\mathcal S}$ commute, $\R^d$ can be  decomposed into the direct sum $T_1\oplus \cdots \oplus T_\ell$ of some subspaces  with dimensions $q_1,\ldots, q_\ell$,  so that
 for each pair $(j,p)\in \Lambda\times\{1,\ldots, \ell\}$, $M_jT_p\subset T_p$ and $M_j$ is ``weakly conformal'' on $T_p$ in the sense that there exists $a_{j,p}\geq  0$ so that
  $\lim_{n\to \infty}\|M_j^nv\|^{1/n}=a_{j,p}$ for  $v\in T_{p}\backslash\{0\}$. Hence under a suitable coordinate change, ${\mathcal S}$ can be written as the direct product of some ``weakly conformal'' affine IFSs ${\mathcal S_1}$, \ldots,  ${\mathcal S_\ell}$ on $\R^{q_1}$,\ldots, $\R^{q_\ell}$ (cf. \cite[Definition 2.10]{FengHu2009}).

   Set $\overline{\lambda}_p=\sum_{j\in \Lambda} m([j]) \log a_{j,p}$ for $p=1,\ldots, \ell$. Permutating ${\mathcal  S_j}$'s if necessary,  we may assume that
  $$
  \overline{\lambda}_1\geq \cdots \geq \overline{\lambda}_\ell.
  $$
For $p\in \{1,\ldots, \ell\}$ and let $\tau_p$ be the orthogonal projection from $\R^d$ to $Y_p:=\R^{q_1}\times \cdots \times \R^{q_p}$, and let $m^{\zeta_p}_x$ be the conditional measure of $m$ associated with the measurable partition $\{\pi^{-1}\circ \tau_p^{-1}(y):\ y\in  Y_p\}$ of $\Sigma$. It is implicitly proved in \cite[Theorem 2.11]{FengHu2009} that there exist $h_m(\sigma)=\overline{h}_0\geq \overline{h}_1\geq \cdots\geq \overline{h}_\ell\geq 0$ such that
for $m$-a.e.~$x\in \Sigma$ and $p\in \{1,\ldots, \ell-1\}$, the measure $\pi_*\big(m^{\zeta_p}_x\big)$ is exact dimensional with dimension $\sum_{j=p}^{\ell-1}\frac{\overline{h}_{j+1}-\overline{h}_j}{\overline{\lambda}_{j+1}}$, and moreover,
$\mu=\pi_*m$ is exact dimensional with dimension $\sum_{j=0}^{\ell-1}\frac{\overline{h}_{j+1}-\overline{h}_j}{\overline{\lambda}_{j+1}}$. (We remark that this is only proved in \cite{FengHu2009} in  the case when ${\mathcal S}$ is invertible and contracting. But it can be extended to the general case like Theorem \ref{cor-1.0}.)  Applying this result to the IFS ${\mathcal S}_1\times\cdots\times {\mathcal S}_{p}$ gives that $(\tau_p)_*\mu$ is exact dimensional with dimension  $\sum_{j=0}^{p-1}\frac{\overline{h}_{j+1}-\overline{h}_j}{\overline{\lambda}_{j+1}}$.

Set $\mu=\pi_*m$. Let $\{\mu_{Y_p^\perp,z}\}$ denote the system of conditional measures of $\mu$ associated with the measurable partition $\{\tau_p^{-1}(y):\; y\in Y_p\}$ of $\R^d$. Similar to the proof of \eqref{e-comp},  we can show that for $m$-a.e.~$x\in \Sigma$, $\mu_{Y_p^\perp, \pi x}=\pi_*\big(m^{\zeta_p}_x\big)$.  It follows that $\mu$ is    dimension conserving with respect to the projection $\tau_p$. Moreover,  $\mu_{Y_p^\perp, z}$ is exact dimensional for $\mu$-a.e~$z$.

Now let $1\leq p_1<\cdots<p_{s'}=\ell$ be those integers so that $$\overline{\lambda}_1=\cdots=\overline{\lambda}_{p_1}>\overline{\lambda}_{p_1+1}=\cdots=\overline{\lambda}_{p_2}>\cdots>
\overline{\lambda}_{p_{s'-1}+1}=\cdots=\overline{\lambda}_{p_{s'}}.$$
It is readily checked that $s=s'$,   $\lambda_i=\overline{\lambda}_{p_i}$  and $V_x^i=W_i:=Y_{p_i}^\perp$  for $1\leq i\leq s$ and $m$-a.e~$x$. In particular, $P_{(W_i)^\perp}=\tau_{p_i}$ for $i=1,\ldots, s-1$. Hence $\mu$ is  dimension conserving with respect to the projections $P_{(W_i)^\perp}$, $i=1,\ldots, s-1$.
\end{proof}

\begin{rem}
 \label{rem-6.3}
 {\rm The proof of Theorem \ref{thm-1.7'} implies the following result:  Let ${\mathcal S}=\{S_j(x)=r_jx+a_j\}_{j\in \Lambda}$ be a self-similar IFS on $\R^d$ with $r_j>0$, average contracting with respect to an ergodic $m\in {\mathcal M}_\sigma(\Sigma)$. Then for any proper subspace $W$ of $\R^d$, $\pi_*m$ is dimension conserving with respect to $P_W$.      This generalizes the result in \cite{FalconerJin2014, Furstenberg2008}. To see it, let $p=\dim W$ and let $v_1,\ldots, v_d$ be an orthonormal basis of $\R^d$ such that ${\rm span}(v_1,\ldots, v_p)=W$.  Then one can check that ${\mathcal S}$ can be written as the product ${\mathcal S}_1 \times\cdots  \times {\mathcal S}_d$ of some one-dimensional IFSs on $X_1,\ldots, X_d$, where $X_i={\rm span}(v_i)$, and moreover $\overline{\lambda}_1=\cdots = \overline{\lambda}_d$.  Now the desired dimension conservation property follows from the proof of Theorem \ref{thm-1.7'}.
}
\end{rem}

\section{Lyapunov dimension}
\label{S-7}

Throughout this section, let $m$ be an ergodic $\sigma$-invariant measure on $\Sigma$ and  ${\bf M}=(M_j)_{j\in \Lambda}$ be a tuple of $d\times d$ real matrices satisfying
 $$\lambda({\bf M},m):=\lim_{n\to \infty}\frac{1}{n}\int \log \|M_{x_0}\cdots M_{x_{n-1}}\| \; dm(x)<0.$$
Let $\mathcal S=\{S_{j}(x)=M_jx+a_j\}_{j\in \Lambda}$ be an affine IFS on $\R^d$. Let $\{(\lambda_i, k_i)\}_{1\leq i\leq s}$ be the Lyapunov spectrum of ${\bf M}$ with respect to $(\Sigma, \sigma^{-1}, m)$. Set
$$
L_0=0 \quad \mbox{ and }\quad  L_i=-\sum_{\ell=1}^{i} \lambda_\ell k_\ell\;\mbox{ for }\; i=1,\ldots, s.
$$
Clearly $L_0<L_1<\cdots<L_s$.
Following \cite{JordanPollicottSimon2007}, we give the following.

\begin{de}
\label{de-LY}
The Lyapunov dimension of $m$ with respect to ${\bf M}$, denoted as $\dim_{\rm LY} (m, {\bf M})$,  is defined to be
$$
\left\{
\begin{array}{ll}
\displaystyle \left(\sum_{\ell=0}^{j-1}k_\ell\right)+\frac{ h_m(\sigma)-L_{j-1}}{(-\lambda_{j})} & \mbox{ if } L_{j-1}\leq  h_m(\sigma) < L_j
\mbox{ for some }j\in \{1,\ldots, s\},\\
&\\
\displaystyle \frac{d \;h_m(\sigma)}{L_s} & \mbox{ if } h_m(\sigma) \geq  L_s.
\end{array}
\right.
$$
\end{de}

Let $\pi$ be the coding map associated with ${\mathcal S}$. Recall that $h_i$, $0\leq i\leq s$, are the conditional entropies of $m$ defined  in  \eqref{e-hi1}, and $h_0=h_m(\sigma)$. The following result says that the Lyapunov dimension of $m$ is always an upper bound for the Hausdorff dimension of $\pi_*m$. This result was first proved in \cite{JordanPollicottSimon2007} under a stronger  assumption that $\|M_j\|<1$ for all $j$.
\begin{pro}
\label{pro-7.2}
 $\dim_{\rm H}  \pi_*m \leq \min\{d,  \dim_{\rm LY}(m, {\bf M})\}$. Moreover, the equality holds if and only if one of the following holds:
 \begin{itemize}
 \item[(1)] $h_m(\sigma) \geq L_s$, and  $h_i=h_m(\sigma)-L_i$ for all $i\in \{1,\ldots, s\}$.
 \item[(2)] $h_m(\sigma)\in [L_{j-1},  L_j)$ for some $j\in \{1,\ldots, s\}$, and
 $$
 h_i=\left\{
 \begin{array}{ll}
 h_m(\sigma)-L_i & \mbox{ if }1\leq i\leq j-1,\\
 0 & \mbox{ if }j\leq i\leq s.

 \end{array}
 \right.
 $$
 \end{itemize}
\end{pro}
\begin{proof}
Since $\lambda({\bf M}, m)<0$, the IFS $\mathcal S$ is average contracting with repect to $m$. By Theorem \ref{thm-1.1},
$\dim_{\rm H}  \pi_*m=\sum_{i=0}^{s-1} \frac{h_{i+1}-h_i}{\lambda_{i+1}}$. Recall that $$0>\lambda({\bf M}, m)=\lambda_1>\cdots>\lambda_s\geq -\infty,$$ and
$$h_m(\sigma)=h_0\geq h_1\geq \cdots \geq h_s\geq 0.$$ Moreover by Corollary \ref{cor-6.1},    $h_{i}-h_{i+1}\leq (-\lambda_{i+1})k_{i+1}$ for each $i$. Hence
 $\dim_{\rm H}  \pi_*m$ is bounded above by
$$
\Delta:=\max\left\{\sum_{i=0}^{s-1}\frac{x_{i+1}-x_i}{\lambda_{i+1}}:\; h_m(\sigma)=x_0\geq  \cdots \geq x_s\geq 0, \;
\frac{x_{i+1}-x_{i}}{\lambda_{i+1}}\leq k_{i+1} \mbox{ for all } i \right\}.
$$

Now it is readily checked that the following hold:
(a) if $h_m(\sigma) \geq L_s$, then $\Delta=d$ and the maximum in defining $\Delta$ is attained uniquely at  $(x_0, x_1,\ldots, x_s)$ where $x_i=h_m(\sigma)-L_i$ for $0\leq i\leq s$;
(b) if $h_m(\sigma)\in [L_{j-1},  L_j)$ for some $j\in \{1,\ldots, s\}$, then
$$\Delta=\left(\sum_{\ell=0}^{j-1}k_\ell\right)+\frac{ h_m(\sigma)-L_{j-1}}{(-\lambda_{j})},
$$
and the maximum is attained uniquely at $(x_1,\ldots, x_s)$ where $x_i=h_m(\sigma)-L_i$ for $i\leq j-1$ and $0$ for $i\geq j$. As a consequence, the results of the proposition hold.
\end{proof}

\begin{rem}
\label{rem-q1}
By Proposition \ref{pro-7.2}, if $\dim_{\rm H} \pi_*m= \min\{d,  \dim_{\rm LY}(m, {\bf M})\}$, then
$$
\sum_{ \ell=1}^{j}\frac{h_\ell-h_{\ell-1}}{\lambda_\ell}=\min\{k_1+\cdots +k_j, \dim_{\rm H}  \pi_*m\} \quad \mbox{ for }j=1,\ldots, s.
$$
This result was partially proved in \cite[Corollay 2.7]{BaranyKaenmaki2015}.
\end{rem}

\begin{pro}
\label{pro-7.3}
Suppose that ${\mathcal S}$ is contracting and satisfies the strong separation condition. Then
the following statements hold.
\begin{itemize}
\item [(i)] $h_s=0$,  $h_m(\sigma)<L_s$ and  $\dim_{\rm LY}(m,{\bf M})<d$.
\item[(ii)] Let  $j$ be the unique element in $\{1,\ldots, s\}$ so that $L_{j-1}\leq h_m(\sigma)<L_j$. Then $\dim_{\rm H}  \pi_*m= \dim_{\rm LY} (m, {\bf M})$ if and only if
\begin{equation}
\label{e-equiv}
\sum_{ \ell=1}^{j-1}\frac{h_\ell-h_{\ell-1}}{\lambda_\ell}=d_{j-1}, \qquad \sum_{\ell=j+1}^s\frac{h_\ell-h_{\ell-1}}{\lambda_\ell}=0,
\end{equation}
where $d_0:=0$ and $d_i:=k_1+\cdots+k_i$ for $1\leq i\leq s$.
\end{itemize}
\end{pro}
\begin{proof}
(i) We first claim that $h_s=0$.  Since $\mathcal S$ satisfies the strong separation condition,  $\xi_s(x)=\{x\}$ for each $x\in \Sigma'$. Thus $\widehat \xi_s=\B(\Sigma')$ and hence $h_s=H_m(\P|\widehat{\xi_s})=0$.

Next we prove that $h_m(\sigma)<L_s$. Clearly this is true if $L_s=\infty$ (equivalently, if $\lambda_s=-\infty$). Below we assume that $\lambda_s>-\infty$.

    Let $K$ denote the self-affine set generated by ${\mathcal S}$.  For $\delta>0$ let $K_\delta$ be the closed $\delta$-neighborhood of $K$, i.e. $K_\delta=\{z:\; d(z, K)\leq \delta\}$. Since $\mathcal S$ satisfies the strong separation condition, we can pick a small $\delta$ such that $S_i(K_\delta)$ ($i\in \Lambda$) are disjoint subsets of the interior of $K_\delta$ and hence ${\mathcal L}^d(K_\delta)>\sum_{i\in \Lambda} {\mathcal L}^d(S_i(K_\delta))$. It follows that $\rho:=\sum_{i\in \Lambda} |\det(M_i)|<1$.

   Since $m$ is ergodic $\sigma$-invariant, by \cite[Lemma 3.2]{FengShmerkin2014} and the Shannon-McMillan-Breiman theorem, for $m$-a.e.~$x\in \Sigma$,
   \begin{equation}
   \label{e-det}
   \lim_{n\to \infty} \frac{\log |\det(M_{x_0\ldots x_{n-1}})|}{n}=-L_s, \quad \lim_{n\to \infty} \frac{\log m([x_0\ldots x_{n-1}])} {n}=-h_m(\sigma).
   \end{equation}

For $\epsilon>0$ and $n\in \N$, let $\Lambda_{n,\epsilon}$ denote the set of words $I$ of length $n$ over the alphabet $\Lambda$ such that
$$|\det(M_I)|\geq e^{-nL_s-n\epsilon},\quad  m([I]) \leq e^{-nh_m (\sigma)+n\epsilon}.$$
By \eqref{e-det}, $\lim_{n\to \infty} \sum_{I\in \Lambda_{n,\epsilon}} m([I])=1$. Notice that
\begin{align*}
\rho^n &=\sum_{I\in \Lambda^n}|\det(M_I)|\geq \sum_{I\in \Lambda_{n,\epsilon}} |\det(M_I)|\\
&\geq \sum_{I\in \Lambda_{n,\epsilon}}  e^{-nL_s-n\epsilon} \frac{ m([I]) }{e^{-nh_m (\sigma)+n\epsilon}}\\
&=e^{-n(L_s-h_m(\sigma)+2\epsilon)}\cdot \left( \sum_{I\in \Lambda_{n,\epsilon}}  m([I]) \right).
\end{align*}
Letting $n\to \infty$ and $\epsilon\to 0$, we obtain the desired inequality $h_m(\sigma)\leq L_s+\log \rho<L_s$.  Now the inequality $\dim_{\rm LY}(m, {\bf M})<d$ follows directly from Definition \ref{de-LY}. This proves (i).

Finally we prove (ii).  Since $h_s=0$ and $0\leq h_{\ell-1}-h_\ell\leq (-\lambda_\ell) k_\ell$ for each $\ell$ by Corollary \ref{cor-6.1}, we see that \eqref{e-equiv} holds if and only if  $h_{\ell-1}-h_\ell=(-\lambda_\ell) k_\ell$ for $1\leq \ell\leq j-1$ and $h_{\ell}=0$ for $j\leq \ell\leq s$. By Proposition \ref{pro-7.2},
this is equivalent to that $\dim_{\rm H} \pi_*m=\dim_{\rm LY}(m, {\bf M})$.
\end{proof}

\begin{rem}
\label{rem-6.4}
{\rm
Theorem \ref{cor-1.0} (resp.~Theorem \ref{thm-1.7'}) can be applied to estimate the dimension of slices and projections of certain self-affine sets. To see it, let $K$ a self-affine sets generated by a contracting affine IFS   $\{S_j=M_jx+a_j\}_{j\in \Lambda}$ on $\R^d$. Suppose that there exists an ergodic $m\in \mathcal M_\sigma(\Sigma)$ so that
\begin{equation}
\label{e-as1}
\dim_{\rm H} \pi_*m=\dim_{\rm H}K.
\end{equation}
 Follow the notation in Theorem \ref{cor-1.0} and assume $s\geq 2$.  Since the slicing measures $(\pi_*(m^{\xi_0}_x) )_{V_x^i,y}$ are supported on the slices $K\cap (V_x^i+y)$,    by using Theorem \ref{cor-1.0}(i) and a general inequality in Theorem 2.10.25 of Federer \cite{Federer1969}, we obtain that for $i\in \{1,\ldots, s-1\}$ and $m$-a.e.~$x$,
$$
\dim_{\rm H} K\cap (V_x^i+y)=\sum_{\ell=i}^{s-1}\frac{h_{\ell+1}-h_\ell}{\lambda_{\ell+1}}\quad \mbox{ for
$\big(P_{(V_x^i)^\perp}\pi\big)_*\big(m^{\xi_0}_x\big)$-a.e.~$y\in (V_x^i)^\perp$}
$$
and
\begin{equation}
\label{e-r1}
\dim_{\rm H}\left\{y\in P_{(V_x^i)^\perp}(K):\; \dim_{\rm H} K\cap (V_x^i+y)=\sum_{\ell=i}^{s-1}\frac{h_{\ell+1}-h_\ell}{\lambda_{\ell+1}} \right\}=\sum_{\ell=0}^{i-1}\frac{h_{\ell+1}-h_\ell}{\lambda_{\ell+1}}.
\end{equation}
If in addition to the assumption \eqref{e-as1}, we further assume that
$$\dim_{\rm H} \pi_*m=\dim_{\rm LY}(m, {\bf M}),$$
then
\begin{equation}
\label{e-r2}
\dim_{\rm H} P_{(V_x^i)^\perp}(K)=\min\{\dim (V_x^i)^\perp, \dim_{\rm H}K\} \quad \mbox{ for $m$-a.e.~$x$}.
\end{equation}
Indeed by Remark \ref{rem-q1}, the sum in the right-hand side of \eqref{e-r1} is equal to
$$\min\{\dim(V_x^i)^\perp, \dim_{\rm H}  \pi_*m\},$$ and hence equal to $\min\{\dim(V_x^i)^\perp, \dim_{\rm H} K\}$.  Now \eqref{e-r2} follows from \eqref{e-r1}.
}

\end{rem}

\section{Semi-continuity of entropies and dimensions}
\label{S-8}

In this section, we prove Theorems \ref{thm-1.3}-\ref{thm-1.4}.
Set
\begin{equation}
\label{e-fx}
f(x)=\sum_{n=1}^\infty\|M_{x_0}\cdots M_{x_{n-1}}\|\quad \mbox{ for $x\in \Sigma$}.
\end{equation}

\begin{lem}
\label{lem-7.1}
Let $\eta$ be a Borel probability measure on $\Sigma$ with  $\eta(\{f=\infty\})=0$. Then $\big(\pi_{{\bf a}}\big)_*\eta$ depends continuously on ${\bf a}$, in the sense that   $\big(\pi_{{\bf a}_n}\big)_*\eta$ converges to $\big(\pi_{{\bf a}}\big)_*\eta$ weakly when ${\bf a}_n$ converges to ${\bf a}$.
\end{lem}
\begin{proof}
 For $x\in \Sigma$ with  $f(x)<\infty$,   $\pi_{{\bf a}}(x)$ is well-defined for every ${\bf a}\in \R^{d|\Lambda|}$ and moreover,
\begin{equation}
\label{e-pif}
\|\pi_{{\bf a}}(x)-\pi_{{\bf b}}(x)\|\leq f(x)\|{\bf a}-{\bf b}\|.
\end{equation}
For $N\in \N$, set $A_N:=\{x: f(x)<N\}$. Since $\eta(\{f=\infty\})=0$, it follows that $\eta(A_N)\to 1$ as $N\to \infty$.

Let  $({\bf a}_n)\subset \R^{d|\Lambda|}$ so that $\lim_{n\to \infty}{\bf a}_n={\bf a}$.  For convenience, write $\nu_n= \big(\pi_{{\bf a}_n}\big)_*\eta$ and $\nu= \big(\pi_{{\bf a}}\big)_*\eta$. To show that $\nu_n$  converges weakly to $\nu$,  by the Portmanteau theorem, it suffices to show that  $\limsup_{n\to \infty} \nu_n(F)\leq  \nu(F)$  for any compact set $F\subset \R^d$.

Now fix a compact set $F\subset \R^d$. Let $\epsilon>0$. Take a small $r>0$ so that $\nu(V_r(F))\leq \nu(F)+\epsilon$, where $V_r(F)$ stands for the $r$-neighborhood of $F$. Take a large $N$ so that $\eta(\Sigma \setminus A_N)<\epsilon$. Pick $n_0$ so that $\|{\bf a}_n-{\bf a}\|<r/N$ when $n\geq n_0$.

By \eqref{e-pif},  for $x\in A_N$ and $n\geq n_0$ we have  $\|\pi_{{\bf a_n}}(x)-\pi_{{\bf a}}(x)\|\leq N\|{\bf a_n}-{\bf a}\|<r$. Hence
$A_N\cap \pi_{{\bf a}_n}^{-1}(F)\subset A_N\cap \pi_{{\bf a}}^{-1}(V_r(F))$ for $n\geq n_0$. It follows that for $n\geq n_0$,
\begin{align*}
\nu_n(F)&=\eta(\pi_{{\bf a}_n}^{-1}(F))\\
&\leq \eta(\Sigma\setminus A_N)+ \eta(A_N\cap \pi_{{\bf a}_n}^{-1}(F))\\
 &\leq \epsilon+ \eta(A_N\cap \pi_{{\bf a}}^{-1}(V_r(F)))\\
& \leq \epsilon+ \nu(V_r(F))\\
&\leq \nu(F)+2\epsilon.
\end{align*}
 Hence  $\limsup_{n\to \infty} \nu_n(F)\leq  \nu(F)+2\epsilon$. Letting $\epsilon\to 0$ gives $\limsup_{n\to \infty} \nu_n(F)\leq  \nu(F)$, as desired.
\end{proof}
\begin{proof}[Proof of Theorem \ref{thm-1.3}] We first prove  part (1) of the theorem. This is done by extending an idea of Rapaport \cite[Lemma 8]{Rapaport2015}.

 It is implicitly proved in Proposition \ref{pro-3.1} that $m(\{f=\infty\})=0$, where $f$ is defined as in \eqref{e-fx}.
Let $i\in \{1,\ldots, s\}$ and write $\xi_{i,{\bf a}}$ for  $\xi_i$ so as to emphasize its dependence on  ${\bf a}$.  Since
$$
0=m(\{f=\infty\})=\int m^{\xi_{i,{\bf a}}}_x (\{f=\infty\}) d m(x),
$$
the set $\Delta_{\ba}:=\left\{x\in \Sigma': \; m^{\xi_{i,{\bf a}}}_x (\{f=\infty\})=0\right\}$ has full $m$-measure.

 Noticing that $\xi_0$ is independent of ${\bf a}$, and  $\xi_{i,{\bf a}}$ is a refinement of $\xi_0$ (i.e. any set in $\xi_{i,{\bf a}}$ is a subset of an element in $\xi_0$), we have
\begin{align*}
h_{i, {\bf a}}&=H_m(\P|\xi_{i,{\bf a}})\\
&=\int -\log m_x^{\xi_{i,{\bf a}}}(\P(x))\;dm(x)\\
&=\int \int -\log m_y^{\xi_{i,{\bf a}}}(\P(y))\;dm_x^{\xi_0}(y)\; dm(x)\\
&=\int H_{m_x^{\xi_0}}(\P|\xi_{i,{\bf a}})\; dm(x).
\end{align*}

Fix ${\bf a}_0\in \R^{d|\Lambda|}$. In what follows we show that $h_{i, \ba}$ is upper semi-continuous in $\ba$ at ${\bf a}_0$. Since $\Delta_{{\bf a}_0}$ has full $m$-measure,
$h_{i, \ba}=\int_{\Delta_{{\bf a}_0}} H_{m_x^{\xi_0}}(\P|\xi_{i,{\bf a}})\; dm(x)$. Hence it  is sufficient to show that ${\bf a}\mapsto H_{m_x^{\xi_0}}(\P|\xi_{i,{\bf a}})$ is upper semi-continuous at ${\bf a}_0$ for every $x\in \Delta_{{\bf a}_0}$.  For this purpose, fix $x\in \Delta_{{\bf a}_0}$ and write $C=\xi_0(x)$, $W=V_x^i$ and $m_C=m_x^{\xi_0}$. Then by the definition of $\xi_{i,{\bf a}}$,  $$H_{m_x^{\xi_0}}(\P|\xi_{i,{\bf a}})=H_{m_C}(\P|\pi_{\bf a}^{-1}\circ P_{W^\perp}^{-1}(\B(W^\perp))).$$
 Following the proof  of \cite[Lemma 8.5]{Walters1982} or \cite[Lemma 8]{Rapaport2015}  with minor changes, we can construct a sequence $(\beta_n)$ of finite Borel partitions of $W^\perp$ such that (i) $\sigma(\beta_n)\uparrow \B(W^{\perp})$ and (ii) $m_C\circ \pi^{-1}_{{\bf a}_0}(P_{W^\perp}^{-1}(\partial B))=0$ for any $B\in \bigcup_n \beta_n$. Since $\sigma(\beta_n)\uparrow \B(W^{\perp})$,
\begin{align*}
H_{m_C}(\P|\pi_{\bf a}^{-1}\circ P_{W^\perp}^{-1}(\B(W^{\perp}))&=\lim_{n\to \infty}H_{m_C}(\P|\pi_{\bf a}^{-1}\circ P_{W^\perp}^{-1}(\sigma(\beta_n))\\
&=\lim_{n\to \infty}\left[ \sum_{A\in \P}\sum_{B\in \beta_n}  u\left((m_C|_{A})\circ \pi_{\bf a}^{-1} (P_{W^\perp}^{-1}(B))\right) \right.\\
&\qquad\qquad \left. -\sum_{B\in \beta_n} u\left(m_C\circ \pi_{\bf a}^{-1} (P_{W^\perp}^{-1}(B))\right) \right],
\end{align*}
where $u(z):=-z\log z$  and $m_C|_{A}(E)=m_C(A\cap E)$. Since $x\in \Delta_{{\bf a}_0}$,   $m_C(\{f=\infty\})=0$. By Lemma \ref{lem-7.1},  the measures $(\pi_{\bf a})_*(m_C)$ and $(\pi_{\bf a})_*(m_C|_{A})$ ($A\in \P$) depend continuously on ${\bf a}$; and so do $ \big(P_{W^{\perp}}\pi_{\bf a}\big)_*(m_C)$ and $\big(P_{W^{\perp}}\pi_{\bf a}\big)_*(m_C|_{A})$.
 For $A\in \P$ and  $B\in \bigcup_n\beta_n$, since $m_C\circ \pi^{-1}_{{\bf a}_0}(P_{W^\perp}^{-1}(\partial B))=0$, we have also $(m_C|_A)\circ \pi^{-1}_{{\bf a}_0}(P_{W^\perp}^{-1}(\partial B))=0$;
 it follows that, as functions of ${\bf a}$,   $u\left(m_C\circ \pi_{\bf a}^{-1} (P_{W^\perp}^{-1}(B))\right)$ and $u\left((m_C|_{A})\circ \pi_{\bf a}^{-1} (P_{W^\perp}^{-1}(B))\right)$ ($A\in \P$) are continuous at ${\bf a}_0$, and so is $H_{m_C}(\P|\pi_{\bf a}^{-1}\circ P_{W^\perp}^{-1}(\sigma(\beta_n))$. Hence
${\bf a}\mapsto H_{m_C}(\P|\pi_{\bf a}^{-1}\circ P_{W^\perp}^{-1}(\B(W^{\perp}))$ is upper semi-continuous at ${\bf a}_0$, as desired. This proves the upper semi-continuity of $h_{i,\ba}$.

Next we prove the lower semi-continuity of the mapping ${\bf a}\mapsto \dim_{{\rm H}}( (\pi_{\bf a})_*m)$.
By Theorem \ref{thm-1.1}, we have
\begin{equation}
\label{e-uppc}
\dim_{\rm H}( (\pi_{\bf a})_*m)=\sum_{i=0}^{s} t_i h_{i, {\bf a}},
\end{equation}
where $t_0=-\frac{1}{\lambda_1}$ and $t_i=
\frac{1}{\lambda_i}- \frac{1} {\lambda_{i+1}} $
for $i=1,\ldots, s$, with convention $\lambda_{s+1}:=-\infty$. Notice that $t_0>0$,   $t_i\leq 0$ for $1\leq i\leq s$ and moreover, $h_{0, {\bf a}}\equiv h_\sigma(m)$.  By part (1), $h_{1, {\bf a}}, \ldots, h_{s,  {\bf a}}$ are upper semi-continuous in ${\bf a}$.  Hence by \eqref{e-uppc},  $\dim_{\rm H}( (\pi_{\bf a})_*m)$ is lower semi-continuous in ${\bf a}$.
\end{proof}

\begin{rem}
\label{rem-8.1}
{\rm 
Theorem \ref{thm-1.3} can be further extended. For  given $m$ and  ${\bf M}=(M_j)_{j\in \Lambda}$, let $\mathcal S_{{\bf r}, {\bf a}}$ denote the IFS $\{r_jM_jx+a_j\}_{j\in \Lambda}$ where ${\bf r}=(r_j)_{j\in \Lambda}\in (\R\backslash \{0\})^\Lambda$ so that $S_{{\bf r}, {\bf a}}$ is average contracting with respect to $m$.     Notice that the Oseledets subspaces with respect to $m$ and $(r_jM_j)_{j\in \Lambda}$ are independent of ${\bf r}$. A slight modification of the above proof  establishes the upper semi-continuity of $({\bf r}, {\bf a})\mapsto h_{i,{\bf r}, {\bf a}}$ and the lower semi-continuity of
$({\bf r}, {\bf a})\mapsto  \dim_{\rm H}( (\pi_{{\bf r}, {\bf a}})_*m)$.

Similarly, for  given $m$ let $\mathcal S_{{\bf r}, {\bf O}, {\bf a}}$ denote the IFS $\{r_jO_jx+a_j\}_{j\in \Lambda}$ of similitudes, where ${\bf r}=(r_j)_{j\in \Lambda}\in (\R\backslash \{0\})^\Lambda$, ${\bf O}=(O_j)_{j\in \Lambda}\in O(d)^\Lambda$, ${\bf a}=(a_j)_{j\in \Lambda}\in \R^{d|\Lambda|}$ so that $S_{{\bf r}, {\bf O},{\bf a}}$ is average contracting with respect to $m$. Then the mapping $({\bf r}, {\bf O}, {\bf a})\mapsto  \dim_{\rm H}( (\pi_{{\bf r}, {\bf O}, {\bf a}})_*m)$ is lower semi-continuous.
}
\end{rem}
\begin{proof}[Proof of Theorem \ref{cor-1.7}]
We first prove (i). Let $m$ be an ergodic $\sigma$-invariant measure $m$ on $\Sigma$.
   For $n\in \N$, set
$$
\Omega_n:=\left\{\ba \in \R^{d|\Lambda|}:\; \dim_{\rm H} ((\pi_{\bf a})_*m)\leq \min(d, \dim_{\rm LY} (m, {\bf M}))-\frac{1}{n}\right\}.
$$
Since $\dim_{\rm H} ( (\pi_{\bf a})_*m)$ is lower semi-continuous in ${\bf a}$ by Theorem \ref{thm-1.3}, $\Omega_n$ is closed for each $n$. Meanwhile, it was proved in \cite{JordanPollicottSimon2007} that $\dim_{\rm H} ( (\pi_{\bf a})_*m)=\min(d, \dim_{\rm LY} (m, {\bf M}))$ for $\mathcal L^{d|\Lambda|}$-a.e.~$\ba$. Hence for each $n$,  $\Omega_n$ is a closed set of zero Lebesgue measure, so it is nowhere dense. This is enough to conclude (i).

Next we prove (ii).
It was shown by K\"aenm\"aki \cite{Kaenmaki2004} that there exists an ergodic $\sigma$-invariant measure $\eta$ on $\Sigma$ such that
$\dim_{\rm LY} (\eta, {\bf M})=\dim_{\rm AFF} ({\bf M})$.  Fix such $\eta$. Note that for each $\ba$,
$$\dim_{\rm H} ((\pi_{\bf a})_*\eta)\leq \dim_{\rm H} K({\bf M}, \ba)\leq \min(d, \dim_{\rm AFF} ({\bf M})).$$
It implies that
\begin{align*}
&\left\{\ba\in \R^{d|\Lambda|}:\; \dim_{\rm H} K({\bf M}, \ba)\neq \min(d, \dim_{\rm AFF}({\bf M}))\right\}\\
&\mbox{}\quad\subset \left\{\ba\in \R^{d|\Lambda|}:\; \dim_{\rm H} ((\pi_{\bf a})_*\eta)\neq \min(d, \dim_{\rm LY} (\eta, {\bf M}))\right\}.
\end{align*}
Now (ii) follows from (i).
\end{proof}

To prove Theorem \ref{thm-1.4} we need the following.

\begin{lem}[{\cite[Lemma 22]{Rapaport2015}}]
\label{lem-rap}
Let $\mu$ be a probability Borel measure on $\R^d$ and $1\leq k< d$. Then the following statements hold.
\begin{itemize}
\item[(i)] If $\dim_{\rm H}\mu\leq k$ then for $0<t\leq  \dim_{\rm H}\mu$,
$$
\dim_{\rm H} \{W\in G(d, k):\; \dim_{\rm H} ((P_W)_*\mu) <t\}\leq k(d-k-1)+t.
$$
\item[(ii)]
 If $\dim_{\rm H}\mu\geq k$ then for $\dim_{\rm H}\mu-k(d-k)<t\leq  k$,
$$
\dim_{\rm H} \{W\in G(d, k):\; \dim_{\rm H} ((P_W)_*\mu)<t\}\leq k(d-k)+t-\dim_{\rm H}\mu.
$$
\end{itemize}
 \end{lem}
\begin{proof}[Proof of Theorem \ref{thm-1.4}]  The proof is  mainly adapted from \cite{Rapaport2015}. For the convenience of the reader, we include  the details.
 Write $\mu= \pi_*m$.  Since ${\mathcal S}$ satisfies the strong separation condition, by Proposition \ref{pro-7.3} we have  $h_s=0$, $h_m(\sigma)<\sum_{\ell=1}^s (-\lambda_\ell) k_\ell$ and  $\dim_{\rm LY}(m, {\bf M})<d$.  Let $i$ be the unique element in $\{1,\ldots, s\}$ so that $d_{i-1}  \leq \dim_{\rm LY} (m, {\bf M})<d_i$.
(Recall that $d_0=0$ and $d_j=k_1+\cdots+k_j$ for $j\geq 1$.) By Definition~\ref{de-LY}, we have $h_m(\sigma)\in [L_{i-1}, L_i)$ where $L_0:=0$ and $L_j:=-\sum_{\ell=1}^{j}\lambda_\ell k_\ell$ for $j\geq 1$. Below we prove the equality $\dim_{\rm H}\mu=\dim_{\rm LY} (m, {\bf M})$ under the assumption that one of the scenarios (a), (b), (c) occurs.

 We first consider the scenario  (a). In this case, $s=1$ and by Theorem \ref{thm-1.1}, $$\dim_{\rm H} \mu=\frac{h_1-h_0}{\lambda_1}=-\frac{h_m(\sigma)}{\lambda_1}= \dim_{\rm LY}(m, {\bf M}).$$

Next we consider the scenario (b). In this case, $i=s$ and so $h_m(\sigma)\in [L_{s-1}, L_s)$.
To show that $\dim_{\rm H}\pi_*m=\dim_{\rm LY}(m, {\bf M})$, it suffices to show that
 \begin{equation}
 \label{e-LYdim}
  h_{s-1}= h_m(\sigma)+k_1\lambda_1+\cdots +k_{s-1}\lambda_{s-1}.
 \end{equation}
Indeed if \eqref{e-LYdim} holds, then $h_0-h_{s-1}=\sum_{\ell=1}^{s-1} (-\lambda_\ell)k_\ell$, which forces that $h_{\ell-1}-h_\ell=(-\lambda_\ell)k_\ell$ for  $1\leq \ell\leq s-1$ (recalling that $h_{\ell-1}-h_\ell\leq (-\lambda_\ell)k_\ell$ for all $1\leq \ell\leq s$ by Corollary \ref{cor-6.1}); hence
$$\sum_{\ell=1}^{s-1}\frac{h_\ell-h_{\ell-1}}{\lambda_\ell}=d_{s-1},$$
so  \eqref{e-equiv} holds for $j=s$,  then by Proposition \ref{pro-7.3},  we obtain that $\dim_{\rm H}\mu=\dim_{\rm LY}(m, {\bf M})$.

 To show \eqref{e-LYdim} we first prove that  \begin{equation}
 \label{e-LYdim'}
  h_{s-1}\geq h_m(\sigma)+k_1\lambda_1+\cdots +k_{s-1}\lambda_{s-1}.
 \end{equation}   To see this, replacing  ${\mathcal S}$ by one of its iterations if necessary, we may assume that $\|M_j\|<1/2$ for all $j\in \Lambda$. By Theorem 1.9 in \cite{JordanPollicottSimon2007}, for ${\mathcal L}^{d|\Lambda|}$-a.e.~$\ba\in \R^{d|\Lambda|}$, $$\dim_{\rm H}  ((\pi_{\ba})_*m)=\dim_{\rm LY}(m, {\bf M}).$$  Hence by Proposition \ref{pro-7.2},
for ${\mathcal L}^{d|\Lambda|}$-a.e.~$\ba\in \R^{d|\Lambda|}$,  $$h_{s-1,\ba}=h_m(\sigma)+k_1\lambda_1+\cdots +k_{s-1}\lambda_{s-1},$$
 here and in the next sentence, we write $h_{s-1,\ba}=h_{s-1}$ to indicate its dependence on $\ba$.
Since $h_{s-1,\ba}$ is upper semi-continuous in {\ba} by Theorem \ref{thm-1.3}, it follows that  $h_{s-1,\ba}\geq h_m(\sigma)+k_1\lambda_1+\cdots +k_{s-1}\lambda_{s-1}$ for all $\ba\in \R^{d|\Lambda|}$. This proves \eqref{e-LYdim'}.

Now suppose on the contrary that  \eqref{e-LYdim} does not hold. Then by \eqref{e-LYdim'}, there exists $\delta>0$ such that  $h_{s-1}= h_m(\sigma)+k_1\lambda_1+\cdots +k_{s-1}\lambda_{s-1}+\delta$. By Theorem \ref{cor-1.0} (iii), for $(\Pi_{s-1})_*m$-a.e.~$W\in G(d, d-d_{s-1})$,
\begin{align*}
\dim_{\rm H}  ((P_{W^\perp})_*\mu)&\leq \sum_{\ell=1}^{s-1} \frac{h_\ell-h_{\ell-1}}{\lambda_{\ell}}\\
&=\dim_{\rm H}\mu-\frac{h_s-h_{s-1}}{\lambda_s}\\
&=\dim_{\rm H}\mu-\frac{h_{s-1}}{(-\lambda_s)}\\
&= \dim_{\rm H}\mu +d_{s-1}-\dim_{\rm LY}(m, {\bf M})-\delta/(-\lambda_s),
\end{align*}
where in the last equality, we use the fact that $$\dim_{\rm LY}(m, {\bf M})=d_{s-1}+\frac{h_0-L_{s-1}}{(-\lambda_s)}=d_{s-1}+\frac{h_{s-1}-\delta}{(-\lambda_s)}.$$
Let ${\bf Y}$ denote the set of $W\in G(d, d-d_{s-1})$ such that $$\dim_{\rm H}  ((P_{W^\perp})_*\mu)\leq \dim_{\rm H}\mu +d_{s-1}-\dim_{\rm LY}(m, {\bf M})-{\delta}/{(-\lambda_s)}.$$ Then $m\circ (\Pi_{s-1})^{-1}({\bf Y})=1$,  so by    \eqref{e-condLY},
\begin{equation}
\label{e-lower}
\begin{split}\dim_{\rm H} {\bf Y}&\geq \dim_{\rm H}^*( (\Pi_{s-1})_*m)\\
&\geq d_{s-1}(d-d_{s-1})+d_{s-1}-\dim_{\rm LY}(m, {\bf M}).
\end{split}
\end{equation}
On the other hand,  we can get an upper bound estimate for $\dim_{\rm H} {\bf Y}$ by using Lemma \ref{lem-rap}. Indeed,  if $\dim_{\rm H}\mu\leq d_{s-1}$,  then by Lemma \ref{lem-rap}(i) applied to $k=d_{s-1}$ and $t=\dim_{\rm H}\mu +d_{s-1}-\dim_{\rm LY}(m, {\bf M})-{\delta}/{(-\lambda_s)}$, we see that
   \begin{align*}
\dim_{\rm H} {\bf Y}&\leq d_{s-1}(d-d_{s-1})+\dim_{\rm H}\mu-\dim_{\rm LY}(m, {\bf M})-{\delta}/{(-\lambda_s)}\\
&\leq d_{s-1}(d-d_{s-1})+d_{s-1}-\dim_{\rm LY}(m, {\bf M})-{\delta}/{(-\lambda_s)};
\end{align*}
Conversely if $\dim_{\rm H}\mu>d_{s-1}$, then by Lemma \ref{lem-rap}(ii) applied to $k=d_{s-1}$ and $t=\dim_{\rm H}\mu +d_{s-1}-\dim_{\rm LY}(m, {\bf M})-{\delta}/{(-\lambda_s)}$, we get the same upper bound for $\dim_{\rm H} {\bf Y}$, which  contradicts with \eqref{e-lower}. This proves \eqref{e-LYdim}.

Finally we consider the scenario (c).  In this case, $h_m(\sigma)\in [L_{i-1}, L_i)$. Clearly the assumptions \eqref{e-al2}-\eqref{e-al1} imply that $$d_{i-1}\leq \dim_{\rm H}\mu \leq \dim_{\rm LY}(m, {\bf M}) \leq d_i.$$
 To prove $\dim_{\rm H}\mu=\dim_{\rm LY}(m, {\bf M})$,  by Proposition \ref{pro-7.3} it suffices to prove that
$\sum_{\ell=1}^{i-1} \frac{h_\ell-h_{\ell-1}}{\lambda_\ell}=d_{i-1}$ and $\sum_{\ell=i+1}^s \frac{h_\ell-h_{\ell-1}}{\lambda_\ell}=0$.
As $d_0=0$, the first equality holds automatically when $i=1$.

Now we first prove that $\sum_{\ell=1}^{i-1} \frac{h_\ell-h_{\ell-1}}{\lambda_\ell}=d_{i-1}$.  To avoid triviality, we assume that $i\geq 2$.  For $n\in \N$, let ${\bf X}_n$ denote the set of $W\in G(d, d-d_{i-1})$ so that $\dim_{\rm H} ((P_{W^\perp})_*\mu)  < d_{i-1}-{1}/{n}$.  By Lemma \ref{lem-rap}(ii) applied to $k=d_{i-1}$ and $t=d_{i-1}-1/n$,
\begin{equation*}
\begin{split}
\dim_{\rm H} {\bf X}_n &\leq d_{i-1}(d-d_{i-1})+d_{i-1}-({1}/{n})-\dim_{\rm H}\mu\\
&< \dim_{\rm H}^*((\Pi_{i-1})_*m)\qquad(\mbox{by \eqref{e-al1}}).
\end{split}
\end{equation*}
It follows that $m\circ (\Pi_{i-1})^{-1}(X_n)<1$ and hence $\dim_{\rm H}((P_{W^\perp})_*\mu)> d_{i-1}- {1}/{n}$ on a set of positive $(\Pi_{i-1})_*m$-measure. However by Theorem \ref{cor-1.0}(iii),
\begin{equation}
\label{e-915}
\dim_{\rm H} ((P_{W^\perp})_*\mu)\leq \sum_{\ell=1}^{i-1} \frac{h_\ell-h_{\ell-1}}{\lambda_\ell} \quad \mbox{  for
$(\Pi_{i-1})_*m$-a.e.~$W$}.
\end{equation}
It follows that $\sum_{\ell=1}^{i-1} \frac{h_\ell-h_{\ell-1}}{\lambda_\ell}\geq d_{i-1}-{1}/{n}$.  As $n$ is arbitrary, we obtain that $\sum_{\ell=1}^{i-1} \frac{h_\ell-h_{\ell-1}}{\lambda_\ell}\geq d_{i-1}$. Since $h_{\ell-1}-h_\ell\leq (-\lambda_\ell)k_\ell$ for each $\ell$ by Corollary \ref{cor-6.1}, we have $\sum_{\ell=1}^{i-1} \frac{h_\ell-h_{\ell-1}}{\lambda_\ell}= d_{i-1}$, as desired.

Next we prove that  $\sum_{\ell=i+1}^{s} \frac{h_\ell-h_{\ell-1}}{\lambda_\ell}= 0$. For $n\in \N$, let ${\bf Z}_n$ denote the set of $W\in G(d, d-d_{i})$ so that $\dim_{\rm H}((P_{W^\perp})_*\mu) < \dim_{\rm H}\mu- {1}/{n}$. By Lemma \ref{lem-rap}(i) applied to $k=d_i$ and $t=\dim_{\rm H}\mu-{1}/{n}$,
\begin{equation*}
\begin{split}
\dim_{\rm H} {\bf Z}_n &\leq d_{i}(d-d_{i})-d_{i}+\dim_{\rm H}\mu-({1}/{n})\\
&\leq d_{i}(d-d_{i})-d_{i}+ \dim_{\rm LY}(m, {\bf M})-({1}/{n})\\
&< \dim_{\rm H}^*((\Pi_{i})_*m)\qquad(\mbox{by \eqref{e-al2}}).
\end{split}
\end{equation*}
Hence $m\circ (\Pi_{i})^{-1}({\bf Z}_n)<1$ and so  $\dim_{\rm H} ((P_{W^\perp})_*\mu)> \dim_{\rm H}\mu- {1}/{n}$ on a set of positive $(\Pi_{i})_*m$-measure. This combining with  \eqref{e-915} (in which we replace $i-1$ by $i$) yields that $\sum_{\ell=1}^{i} \frac{h_\ell-h_{\ell-1}}{\lambda_\ell}\geq \dim_{\rm H}\mu-{1}/{n}$.  Letting $n\to \infty$ gives  $\sum_{\ell=1}^{i} \frac{h_\ell-h_{\ell-1}}{\lambda_\ell}\geq \dim_{\rm H}\mu$,  which, together with \eqref{ly-dim},  implies that   $\sum_{\ell=i+1}^{s} \frac{h_\ell-h_{\ell-1}}{\lambda_\ell}= 0$.  This completes the proof of the theorem.
\end{proof}
\medskip

{\noindent \bf  Acknowledgements}.
The author is indebted to Julien Barral, Xiong Jin, Fran\c{c}ois Ledrappier and Ariel Rapaport  for some helpful comments, and to Yufeng Wu for catching many typos.  He  thanks the anonymous referees  for many  suggestions that led to the improvement of the
paper.
This research was partially supported by a HKRGC
GRF grant  and the Direct Grant for Research in CUHK.


\begin{thebibliography}{10}

\bibitem{Baranski2007}
K.~Bara\'{n}ski.
\newblock Hausdorff dimension of the limit sets of some planar geometric
  constructions.
\newblock {\em Adv. Math.}, 210(1):215--245, 2007.

\bibitem{Barany2015}
B.~B\'ar\'any.
\newblock On the {L}edrappier-{Y}oung formula for self-affine measures.
\newblock {\em Math. Proc. Cambridge Philos. Soc.}, 159(3):405--432, 2015.

\bibitem{BaranyHochmanRapaport2017}
B.~B\'{a}r\'{a}ny, M.~Hochman, and A.~Rapaport.
\newblock Hausdorff dimension of planar self-affine sets and measures.
\newblock {\em Invent. Math.},  216(3):601--659, 2019.


\bibitem{BaranyKaenmaki2015}
B.~B\'{a}r\'{a}ny and A.~K\"{a}enm\"{a}ki.
\newblock Ledrappier-{Y}oung formula and exact dimensionality of self-affine
  measures.
\newblock {\em Adv. Math.}, 318:88--129, 2017.

\bibitem{BKK2018}
B.~B\'{a}r\'{a}ny, A.~K\"{a}enm\"{a}ki, and H.~Koivusalo.
\newblock Dimension of self-affine sets for fixed translation vectors.
\newblock {\em J. Lond. Math. Soc. (2)}, 98(1):223--252, 2018.

\bibitem{BaranyRams2018}
B.~B\'{a}r\'{a}ny and M.~Rams.
\newblock Dimension maximizing measures for self-affine systems.
\newblock {\em Trans. Amer. Math. Soc.}, 370(1):553--576, 2018.

\bibitem{BaranyRamsSimon2016}
B.~B\'{a}r\'{a}ny, M.~Rams, and K.~Simon.
\newblock On the dimension of self-affine sets and measures with overlaps.
\newblock {\em Proc. Amer. Math. Soc.}, 144(10):4427--4440, 2016.

\bibitem{BRS2018}
B.~B\'ar\'any, M.~Rams, and K.~Simon.
\newblock Dimension of the repeller for a piecewise expanding affine map.
\newblock {\em preprint, arXiv:1803.03788}, 2018.

\bibitem{BarralFeng2013}
J.~Barral and D.-J. Feng.
\newblock Multifractal formalism for almost all self-affine measures.
\newblock {\em Comm. Math. Phys.}, 318(2):473--504, 2013.

\bibitem{BarreiraPesinSchmeling1999}
L.~Barreira, Y.~Pesin, and J.~Schmeling.
\newblock Dimension and product structure of hyperbolic measures.
\newblock {\em Ann. of Math. (2)}, 149(3):755--783, 1999.

\bibitem{Bedford1984}
T.~Bedford.
\newblock Crinkly curves, markov partitions and box dimensions in self-similar
  sets.
\newblock PhD Thesis, The University of Warwick, 1984.

\bibitem{Bedford1991}
T.~Bedford.
\newblock Applications of dynamical systems theory to fractals---a study of
  cookie-cutter {C}antor sets.
\newblock In {\em Fractal geometry and analysis ({M}ontreal, {PQ}, 1989)},
  volume 346 of {\em NATO Adv. Sci. Inst. Ser. C Math. Phys. Sci.}, pages
  1--44. Kluwer Acad. Publ., Dordrecht, 1991.

\bibitem{BochiMorris2018}
J.~Bochi and I.~D. Morris.
\newblock Equilibrium states of generalised singular value potentials and
  applications to affine iterated function systems.
\newblock {\em Geom. Funct. Anal.}, 28(4):995--1028, 2018.



\bibitem{BougerolPicard1992}
P.~Bougerol and N.~Picard.
\newblock Strict stationarity of generalized autoregressive processes.
\newblock {\em Ann. Probab.}, 20(4):1714--1730, 1992.

\bibitem{Bowen1975}
R.~Bowen.
\newblock {\em  Equilibrium states and the ergodic theory of Anosov diffeomorphisms. }
\newblock Lecture Notes in Mathematics, Vol. 470. Springer-Verlag, Berlin-New York, 1975.

\bibitem{Brandt1986}
A.~Brandt.
\newblock The stochastic equation {$Y_{n+1}=A_nY_n+B_n$} with stationary
  coefficients.
\newblock {\em Adv. in Appl. Probab.}, 18(1):211--220, 1986.

\bibitem{CaoFengHuang2008}
Y.-L. Cao, D.-J. Feng, and W.~Huang.
\newblock The thermodynamic formalism for sub-additive potentials.
\newblock {\em Discrete Contin. Dyn. Syst.}, 20(3):639--657, 2008.

\bibitem{DasSimmons2017}
T.~Das and D.~Simmons.
\newblock The {H}ausdorff and dynamical dimensions of self-affine sponges: a
  dimension gap result.
\newblock {\em Invent. Math.}, 210(1):85--134, 2017.

\bibitem{EckmannRuelle1985}
J.-P. Eckmann and D.~Ruelle.
\newblock Ergodic theory of chaos and strange attractors.
\newblock {\em Rev. Modern Phys.}, 57(3, part 1):617--656, 1985.

\bibitem{EinsiedlerWard2011}
M.~Einsiedler and T.~Ward.
\newblock {\em Ergodic theory with a view towards number theory}, volume 259 of
  {\em Graduate Texts in Mathematics}.
\newblock Springer-Verlag London, Ltd., London, 2011.

\bibitem{Falconer1988}
K.~J. Falconer.
\newblock The {H}ausdorff dimension of self-affine fractals.
\newblock {\em Math. Proc. Cambridge Philos. Soc.}, 103(2):339--350, 1988.

\bibitem{Falconer1997}
K.~J. Falconer.
\newblock {\em Techniques in fractal geometry}.
\newblock John Wiley \& Sons, Ltd., Chichester, 1997.

\bibitem{Falconer2003}
K.~J. Falconer.
\newblock {\em Fractal geometry--Mathematical foundations and applications}.
\newblock John Wiley \& Sons, Inc., Hoboken, NJ, second edition, 2003.

\bibitem{Falconer2013}
K.~J. Falconer.
\newblock Dimensions of self-affine sets: a survey.
\newblock In {\em Further developments in fractals and related fields}, Trends
  Math., pages 115--134. Birkh\"{a}user/Springer, New York, 2013.

\bibitem{FalconerJin2014}
K.~J. Falconer and X.~Jin.
\newblock Exact dimensionality and projections of random self-similar measures
  and sets.
\newblock {\em J. Lond. Math. Soc. (2)}, 90(2):388--412, 2014.

\bibitem{FalconerKempton2018}
K.~J. Falconer and T.~Kempton.
\newblock Planar self-affine sets with equal {H}ausdorff, box and affinity
  dimensions.
\newblock {\em Ergodic Theory Dynam. Systems}, 38(4):1369--1388, 2018.

\bibitem{FanLauRao2002}
A.-H. Fan, K.-S. Lau, and H.~Rao.
\newblock Relationships between different dimensions of a measure.
\newblock {\em Monatsh. Math.}, 135(3):191--201, 2002.

\bibitem{Federer1969}
H.~Federer.
\newblock {\em Geometric measure theory}.
\newblock Springer-Verlag New York Inc., New York, 1969.

\bibitem{FengHu2009}
D.-J. Feng and H.~Hu.
\newblock Dimension theory of iterated function systems.
\newblock {\em Comm. Pure Appl. Math.}, 62(11):1435--1500, 2009.

\bibitem{FengKaenmaki2011}
D.-J. Feng and A.~K\"aenm\"aki.
\newblock Equilibrium states of the pressure function for products of matrices.
\newblock {\em Discrete Contin. Dyn. Syst.}, 30(3):699--708, 2011.

\bibitem{FengShmerkin2014}
D.-J. Feng and P.~Shmerkin.
\newblock Non-conformal repellers and the continuity of pressure for matrix
  cocycles.
\newblock {\em Geom. Funct. Anal.}, 24(4):1101--1128, 2014.

\bibitem{FraserJordanJurga2017}
J.~Fraser, T.~Jordan, and N.~Jurga.
\newblock Dimensions of equilibrium measures on a class of planar self-affine
  sets.
 \newblock {\em J. Fractal Geom.},  7(1):87-111,  2020.

\bibitem{FroylandLloydQuas2010}
G.~Froyland, S.~Lloyd, and A.~Quas.
\newblock Coherent structures and isolated spectrum for {P}erron-{F}robenius
  cocycles.
\newblock {\em Ergodic Theory Dynam. Systems}, 30(3):729--756, 2010.

\bibitem{Furstenberg2008}
H.~Furstenberg.
\newblock Ergodic fractal measures and dimension conservation.
\newblock {\em Ergodic Theory Dynam. Systems}, 28(2):405--422, 2008.

\bibitem{FurstenbergKesten1960}
H.~Furstenberg and H.~Kesten.
\newblock Products of random matrices.
\newblock {\em Ann. Math. Statist.}, 31:457--469, 1960.

\bibitem{GatzourasLalley1992}
D.~Gatzouras and S.~P. Lalley.
\newblock Hausdorff and box dimensions of certain self-affine fractals.
\newblock {\em Indiana Univ. Math. J.}, 41(2):533--568, 1992.

\bibitem{HochmanRapaport2019}
M.~Hochman and A.~Rapaport.
\newblock Hausdorff dimension of planar self-affine sets and measures with overlaps.
\newblock {\em preprint, arXiv:1904.09812}, 2019.

\bibitem{HochmanShmerkin2012}
M.~Hochman and P.~Shmerkin.
\newblock Local entropy averages and projections of fractal measures.
\newblock {\em Ann. of Math. (2)}, 175(3):1001--1059, 2012.

\bibitem{HochmanSolomyak2017}
M.~Hochman and B.~Solomyak.
\newblock On the dimension of {F}urstenberg measure for {$SL_2(\Bbb R)$} random
  matrix products.
\newblock {\em Invent. Math.}, 210(3):815--875, 2017.

\bibitem{Hutchinson1981}
J.~E. Hutchinson.
\newblock Fractals and self-similarity.
\newblock {\em Indiana Univ. Math. J.}, 30(5):713--747, 1981.

\bibitem{Jordan2011}
T.~Jordan.
\newblock Unpublished note.
\newblock 2011.

\bibitem{JordanPollicottSimon2007}
T.~Jordan, M.~Pollicott, and K.~Simon.
\newblock Hausdorff dimension for randomly perturbed self affine attractors.
\newblock {\em Comm. Math. Phys.}, 270(2):519--544, 2007.

\bibitem{Kaenmaki2004}
A.~K\"aenm\"aki.
\newblock On natural invariant measures on generalised iterated function
  systems.
\newblock {\em Ann. Acad. Sci. Fenn. Math.}, 29(2):419--458, 2004.

\bibitem{KenyonPeres1996}
R.~Kenyon and Y.~Peres.
\newblock Measures of full dimension on affine-invariant sets.
\newblock {\em Ergodic Theory Dynam. Systems}, 16(2):307--323, 1996.

\bibitem{Kingman1968}
J.~F.~C. Kingman.
\newblock The ergodic theory of subadditive stochastic processes.
\newblock {\em J. Roy. Statist. Soc. Ser. B}, 30:499--510, 1968.

\bibitem{LedrappierYoung1985}
F.~Ledrappier and L.-S. Young.
\newblock The metric entropy of diffeomorphisms. {I}. {C}haracterization of
  measures satisfying {P}esin's entropy formula. {II}. {R}elations between
  entropy, exponents and dimension.
\newblock {\em Ann. of Math. (2)}, 122(3):509--539; 540--574, 1985.

\bibitem{LiuQian1995}
P.-D.~Liu and  Q.~Min.
\newblock {\em Smooth ergodic theory of random dynamical systems.}
\newblock Lecture Notes in Mathematics, 1606. Springer-Verlag, Berlin, 1995.

\bibitem{Mane1987}
R.~Ma\~n\'e.
\newblock {\em Ergodic theory and differentiable dynamics}.
\newblock Springer-Verlag, Berlin, 1987.

\bibitem{Maker1940}
P.~T. Maker.
\newblock The ergodic theorem for a sequence of functions.
\newblock {\em Duke Math. J.}, 6:27--30, 1940.

\bibitem{Mattila1995}
P.~Mattila.
\newblock {\em Geometry of sets and measures in {E}uclidean spaces}.
\newblock Cambridge University Press, Cambridge, 1995.

\bibitem{McMullen1984}
C.~McMullen.
\newblock The {H}ausdorff dimension of general {S}ierpi\'nski carpets.
\newblock {\em Nagoya Math. J.}, 96:1--9, 1984.

\bibitem{MihailescuUrbanski2016}
E.~Mihailescu and M.~Urba\'{n}ski.
\newblock Random countable iterated function systems with overlaps and
  applications.
\newblock {\em Adv. Math.}, 298:726--758, 2016.

\bibitem{Miller2008} B.~Miller.
\newblock The existence of measures of a given cocycle. {I}. {A}tomless, ergodic $\sigma$-finite measures.
\newblock {\em Ergodic Theory Dynam. Systems},  28(5):1599--1613, 2008.

\bibitem{MorrisShmerkin2018}
I.~Morris and P.~Shmerkin.
\newblock On equality of {H}ausdorff and affinity dimensions, via self-affine
  measures on positive subsystems.
\newblock {\em Trans. Amer. Math. Soc.}, 371:1547--1582, 2019.

\bibitem{Oseledets1968}
V.~I. Oseledec.
\newblock A multiplicative ergodic theorem. {L}iapunov characteristic numbers
  for dynamical systems.
\newblock {\em Trans. Moscow. Math. Soc.}, 19:197--231, 1968.

\bibitem{Parry1969}
W.~Parry.
\newblock {\em Entropy and generators in ergodic theory}.
\newblock W. A. Benjamin, Inc., New York-Amsterdam, 1969.

\bibitem{Parry1981}
W.~Parry.
\newblock {\em Topics in ergodic theory}.
\newblock Cambridge University Press, Cambridge-New York, 1981.

\bibitem{Patzschke1997}
N.~Patzschke.
\newblock Self-conformal multifractal measures.
\newblock {\em Adv. in Appl. Math.}, 19(4):486--513, 1997.

\bibitem{PeresSolomyak2000}
Y.~Peres and B.~Solomyak.
\newblock Existence of {$L^q$} dimensions and entropy dimension for
  self-conformal measures.
\newblock {\em Indiana Univ. Math. J.}, 49(4):1603--1621, 2000.

\bibitem{Petersen1983}
K.~Petersen.
\newblock {\em Ergodic theory}.
\newblock Cambridge University Press, Cambridge, 1983.

\bibitem{PrzytyckiUrbanski1989}
F.~Przytycki and M.~Urba\'{n}ski.
\newblock On the {H}ausdorff dimension of some fractal sets.
\newblock {\em Studia Math.}, 93(2):155--186, 1989.


\bibitem{QianXie2008}
M.~Qian and J.-S. Xie.
\newblock Entropy formula for endomorphisms: relations between entropy,
  exponents and dimension.
\newblock {\em Discrete Contin. Dyn. Syst.}, 21(2):367--392, 2008.

\bibitem{Rapaport2017}
A.~Rapaport.
\newblock A self-similar measure with dense rotations, singular projections and discrete slices.
\newblock {\em Adv. Math.} 321:529--546, 2017.

\bibitem{Rapaport2015}
A.~Rapaport.
\newblock On self-affine measures with equal {H}ausdorff and {L}yapunov
  dimensions.
\newblock {\em Trans. Amer. Math. Soc.}, 370(7):4759--4783, 2018.




\bibitem{Rohlin1949}
V.~A. Rohlin.
\newblock On the fundamental ideas of measure theory.
\newblock {\em Amer. Math. Soc. Translation}, 1952(71):55, 1952.

\bibitem{Rossi2014}
E.~Rossi.
\newblock Local dimensions of measures on infinitely generated self-affine
  sets.
\newblock {\em J. Math. Anal. Appl.}, 413(2):1030--1039, 2014.

\bibitem{Rudin1987}
W.~Rudin.
\newblock {\em Real and complex analysis}.
\newblock McGraw-Hill Book Co., New York, third edition, 1987.

\bibitem{SagliettiShmerkinSolomyak2018}
S.~Saglietti, P.~Shmerkin, and B.~Solomyak.
\newblock Absolute continuity of non-homogeneous self-similar measures.
\newblock {\em Adv. Math.}, 335:60--110, 2018.

\bibitem{Shu2010}
L.~Shu.
\newblock Dimension theory for invariant measures of endomorphisms.
\newblock {\em Comm. Math. Phys.}, 298(1):65--99, 2010.

\bibitem{Solomyak1998}
B.~Solomyak.
\newblock Measure and dimension for some fractal families.
\newblock {\em Math. Proc. Cambridge Philos. Soc.}, 124(3):531--546, 1998.

\bibitem{Walters1982}
P.~Walters.
\newblock {\em An introduction to ergodic theory}.
\newblock Springer-Verlag, New York-Berlin, 1982.

\bibitem{Young1982}
L.~S. Young.
\newblock Dimension, entropy and {L}yapunov exponents.
\newblock {\em Ergodic Theory Dynam. Systems}, 2(1):109--124, 1982.

\end{thebibliography}
\end{document}